\documentclass[12pt]{amsart}
\usepackage{txfonts}
\usepackage{amscd,amssymb,mathabx,mathtools,multirow}
\usepackage{pdfpages}
\PassOptionsToPackage{usenames,dvipsnames}{xcolor}
\usepackage{verbatim}
\usepackage{mdwlist}
\usepackage{etoolbox}
\AtEndEnvironment{proof}{\setcounter{claim}{0}}

\usepackage{enumerate}
\makeatletter
\let\@@enum@org\@@enum@
\def\@@enum@[#1]{\@@enum@org[\normalfont #1]}
\makeatother
\usepackage[all]{xy}
\usepackage{subfig}
\usepackage{graphicx,calrsfs,url}
\usepackage{tikz}
\usepackage{tikz-qtree}
\usepackage{tkz-euclide}
\usepackage{pgfplots}
\usetkzobj{all}
\usetikzlibrary{intersections}
\usetikzlibrary{calc}
\usetikzlibrary{decorations.markings}
\usetikzlibrary{backgrounds,shapes}

\DeclarePairedDelimiter{\form}{\langle}{\rangle}

\newcommand\ba{\begin{align*}}
\newcommand\ea{\end{align*}}
\newcommand\be{\begin{enumerate}}
\newcommand\ee{\end{enumerate}}
\newcommand\bp{\begin{proof}}
\newcommand\ep{\end{proof}}
\newcommand\bpp{\begin{prop}}
\newcommand\epp{\end{prop}}
\newcommand\bd{\begin{defn}}
\newcommand\ed{\end{defn}}

\numberwithin{equation}{subsection}

\newcommand\stab{\mathrm{Stab}}
\newcommand\eu{\mathrm{eu}}
\newcommand\fform[1]{\langle\!\langle #1\rangle\!\rangle}

\newcommand\nought{_0}
\newcommand\bC{\mathbb{C}}

\newcommand\bN{\mathbb{N}}
\newcommand\bR{\mathbb{R}}
\newcommand\bK{\mathbb{K}}
\newcommand\bQ{\mathbb{Q}}
\newcommand\bZ{\mathbb{Z}}
\newcommand\bH{\mathbb{H}}

\newcommand\CC{\mathcal{C}}
\newcommand\EE{\mathcal{E}}
\newcommand\FF{\mathcal{F}}
\newcommand\FG{\mathfrak{G}}
\newcommand\GG{\mathcal{G}}
\newcommand\DD{\mathcal{D}}
\newcommand\HH{\mathcal{H}}

\newcommand\LL{\mathcal{L}}
\newcommand\PP{\mathcal{P}}
\newcommand\CS{\mathcal{S}}

\newcommand\Inn{\operatorname{Inn}}

\newcommand\Hom{\operatorname{Hom}}

\newcommand\Isom{\operatorname{Isom}}

\newcommand\rot{\operatorname{rot}}
\newcommand\Rot{\operatorname{Rot}}

\newcommand\supp{\operatorname{supp}}
\newcommand\Id{\operatorname{Id}}

\newcommand\Mod{\operatorname{Mod}}

\DeclareMathOperator\Homeo{Homeo}

\newcommand\yt{\widetilde}
\newcommand\sse{\subseteq}
\newcommand\co{\colon}

\DeclareMathOperator\tr{tr}
\DeclareMathOperator\Fix{Fix}
\DeclareMathOperator\Out{Out}
\DeclareMathOperator\Aut{Aut}

\DeclareMathOperator\PSL{PSL}

\DeclareMathOperator\SO{SO}
\DeclareMathOperator\Diff{Diff}

\def\thetitle{{Flexibility of group actions on the circle}}
\def\theauthors{{Sang-hyun Kim and Thomas Koberda and Mahan Mj}}
\usepackage{hyperref}
\hypersetup{
   colorlinks=false,
   plainpages,
   urlcolor=black,
   linkcolor=black,
   pdftitle=   {\thetitle},
   pdfauthor=  {\theauthors}
}

\newtheorem{thm}{Theorem}[section]
\newtheorem{lem}[thm]{Lemma}
\newtheorem{cor}[thm]{Corollary}
\newtheorem{prop}[thm]{Proposition}

\newtheorem{que}[thm]{Question}
\newtheorem*{claim*}{Claim}

\newtheorem{claim}{Claim}

\theoremstyle{remark}
\newtheorem{exmp}[thm]{Example}
\newtheorem{rem}[thm]{Remark}
\newtheorem{notation}[thm]{Notation}

\theoremstyle{definition}
\newtheorem{defn}[thm]{Definition}

\usepackage[nopostdot]{glossaries}

\makeindex

\begin{document}
\title\thetitle
\date{\today}

\author[S. Kim]{Sang-hyun Kim}
\address{Department of Mathematical Sciences, Seoul National University, Seoul, Korea}
\email{s.kim@snu.ac.kr}
\urladdr{http://cayley.kr}

\author[T. Koberda]{Thomas Koberda}
\address{Department of Mathematics, University of Virginia, Charlottesville, VA 22904-4137, USA}
\email{thomas.koberda@gmail.com}
\urladdr{http://faculty.virginia.edu/Koberda}

\author[M. Mj]{Mahan Mj}
\address{School of Mathematics, Tata Institute of Fundamental Research, Homi Bhabha Road, Colaba, Mumbai 400005, India}
\email{mahan@math.tifr.res.in}
\urladdr{http://www.math.tifr.res.in/~mahan/}



\begin{abstract}
In this  partly expository monograph we develop a general framework for producing uncountable families of exotic actions of certain classically studied groups acting on the circle.
We show that if $L$ is a nontrivial limit group then 
the nonlinear representation variety $\mathrm{Hom}(L,\mathrm{Homeo}_+(S^1))$ contains uncountably many semi-conjugacy classes of faithful actions on $S^1$ with pairwise disjoint rotation spectra (except for $0$)
such that each representation lifts to $\bR$.
For the case of most Fuchsian groups $L$,
we prove further that this flexibility phenomenon occurs even locally,
thus complementing a result of K. Mann.
We prove that each non-elementary free or surface group admits an action  on $S^1$ that is never semi-conjugate to any action that factors through a finite--dimensional connected Lie subgroup in $\Homeo_+(S^1)$. 
It is exhibited that the mapping class groups of bounded surfaces have non-semi-conjugate faithful actions  on $S^1$.
In the process of establishing these results, 
we prove 
general combination theorems for indiscrete subgroups of $\mathrm{PSL}_2(\mathbb{R})$ which apply to most Fuchsian groups and to all limit groups.
We also show a Topological Baumslag Lemma, and general combination theorems for representations into Baire topological groups.
The abundance of $\bZ$--valued subadditive defect--one quasimorphisms on these groups would follow as a corollary.
We also give a mostly self-contained reconciliation of the various notions of semi-conjugacy in the extant literature by showing that they are all equivalent.
\end{abstract}

\maketitle

\setcounter{tocdepth}{2}
\tableofcontents

\section{Introduction}

In this monograph, we study finitely generated groups which are classically known to act faithfully on the circle. The purpose of this monograph is to give a systematic construction of uncountable families of actions of these groups which have ``essentially different'' dynamics. The tools described allow us to construct many \emph{exotic} actions of classically studied groups, i.e. actions which are not semi-conjugate to the ``usual" or ``standard" actions of these groups.

This monograph is partially expository and partially original. We develop theory as coherently as possible, with some methods that are well--known to experts, and others which to our knowledge are our own.

In the course of developing tools to build exotic group actions, we prove several \emph{combination theorems}, which construct a framework for building dense subgroups of $\PSL_2(\bR)$,
viewed as a subgroup of the group $\ensuremath{\Homeo_+(S^1)}$ of orientation--preserving homeomorphisms of the circle. Dense subgroups of $\PSL_2(\bR)$ fall outside the purview of classically studied objects, and to study them we employ methods from representation varieties, dynamics in one dimension, and hyperbolic geometry in two and three dimensions. The interested reader may also consult \cite{hinkkanen90,kova99,rebelo}, for some results related to the themes of this monograph.

\subsection{Some basic notions}
\emph{Throughout this introduction, we will let $L$ be a finitely generated group.}
By a \emph{circle action}\index{circle action} of $L$, 
we will mean a homomorphism \[\rho\co L\to \Homeo_+(S^1).
\] 
A circle action $\rho$ of $L$ is \emph{liftable}\index{liftable} if $\rho$ factors as 
\[\xymatrix{L\ar[r]& \Homeo_\bZ(\bR)\ar[r]&  \Homeo_+(S^1)},\] 
where $\Homeo_\bZ(\bR)$
is the set of orientation-preserving homeomorphisms $f$ on the real line satisfying  \[f(x+1)=f(x)+1.\]

Roughly speaking, we say that two circle actions $\rho_0$ and $\rho_1$ of $L$ are \emph{semi-conjugate}\index{semi-conjugacy}, and write 
$\rho_0
\sim_{\mathrm{semi}}
\rho_1$, 
if there is a monotone degree one map $h\colon S^1\to S^1$ such that \[h\circ\rho_0=\rho_1\circ h\]
after ``lifting'' the actions to $\bR$; see Definition~\ref{defn:r-semi}. We say that $\rho_0$ and $\rho_1$ are \emph{equivalent}\index{equivalent circle action} if there exists an $\alpha\in\Aut(L)$ such that $\rho_0\circ\alpha$ is semi-conjugate to $\rho_1$. Semi-conjugacy and equivalence are indeed an equivalence relation. The reader is directed to Section \ref{sec:prelim} and the Appendix for a comprehensive discussion of this definition. 

Let $\bH^2$ denote the hyperbolic plane.
We call a homomorphism \[\rho\co L\to\PSL_2(\bR)\] as a \emph{projective representation}\index{representation!projective} of $L$, which is also regarded as a circle action: 
\[\PSL_2(\bR)\cong\Isom^+(\bH^2)\hookrightarrow\Homeo_+(S^1).\]
A projective representation $\rho\co L\to\PSL_2(\bR)$ is \emph{parabolic-free}\index{parabolic-free projective representation} if $\rho(L)$ does not contain nontrivial parabolic elements. 
A parabolic-free representation is a generalization of a hyperbolic structure without cusps. 

\begin{notation}
Let $L$ be a group.
\be
\item
We denote by $M(L)$ the double coset space 
\[M(L) = \Out(L)\backslash \Hom(L,\Homeo_+(S^1))/\text{semi-conjugacy}.\]
Note that only $\Out(L)$ is relevant here as  inner automorphisms  give rise to the same action up to a conjugacy, and hence, up to semi-conjugacy.
\item
We denote by $X_{\mathrm{proj}}(L)$
the set of faithful parabolic-free projective representations of $L$. 
\ee
\end{notation}
In other words, $M(L)$ is the space of the equivalence classes of circle actions by $L$,
which in turn can be regarded as the ``moduli space'' of circle actions by $L$.
The motivating question for this monograph is the following.

\begin{que}\label{que:main}
For which finitely generated group $L$ is the image of $X_{\mathrm{proj}}(L)$ in $M(L)$ nonempty or uncountable? 
What is the topology of the image of $X_{\mathrm{proj}}(L)$ in $M(L)$?
\end{que}

We are now ready to summarize the principal results of this monograph.

\subsection{Combination Theorems and Indiscrete Subgroups of $\PSL_2(\bR)$}
Before stating the main combination theorem, we need to introduce some terminology. For a group $L$, we let $T_L$ denote the set of torsion elements.

A set of representations $\Lambda\sse\Hom(L,\PSL_2(\bR))$ is \emph{tracially disjoint}\index{tracially disjoint}
if every pair of distinct elements $\rho,\lambda\in \Lambda$ satisfy
 \[\tr^2\circ\rho(L\setminus T_L)\cap\tr^2\circ\lambda(L\setminus T_L)=\varnothing.\]

Recall from our definition that 
a projective representation that factors as \[\xymatrix{L\ar[r]& \PSL_2^\sim(\bR)\ar[r]^p &  \PSL_2(\bR)}\]  
is called liftable; here, $p$ is the universal covering map.  We will say a subgroup $L\le\PSL_2(\bR)$ is liftable if so is the embedding $L\to\PSL_2(\bR)$.

\bd\label{defn:flex-intro}
A finitely generated group $L$ is called \emph{flexible}\index{flexible} if there exists an uncountable, tracially disjoint subset $\Lambda$ of indiscrete faithful parabolic-free projective representations.
If, furthermore, $\Lambda$ can be chosen so that each representation in $\Lambda$ is liftable, then we say $L$ is \emph{liftable-flexible}\index{flexible!liftable-flexible}.
\ed

Excluding the torsion elements is necessary, since if $g\in\PSL_2(\bR)$ has order $n$ in $L$ then there are only $n$ possible values that $\tr^2(g)$ can take.
We will see later that
the space $M(L)$ is uncountable for a flexible group $L$. 
See Corollary~\ref{cor:consequence-psl} for further implications of flexibility.

We let $\Inn(g)$ denote the conjugation map
\[h\mapsto ghg^{-1}.\]
We also use the notation $h^g = g^{-1}hg$.
Recall a subgroup $C\le A$ is \emph{malnormal}\index{malnormal} if $C^g\cap C=\{1\}$ for all $g\in A\setminus C$.
Let $A$ be a group and let $\phi\co C\to C'$ be an isomorphism between subgroups of $A$. We let \[A\ast_{\phi\co C\to C'}=A\ast_\phi=A\ast_s\] denote the HNN extension of $A$ amalgamated along the map $\phi$, where $s$ is used as the stable generator. 

Let $\FF$ and $\FF^\sim$ denote the classes of flexible and liftable-flexible groups, respectively. The following result is the principal combination theorem which we establish in this monograph.

\begin{thm}[Combination Theorem for Flexible Groups]\label{thm:main-comb-intro}
Let $\HH=\FF$ or $\FF^\sim$.
\be
\item
Infinite cyclic groups are in $\HH$.
\item
If $H\le L\in\HH$ and $[L:H]<\infty$, then $H\in\HH$.
\item
If $A$ and $B$ are in $\HH$, then so is $A\ast B$.
\item
Suppose $A,B\le L$ and $C\le A\cap B$ satisfy \[L=\form{A,B}\in\HH.\]
If $C$ is nontrivial malnormal abelian in $A$ and in $B$,
then $A\ast_C B\in \HH$.
\item 
Suppose $A,C\le L$ and $s\in L$ satisfy \[L=\form{A,s}\in\HH\] and that $C$ and $C^s$ are nontrivial malnormal abelian subgroups of $A$. Assume either 

\begin{itemize}
\item $s\in Z(C)$, or
\item $C$ and $C^s$ are not conjugate in $A$.
\end{itemize}
Then  $A\ast_\phi\in\HH$,
where
$\phi=\Inn(s)\co C\to C^s$.
\item
If $A$ is in $\FF$ and $B$ is a finite cyclic group, then $A\ast B\in\FF$.
\ee
\end{thm}

The groups in $\FF^\sim$ are torsion-free,
so part (6) does not hold for $\FF^\sim$.
\begin{rem}\label{rem:main-comb-intro}
If $L$ is a torsion-free flexible group and if $C$ is a maximal abelian subgroup of $L$, then $C$ is malnormal. 
This is because for a hyperbolic maximal abelian subgroup $\mu\le\PSL_2(\bR)$, the set $N(\mu)\setminus\mu$ consists of elements with order two. Here and throughout, we shall say that a subgroup is hyperbolic maximal abelian if it is abelian, consists of hyperbolic elements, and is maximal satisfying these conditions.
\end{rem}

By a \emph{surface group}\index{surface group}, we will mean the fundamental group of a closed orientable hyperbolic surface.
Although looks technical, the condition (\ref{main:hnn}) is precisely required when one considers a decomposition of a surface group by a nonseparating simple closed curve. 
A relatively straightforward application of Theorem \ref{thm:main-comb-intro} above is for a \emph{double}\index{double (of a group)} of a flexible group. 
For a group $A$ and a subgroup $C$, we denote by $A\ast_C\bar A$ the amalgamation of two copies of $A$ along $\Id_C$.
We write $A\ast_C = A\ast_C (C\times\bZ)$. Thus, $A\ast_C$ is an HNN extension of $A$ along $C$ where the two inclusions of $C$ into $A$ are identical. By putting $A=B$ or $s=1$ in the Combination Theorem, we obtain:

\begin{cor}\label{cor:double-intro}
Let $A$ be a group and let $C$ be a malnormal abelian subgroup of $A$.
If $A$ is flexible or liftable-flexible, then so are $A\ast_C$ and $A\ast_C \bar A$.
\end{cor}

With some more work, we obtain the following very general result about circle actions of \emph{limit groups}\index{limit group}. Recall that a limit group is a finitely generated group which is fully residually free.

\begin{cor}
Every nontrivial limit group is liftable-flexible.
\end{cor}

A finite type hyperbolic 2--orbifold $S$, and its fundamental group, are called \emph{splittable}\index{$2$--orbifold!splittable} if some simple closed curve splits $S$ into hyperbolic pieces (Definition~\ref{defn:split}).
For instance, all surface groups and Fuchsian groups satisfying $\chi\le-1$ are splittable.

Flexibility is important to us because it allows us to construct many dynamically distinct actions of groups on the circle. 
Within the context of groups which ``classically" act on the circle through discrete projective representations, we have the following.

\begin{thm}\label{thm:fuchs-flex-intro}
\be
\item
Finite type splittable Fuchsian groups are flexible.
\item
Finite type torsion-free Fuchsian groups are liftable-flexible.
\ee
\end{thm}

In the course of establishing Theorem \ref{thm:main-comb-intro}, we prove several other results which are of independent interest. Probably the most important of these are several general versions of Baumslag's Lemma.

Classically, Baumslag's Lemma~\cite{Baumslag1962} is a criterion for certifying that elements of free groups do not reduce to the identity, and is fundamental in the study of limit groups. We prove several generalizations of Baumslag's Lemma, which we will not state here for the sake of brevity. These include a Topological Baumslag Lemma for group actions on an arbitrary Hausdorff topological space (Lemma \ref{lem:baumslag-gen}), a Projective Baumslag Lemma which applies to subgroups of $\PSL_2(\bC)$ (Lemma \ref{lem:baumslag-psl}), and a Discrete Baumslag Lemma which applies to acylindrically hyperbolic groups and mapping class groups of surfaces (Theorem~\ref{thm:baumslag-acyl} and Corollary \ref{cor:mcg}). Baumslag's Lemmas coupled with a Baire category argument for representation varieties allows us to prove a number of combination theorems for faithful (but possibly indiscrete) representations, including Theorem~\ref{thm:main-comb-intro}.

\subsection{Uncountable Families of Exotic Group Actions on the Circle}\label{ss:exotic}

The combination theorems discussed above are the technical tools we use to construct large families of actions of ``classical" groups on the circle with essentially different dynamical behaviors, which is to say that the actions are not semi--conjugate to each other.

We distinguish between various semi-conjugacy classes of actions by estimating their \emph{rotation spectra}. Recall that if \[f\in\Homeo_+(S^1)\] then $f$ has a well-defined \emph{rotation number}\index{rotation number} $\rot(f)$, which is given by lifting $f$ to \[\yt{f}\in\Homeo_+(\bR)\] and computing the limit \[\lim_{n\to\infty} \frac{\yt{f}^n(0)}{n}\pmod\bZ.\] The \emph{rotation spectrum}\index{rotation spectrum} of a circle action of $L$ is given by 
\[\rot\circ\rho(L)=\{\rot\circ\rho(g)\co g\in L\}\sse\bR/\bZ.\] One can also define the \emph{marked rotation spectrum}\index{marked rotation spectrum} of $\rho$ by \[\rot\circ\rho=\{(g,\rot\circ\rho(g))\co g\in L\}\in (\bR/\bZ)^L.\]

The rotation number is a semi-conjugacy invariant (Theorem~\ref{thm:equiv}).
Hence the rotation spectrum of an action is an equivalence invariant of an action, and the marked rotation spectrum is a semi-conjugacy invariant of an action. The notion of rotation spectrum has attracted much attention in the literature; see~\cite{CalegariForcing,GhysTop87} for instance, as well as the discussion in Subsection \ref{subsubsec:calegari} below.

If $\rho$ is a circle action of $L$, then its Euler cocycle $\rho^*\eu$ defines 
a class $\rho^*\eu_b\in H^2_b(L;\bZ)$; see Theorem~\ref{t:ghys2}.
We treat the following consequence of flexibility in Section~\ref{sec:flex}.

\begin{cor}\label{cor:consequence-psl-intro}
If $L$ is a flexible group, then there exists an uncountable subset $\Lambda$
of $X_{\mathrm{proj}}(L)$ such that the following hold.
\be[(i)]
\item The natural map $\Lambda\to M(L)$ defined by $\rho\mapsto[\rho]$ is an injection.
\item
The image of $\Lambda$ is $\bZ$--linearly independent
in $H^2_b(L;\bZ)$.
\item
If, furthermore, $L$ is liftable-flexible, then $\Lambda$ maps into the kernel of 
\[H^2_b(L,\bZ)\to H^2(L,\bZ).\]
\ee
\end{cor}

In~\cite{Mann2015IM}, K. Mann shows that if $L$ is a surface group and if $Z$ is a component of the representation space \[X=\Hom(L,\Homeo_+(S^1))\] which contains a discrete faithful representation from $L$ into a finite covering group of $\PSL_2(\bR)$, then $Z$ consists of one semi-conjugacy class. 
We exhibit an opposite, ``local flexibility'' phenomenon for indiscrete representations of Fuchsian groups (possibly with torsion). See Section~\ref{sec:almost faithful} for the definition of an \emph{almost faithful}\index{almost faithful} path.

\begin{cor}[See Theorem~\ref{thm:fuchs-flex}]\label{cor:fuchs-intro}
Let $L$ be a Fuchsian group and let $\rho_0\in X_{\mathrm{proj}}(L)$ be indiscrete.
Then for each countable set $W\sse\bR$,
 there exists an almost faithful path $\{\rho_t\}_{t\in[0,1]}$ in $\Hom(L,\PSL_2(\bR))$ such that  \[\tr^2\circ\rho_1(L\setminus T_L)\cap W=\varnothing.\]
\end{cor}

In other words, every open neighborhood of $\rho_0$ contains uncountably many, pairwise non-semi-conjugate representations from $X_{\mathrm{proj}}(L)$.
In particular, the point $[\rho_0]\in M(L)$ is an accumulation point of the image of $X_{\mathrm{proj}}(L)$ and this gives a partial answer to the second half of Question~\ref{que:main}.
Note that one cannot further require $\rho_s\not\sim_{\mathrm{semi}}\rho_t$ for all $0\le s<t\le 1$; see Proposition~\ref{prop:faithful-arc}.

It is an open question whether or not the space $X$ above has infinitely many connected components~\cite{Mann2015IM}. 
The representations in \[Y=\Hom(L,\PSL_2(\bR))\] with  Euler number $0$ form a connected component~\cite{Goldman1988IM} of $Y$, and we show that this component contains uncountably many equivalence classes, as follows easily from Corollary~\ref{cor:consequence-psl-intro}. 
More generally, we have the following.

\begin{cor}\label{cor:uncountable class}
Let $L$ be a finite type splittable Fuchsian group. 
Then some path--component of \[\Hom(L,\Homeo_+(S^1))\] contains uncountably many distinct equivalence classes of faithful actions.
\end{cor}

Corollary \ref{cor:uncountable class} is complemented by work of M. Wolff. See Subsection \ref{subsubsec:teichmueller} below.

We note another consequence of Corollary \ref{cor:consequence-psl-intro}.
Suppose $f\co L\to \bZ$ is a map.
We define the \emph{coboundary}\index{coboundary} of $f$ as
\[\partial f(g,h)=f(g)+f(h)-f(gh)\]
and the \emph{defect}\index{defect} of $f$ as
\[ D(f)=\|\partial f\|_\infty=\sup\{ |\partial f(g,h)|\co g,h\in L\}.\]
We say $f$ is a \emph{quasimorphism}\index{quasimorphism} if
$D(f)<\infty$. 
A \emph{minimal}\index{minimal!quasimorphism} quasimorphism is an integer--valued subadditive defect--one quasimorphism. See Section~\ref{sec:flex} for details.

\begin{cor}\label{cor:bddcoh}
If $L$ is liftable-flexible,  then there exists a set $\Lambda$ of minimal quasimorphisms
such that \[\{[\rho]\in \mathrm{HQM}(L;\bZ)\co \rho\in \Lambda\}\]
is uncountable and linearly independent. 
\end{cor}

Here, $\mathrm{HQM}(L;\bZ)$
denote the set of quasimorphisms from $L$ to $\bZ$, quotiented by homomorphisms and bounded maps.
See Section \ref{sec:prelim} for a more detailed discussion of quasimorphisms, circle actions, and bounded cohomology. We note here a corollary for limit groups.

\begin{cor}
Every limit group admits an uncountable, $\bZ$--linearly independent set of minimal quasimorphisms.
\end{cor}

\subsection{An Axiomatic Approach to Combination Theorems}
In Section \ref{sec:axiom} below, we develop an axiomatic framework for combination theorems for group representations where the target group is a general Baire topological group $\mathfrak{G}$. 
A far more concrete setup is discussed in Section \ref{sec:flex},
and the axioms in Section \ref{sec:axiom} might seem technical and (overly) abstract to the reader at the first reading. For this reason, we have chosen the present expository mode of providing many examples in Section \ref{sec:flex} and motivating the abstract axiomatic nature of Section \ref{sec:axiom} that follows.

The notion of a trace is replaced by a more general notion of a \emph{tracial structure}\index{tracial structure}.
The most important concepts encoded in a tracial structure (Definition \ref{defn:tracial}) are a generalized trace 
(a continuous function  on conjugacy classes) and a notion of genericity, in the form of Bausmlag Lemma. 
The generalized trace is simply a continuous map \[\varphi\colon\mathfrak{G}\to\mathcal{S},\] where $\mathcal{S}$ is a topological space, and where $\varphi$ is constant on conjugacy classes. The notion of generalized parabolic elements in this situation is defined accordingly: \[\mathcal{P}=\varphi^{-1}\circ \varphi(1),\] i.e. the generalized parabolic elements are precisely those whose traces agree with that of the identity with respect to the tracial structure given by $\varphi$. 
Using this structure, we will prove an axiomatic combination theorem (Theorem~\ref{thm:comb-gen}), which generalizes Theorem~\ref{thm:main-comb-intro}.

\subsection{Flexibility and Rigidity}\label{rotspec}

Among group actions on the circle with a given rotation spectrum, there appears to be a mixture of rigidity and flexibility, as we illustrate with the following results.

On the one hand, projective representations impose a high degree of rigidity. The following fact is relatively well--known (see for instance Katok's book~\cite{KatokBook}, Theorem 2.5.4):
\begin{prop}[cf. Theorem \ref{thm:rigid}]\label{prop:analytic rigid}
Let $L$ be a finitely generated group and let $\phi$ be a faithful projective action of $L$ with $\rot\circ \phi(L)=\{0\}$. Then $\phi(L)$ is conjugate to a discrete subgroup of $\PSL_2(\bR)$. In particular, either:
\begin{enumerate}
\item
The group $L$ is the fundamental group of a closed orientable hyperbolic surface $S$, and $\phi$ corresponds to a complete hyperbolic structure on $S$; in particular, up to an automorphism of $L$, there are only two conjugacy classes of such actions;
\item
The group $L$ is free, and there are finitely many equivalence classes of such projective $L$--actions on $S^1$, two for each homeomorphism type of a surface $S$ with $\pi_1(S)\cong L$.
\end{enumerate}
\end{prop}

On the other hand, even marked rotation spectra do not form a complete equivalence invariant of surface group actions:

\begin{thm}[cf.~Theorem~\ref{thm:free nonlinear}]\label{thm:closed nonlinear-intro}
There exists a faithful action $\phi$ of 
a finitely generated nonabelian free group $L$ on $S^1$ with (marked) rotation spectrum $\{0\}$ and such that $\phi$ is not semi-conjugate to a subgroup of $\PSL_2^{(k)}(\bR)$, for any $k\geq 1$.
\end{thm}

Here, $\PSL_2^{(k)}(\bR)$ denotes the connected $k$--fold cover of $\PSL_2(\bR)$. In particular, the action of the free group in Theorem~\ref{thm:closed nonlinear-intro} is not semi-conjugate to a projective action.

The proof of Theorem \ref{thm:closed nonlinear-intro} is remarkable in that it uses several nontrivial facts about hyperbolic $3$--manifolds which fiber over the circle. 
We remark furthermore that Theorem \ref{thm:closed nonlinear-intro} considers flexibility of actions within a class of actions with a fixed rotation spectrum. Thus, this result is one of the several in this monograph which are completely different from (as opposed to complementary to) the other results in literature which study flexibility and rigidity of actions via rotation numbers (cf.~\cite{CalegariForcing,CW2011,GhysTop87,Mann-hb}, for instance).

K. Mann has pointed out a statement analogous to Theorem \ref{thm:closed nonlinear-intro} for surface groups, which appeared in her thesis~\cite{MannThesis}. Her argument shares certain features with the proof of Theorem~\ref{thm:closed nonlinear-intro} given in this monograph, though the crux of her argument (which relies on Thurston norm and maximal Euler classes) and motivations for considering the problem in the first place are different from those of the authors.
Result similar to Theorem \ref{thm:closed nonlinear-intro} have also been produced by Barbot--Fenley~\cite{FenleyBarbot}, using pseudo-Anosov flows. See Subsection \ref{subsubsec:exotic} below.

Hyperbolic $3$--manifold groups contain a profusion of closed surface groups, both quasi--Fuchsian and geometrically infinite, and fibered hyperbolic $3$--manifold groups admit natural circle actions, i.e. the universal circle action (see Section~\ref{sec:mcg} and Subsection~\ref{ss:3mfd}. We have the following semi--conjugacy characterizations of such surface group actions, subject to some mild natural hypotheses:

\begin{thm}[cf. Theorem~\ref{ct=min}]
Let $G$ be a quasi--Fuchsian subgroup of a closed fibered hyperbolic $3$--manifold group corresponding to a quasi--Fuchsian surface which is transverse to the flow induced by a fixed fibration of the manifold. Then the universal circle action of $G$ is semi--conjugate to a Fuchsian group action.
\end{thm}

\begin{thm}[cf. Theorem~\ref{thm:geominf}]
Let $M$ be a fibered hyperbolic $3$--manifold with a fixed fibration with fiber $S$, and let $G$ be a fiber subgroup of a fibration of $M$ which lies in the same fibered face of the Thurston norm ball as $S$, up to multiplication by $-1$. Then the universal circle action of $G$ is conjugate to a Fuchsian group action.
\end{thm}

As mentioned above, K. Mann proved a complement to Theorem~\ref{thm:geominf} in her thesis~\cite{MannThesis}, proving that if $G$ lies in a different fibered face of the Thurston norm ball then $G$ is not semi--conjugate to a Fuchsian group action.

The following result exhibits exotic actions of nonabelian free groups on the circle, among other exotic group actions:

\begin{thm}[cf. Theorem~\ref{thm:liegp}]\label{thm:cinfty-intro}
Let $F$ be a nonabelian free group.
Then 
there exists a faithful $C^{\infty}$ action
of $F$ on $S^1$ 
which is not semi-conjugate to 
the action of 
any transitive, finite--dimensional, connected Lie subgroup 
of  $\Homeo_+(S^1)$.
\end{thm}

\subsection{Mapping Class Groups}

We conclude our discussion with a study of distinct semi-conjugacy classes of mapping class group actions on the circle (see Section~\ref{sec:mcg}). Recall that if $S$ is an orientable surface, its mapping class group $\Mod(S)$ is the group orientation--preserving homeomorphisms of $S$ up to isotopy. We will always assume that $S$ has negative (though finite) Euler characteristic. A mapping class group is called \emph{bounded}\index{bounded mapping class group} if $\partial S\neq\varnothing$.

\begin{prop}[cf. Proposition~\ref{prop:mcg}]
A bounded mapping class group has at least two inequivalent actions on the circle.
\end{prop}

\subsection{Notes and References}

\subsubsection{Circle Actions and Quasimorphisms}
Among the original motivations for this monograph was a question posed to the authors by M. Bestvina and K. Fujiwara, namely whether one can explicitly produce distinct, subadditive, 
integer--valued, defect--one quasimorphisms of
a free group or a surface group. 
This monograph answers their question by producing uncountable families of examples which come from quasimorphisms of the corresponding groups.

It was pointed out to the authors by D. Calegari that defect--one (possibly not subadditive) quasimorphisms from a given quasimorphism is relatively straightforward, so that producing uncountably many linearly independent ones is not such a huge generalization. In this monograph however, the quasimorphisms we produce are subadditive and defect--one, these two properties being not so trivial to simultaneously guarantee, since such a quasimorphism (as a second bounded cohomology class) corresponds to a semi-conjugacy class of a circle action that lifts to the real line.

\subsubsection{Generalizations to Other Semi-simple Algebraic Groups}
Many of the results about $\PSL_2(\bR)$ which we discuss in this monograph generalize suitably to other semi-simple algebraic groups. When discussing algebraic groups, we concentrate on $\PSL_2(\bR)$ since we are primarily interested in the dynamical picture. Indeed, lattices in non-split higher rank semi-simple linear algebraic groups admits no interesting actions on the circle~\cite{BM1999,Witte1994,Ghys1999}. Thus, while the continuous version of Baumslag's Lemma (see Lemma \ref{lem:baumslag-psl} below) admits a generalization to other algebraic groups, for instance, we avoid making such generalizations for the sake of brevity and clarity, especially since the dynamical consequences would be rather limited for us. We reiterate, however, that this monograph develops a representation--theoretic perspective for studying dynamics in low dimensions which fits into a larger context of indiscrete subgroups of semi-simple algebraic groups.

\subsubsection{Towards a Teichm\"uller Theory for Indiscrete Representations}\label{subsubsec:teichmueller}
Most classical constructions of free group actions on various spaces, and more generally of free product actions, rely on some version of the ping--pong lemma~\cite{MR1786869,tits-jalg,MR2987617}. The novelty of the approach given here is in the construction of dense subgroups of $\PSL_2(\bR)$ with various exotic properties, which are situations to which ping--pong is unadaptable. We note that, even though we are concentrating on dense subgroups of $\PSL_2(\bR)$, our constructions are generally informed by and often rely on classical discrete subgroup constructions.

The rotation spectrum of a representation is an equivalence invariant of the representation, which plays a role similar to the length spectrum for a discrete representation of a group $L$ into $\PSL_2(\bR)$. Classical Teichm\"uller Theory gives conditions under which length spectra determine a surface group representation up to conjugacy (see~\cite{Otal1990}, for instance). In the situation of indiscrete actions, the length spectrum for hyperbolic elements in the image will be dense inside of $\bR^+$.
Furthermore, it seems likely that rotation spectrum alone will not generally determine the semi-conjugacy class of a faithful action:

\begin{que}\label{que:teichmueller}
Let $L$ be a surface group and let $\rho_0,\rho_1$ be projective representations of $L$ on $S^1$. If $\rot\circ \rho_0(L)=\rot\circ \rho_1(L)$, under what conditions is $\rho_0$ is equivalent to $\rho_1$?
\end{que}

We make two further remarks about rigidity of actions. First, there is a sense in which the marked rotation spectrum (with some extra data) does in fact determine a circle action up to semi-conjugacy, which is given by Matsumoto's Theorem~\cite{Matsumoto1986}. We will give an exact statement of this result in Theorem \ref{thm:equiv}.

Secondly, the work of M. Wolff gives an answer to Question \ref{que:teichmueller} in the case of marked rotation spectra:

\begin{thm}[M. Wolff~\cite{WolffPreprint}, see also~\cite{WolffGT}]
Let $\rho_0$ and $\rho_1$ be projective representations of a finitely generated group $L$.
If $\rho_0$ is non-elementary and indiscrete, and if for all $g\in L$ we have \[\rot\circ\rho_0(g)=\rot\circ\rho_1(g),\] then $\rho_0$ and $\rho_1$ are conjugate in $\PSL_2(\bR)$.
\end{thm}

In the case of unmarked rotation spectra, it seems that relatively few properties of a projective group action are determined by the unmarked rotation spectrum~\cite{MarcheWolff}.

\subsubsection{Dense Limit Subgroups of Algebraic Groups}
The  work in  Section \ref{sec:fuchs}
is closely related to, and partially inspired by, the work of Breuillard--Gelander--Souto--Storm~\cite{BGSS2006} and of Barlev--Gelander~\cite{BG2010JAM} in which those authors studied the relationship between dense free subgroups and dense limit subgroups of semi-simple algebraic groups. In the case of $\PSL_2(\bR)$ actions, we show not only the existence of dense non-free limit subgroups,
(thus recovering the main results of \cite{BGSS2006} and \cite{BG2010JAM} in this case) but in fact an enormous abundance of such subgroups, as witnessed by our constructions of uncountable, tracially disjoint (hence inequivalent) families of representations  in $\PSL_2(\bR)$.

Perhaps the most significant conceptual difference between the present work and that of~\cite{BGSS2006} and of~\cite{BG2010JAM} is the form in which Baumslag lemmas are stated, proven and used. In ~\cite{BGSS2006} and~\cite{BG2010JAM}, the fundamental combinatorial group theory tool is Baumslag's lemma for free groups. In the present work, we use the Projective Baumslag Lemma adapted to the ambient Lie group, and a Topological and Discrete Baumslag Lemma adapted to more general setups in which north--south dynamics is present. It is thus that we are able to apply our machinery to a larger array of groups in order to deduce the dynamical consequences outlined above.

\subsubsection{Relationship to the Work of Calegari and Calegari--Walker}\label{subsubsec:calegari}
Various rigidity and flexibility phenomena concerning rotation numbers for groups acting on the circle have been studied by many authors as we have indicated above, and will indicate further below. As for the ideas of this monograph which concern rotation numbers, the most closely related work is probably the work of Calegari in~\cite{CalegariForcing} and Calegari--Walker in~\cite{CW2011}.

In~\cite{CW2011}, Calegari and Walker study actions of the rank two free group $F_2$ on the circle, and the various rigidity and rationality phenomena that can be described therein. For instance, if two free generators act with rational rotation numbers $r$ and $s$ then 
the supremum $R(w,r,s)$ of the possible rotation numbers achieved by each fixed positive word $w\in F_2$ in those generators is again rational. Moreover, 
the denominator of $R(w,r,s)$ is well--controlled and 
computable. 
The value $R(w,r,s)$ is locally constant from the right for a positive $w$ and $r,s\in \bQ$.

In~\cite{CalegariForcing}, Calegari considers (among other things) the following problem: let $\theta\in\bR/\bZ$. Is there a finitely presented group $L$ acting on $S^1$ 
and an element $\alpha\in L$ such that $\cup_\rho\rot\rho(\alpha)=\{\pm\theta,0\}$ 
where $\rho$ varies over all possible actions of $L$ on the circle? If the answer is yes then $\theta$ is called a \emph{forceable number}. Calegari shows that the set of forceable numbers is very large: lifting the set of forceable numbers to $\bR$, one obtains an infinite dimensional rational vector subspace of $\bR$. Moreover, the set of forceable numbers contains infinitely many algebraically independent transcendental numbers.

Perhaps the most significant difference between this work and that of Calegari~\cite{CalegariForcing} is the fact that in the present work we study possible rotation numbers within the context of actions of a single group, i.e. we do not control all actions of a given finitely presented group on the circle. Our purpose is the opposite, in fact: we use large diversity in rotation numbers to deduce the existence of many actions.

Moreover, as noted above, we expend a considerable effort to produce inequivalent actions of groups on the circle which are not distinguished by rotation numbers, which is to say flexibility of actions with a fixed rotation spectrum.

\subsubsection{Dense Sets of Faithful Projective Surface Group Actions}\label{subsec:deblois-kent}
In~\cite{DK2006Duke}, Deblois and Kent show that for a surface group $L$, the set of faithful representations inside the representation variety \[\Hom(L,\PSL_2(\mathbb{K}))\] is dense in the classical analytic topology, where $\mathbb{K}\in \{\bR,\bC\}$. In fact, they show that the set of faithful representations is a dense $G_{\delta}$ in the classical topology.
In their proof, they even use a fact which may be regarded as a transcendental/number theoretic Baumslag Lemma in order to establish the properness of vanishing loci in the representation variety, though the authors do not call it by that name. The work of Deblois--Kent in the case of $\mathbb{K}=\bR$ can be used to construct many non-conjugate actions on the circle (see~\cite{MannPJM}), though questions of inequivalences of these actions are not addressed.

The purpose of Deblois and Kent was to prove a conjecture of W. Goldman, namely that the set of faithful representations is dense in the representation variety. As such, dynamical considerations did not enter into their discussion. On the one hand, the present work builds a more general framework for studying exotic group actions on the circle. As a consequence of developing such a framework, we deduce the existence of an abundance of exotic circle actions, not just for surface groups but also for free groups, limit groups, and many other groups beyond the scope of the methods used in~\cite{DK2006Duke}. On the other hand, this monograph does not recover the Deblois--Kent result, since their considerations apply to all components of the representation variety, whereas the topology of the representation variety is of secondary importance to us.

Deblois--Kent's methods are geometric in the sense that they rely essentially on the geometrization of surfaces. Our approach is a combination of a representation--theoretic and an algebro--geometric approach which is informed by ideas from geometry. Thus, we are able to deduce dynamical consequences for groups beyond those appearing in two--dimensional hyperbolic geometry. The particular features of our work which do not fall under the purview of~\cite{DK2006Duke} include but are not limited to combination theorems for indiscrete subgroups of $\PSL_2(\bR)$, general Baumslag Lemmas, and applicability to limit groups. Moreover, we construct a path in the representation space of a Fuchsian group such that almost every point on the path is a faithful representation (cf. Theorem~\ref{thm:fuchs-flex}), which does not follow formally from Deblois--Kent. Finally, while the present monograph focuses on projective representations of groups, a large part of the monograph concerns general smooth and even topological actions of groups on the circle.
The following question attempts a natural generalization Deblois--Kent result in our setting:

\begin{que}\label{que:very general}
Let $L$ be a nonabelian limit group. Is the set $X_{\mathrm{proj}}(L)$ very general in $\Hom(L,\PSL_2(\bR))$?
\end{que}

In Question \ref{que:very general}, the terminology \emph{very general}\index{very general} is a certain Baire category condition. See Section \ref{sec:baire} for a precise discussion.

The reader may also compare the work of Bou-Rabee--Larsen~\cite{bourabee}, which contains an extended discussion of
flexibility for surface groups (and in particular what they name the
\emph{Borel property}).

\subsubsection{Projective Actions Versus Analytic Actions}\label{subsec:analytic}
Questions similar to those addressed in this monograph are studied in~\cite{Triestinoetal} in the context of analytic group actions on the circle. A group action on the circle is called \emph{analytic} if every element acts by a homeomorphism which is locally given by a convergent power series in the local parameter. The group $\PSL_2(\bR)$ acts on the circle by analytic diffeomorphisms, though not every analytic action of a group corresponds to a subgroup of $\PSL_2(\bR)$. In~\cite{Triestinoetal}, the authors of that paper define a combinatorial invariant called a \emph{topological skeleton}, which plays the role of a Markov partition, and which allows them to classify virtually free group actions on the circle which are analytic.

Among the consequences of the main results of~\cite{Triestinoetal} are certain rigidity results for semi-conjugacy of actions with a exceptional minimal set. They also obtain results which are related to Theorem \ref{thm:closed nonlinear-intro}. Namely, they produce certain \emph{locally discrete} finitely generated free groups of analytic diffeomorphisms of $S^1$ which have exotic dynamics. For example, they produce a two--generated subgroup of analytic diffeomorphisms which acts minimally, which is free of rank two, where one generator acts like a hyperbolic element of $\PSL_2(\bR)$, and where one generator has two parabolic fixed points. In particular, such an action cannot be semi-conjugate to a projective action. The authors show that the rotation spectrum of the actions they consider are finite.

Thus, the authors of~\cite{Triestinoetal} produce examples of a similar ilk to those furnished in Theorem \ref{thm:closed nonlinear-intro}, in the case of free groups. Their methods are very different from those employed in this monograph, and they do not extend to analytic surface group actions.

\subsubsection{Non--Fuchsian Exotic Actions}\label{subsubsec:exotic}
Groups which act on the circle in ways that resemble Fuchsian group actions but which are not projective actions (as in Theorem \ref{thm:closed nonlinear-intro}) have been studied by several authors, as has been pointed out by M. Triestino. The ideas find their origin in Thurston~\cite{ThurstonGodbillon}, which were then clarified and further developed by Brooks in his appendix to Bott's article~\cite{BrooksAppendix}, and then further studied by Tsuboi~\cite{Tsuboi1984}.

The work in~\cite{Triestinoetal} suggests that locally discrete analytic virtually free group actions should be semi-conjugate into a certain overgroup studied by Tsuboi in~\cite{Tsuboi1984}. As discussed above, other examples of non-Fuchsian exotic actions appear in~\cite{MannThesis} and~\cite{FenleyBarbot}.

\subsubsection{Groups Without Exotic Actions}
There are many groups which naturally act on the circle with various degrees of regularity, but which do not admit any (or at least very few) ``exotic" actions. Precisely, there are finitely generated groups \[G\leq\Diff^k(S^1)\leq\Homeo_+(S^1)\] such that if $k\gg 0$ (generally $k\geq 2$ suffices), 
then a faithful $C^k$ action of $G$ on $\Homeo_+(S^1)$ has periodic points and is therefore semi-conjugate to an action of some finite cyclic group. Natural classes of groups with such properties lie in the class of right-angled Artin groups. For an explicit example, one can consider $G=F_2\times F_2$. Standard applications of Denjoy--Kopell type arguments show that any faithful $C^2$ action of $G$ has a periodic point (cf.~\cite{KKFreeProd2017}). 
We remark that in general, right-angled Artin groups admit no faithful actions on the circle of $C^2$ or higher regularity (see~\cite{BKK16,KKFreeProd2017}, also~\cite{FF2003,Jorquera}).

It seems that the degree of regularity one requires of group actions can have a significant effect on the existence or nonexistence of exotic actions. For instance, note that every right-angled Artin group occurs as a subgroup of the mapping class group of some surface with boundary (see~\cite{Koberda2012} and the references therein), and these mapping class groups do admit exotic actions (see Section \ref{sec:mcg} below), but no faithful mapping class group action is conjugate to a $C^2$ action by~\cite{BKK16}.

\subsubsection{Mapping Class Groups}

The fact that the mapping class group of a surface with a marked point embeds into $\Homeo_+(S^1)$ is an old result of Nielsen (see~\cite{Nielsen1927,Nielsen1929,Nielsen1932}, also~\cite{CassonBleiler,HT1985}). Such actions for mapping class groups with boundary on the circle and on the real line were studied by Handel--Thurston~\cite{HT1985}. The existence of exotic actions of mapping class groups (i.e. ones not semi--conjugate to Nielsen's original action) was posed as a question by Farb~\cite{Farb2006}, and has been also studied by Bowditch--Sakuma~\cite{BowditchSakuma}.

\subsection{Outline of the monograph}
We have striven to make the present monograph as self--contained as possible. For ideas and methods which are well--known or appear in literature, we strive to give complete references. In Section \ref{sec:prelim} we gather various cohomological facts about circle actions, where an object of central importance is the bounded Euler class of an action. We single out Theorem \ref{thm:equiv} for the attention of the reader, since it is difficult to find a definitive statement of the equivalence of all the various definitions of semi-conjugacy in the literature.

Section \ref{sec:baumslag} contains the various incarnations of Baumslag's Lemma, 
which is one of the main technical tools of the monograph.

Section~\ref{sec:fuchs} produces uncountable families of inequivalent Fuchsian and nontrivial limit group actions on the circle, all arising from indiscrete subgroups of $\PSL_2(\bR)$. This discussion generalizes in the form of a combination theorem in Section \ref{sec:flex}. Implications of flexibility and liftability for limit groups and quasimorphisms are discussed here.

Section \ref{sec:axiom} develops a more axiomatic setup for the discussion in Section \ref{sec:fuchs} and in Section \ref{sec:flex}, and establishes indiscrete combination theorems again for all limit groups and in the setup where the target group is a general Baire topological group.

In Section~\ref{sec:mcg}, we construct actions of mapping class groups that are not conjugate to Nielsen's classical actions.
Section~\ref{sec:rot spec} deals with the rigidity of faithful projective representations with zero rotation spectrum, and also provides  examples of smooth actions of free groups on the circle that are not semi-conjugate to actions factoring through finite dimensional Lie groups.

The logical dependence of the sections is as follows. Section \ref{sec:prelim} is basic but can be skipped by readers familiar with circle actions and bounded cohomology. Section \ref{sec:append} complements  Section \ref{sec:prelim}
by giving details and proofs. Section \ref{sec:flex} relies on Sections~\ref{sec:baumslag} and~\ref{sec:fuchs}. Section \ref{sec:axiom}  is logically independent from the other sections, but is conceptually informed by Sections \ref{sec:fuchs}  and \ref{sec:flex}.

\section{Preliminaries}\label{sec:prelim}
\subsection{Actions on the Circle}\label{sec:prelimcirc}

Ghys employed cohomological methods to ``study the dynamics''~\cite{Ghys1987} of discrete group actions 
on the circle.
More precisely, he defined a relation between group actions he called \emph{semi-conjugacy}\index{semi-conjugacy}
and exhibited a correspondence between cohomology classes and semi-conjugacy classes; see Theorem~\ref{t:ghys2}.

Recently, several other closely related notions of semi-conjugacy were suggested in the literature~\cite{Calegari2007,BFH2014,Mann-hb}.
These modifications were aimed at converting semi-conjugacy into an equivalence relation
and further, to have a powerful characterization of circle actions via cohomology classes. For a complete reference on bounded cohomology of discrete groups, the reader is directed to Frigerio's recent book~\cite{FrigerioBook}, which includes a discussion of second bounded cohomology and circle actions.

A reconciliation of these notions of semi-conjugacy is summarized in Theorem~\ref{thm:equiv} below. 
We refer the reader to the Appendix and the cited references therein for proofs of equivalence.
A different and more detailed exposition of many (but not all) of these equivalent notions has been given in \cite{BFH2014}. The ideas go back to Ghys' original article~\cite{Ghys1987}, and to~\cite{Calegari2007,BFH2014,Mann-hb}. The reader is also directed to~\cite{CalegariCassonfest}. We remark that throughout this section, we make no claims as to originality.

Throughout this section, \emph{we let $L$ be a countable group}.
For $r\in\bR$, we let $T(r)$ denote the translation by $r$ on $\bR$ or on $S^1$.
For brevity, we denote \[T=T(1)\co\bR\to\bR.\]

\medskip
\noindent {\bf Rotation number and Euler class}\\
Let us recall the following well-known terminology. Recall that $\Homeo_\bZ(\bR)$
is the set of orientation-preserving homeomorphisms $f$ on the real line satisfying  \[f(x+1)=f(x)+1.\]
For $f\in\Homeo_\bZ(\bR)$, we define the \emph{translation number}\index{translation number}
	\[
\rot^\sim(f)=\lim_{n\to\infty}\frac{f^n(x)-x}{n},\]
	which is independent of the choice of $x\in\bR$.
	The Poincar\'e's \emph{rotation number}\index{rotation number} of
	$f\in\Homeo_+(S^1)$ is defined as
	\[\rot(f)=\rot^\sim(\tilde f) \pmod \bZ\]
	where $\tilde f\in\Homeo_\bZ(\bR)$ is an arbitrary lift of $f$. 
Then $\rot\co\Homeo_+(S^1)\to S^1$ is a continuous \emph{homogeneous class function}\index{homogeneous class function}.
That is, the map $\rot$ is continuous with respect to the usual (uniform) topology on $\Homeo_+(S^1)$,
and \[\rot(f^n)=n\rot(f),\quad \rot(gfg^{-1})=\rot(f)\] for $f,g\in\Homeo_+(S^1)$ and $n\in\bZ$. The readers may refer to \cite{Navas2011} for more details.

For each $x\in\bR$, 
there exists a unique section 
\[s^x
\co \Homeo_+(S^1)\to\Homeo_\bZ(\bR)\]
satisfying the following
for all $f\in\Homeo_+(S^1)$: \[s^x(f)(x)\in[x,x+1).\] 
We can define the \emph{Euler cocycle based at $x$}\index{Euler cocycle},
denoted as \[\eu^x\co \Homeo_+(S^1)\times\Homeo_+(S^1)\to\{0,1\},\]
by the following condition:
\[s^x(f) s^x(g) =  T(\eu^x(f,g)) s^x(fg).\]
Note that the Euler cocycle is a 2-cocycle in group cohomology.

It is a routine exercise to see that $\eu^x$ is indeed a cocycle and takes values in $\{0,1\}$.
Following~\cite{BFH2014}, we denote the corresponding cohomology classes 
\[\eu
=[\eu^x]\in H^2(\Homeo_+(S^1);\bZ),\quad 
\eu_b
=[\eu^x]\in H^2_b(\Homeo_+(S^1);\bZ).\]
The classes $\eu$ and $\eu_b$ are called the \emph{Euler class}\index{Euler class}  and the \emph{bounded Euler class}\index{Euler class!bounded}, respectively. These classes are independent of the choice of $x\in\bR$. 
Moreover,
 \[\eu\in H^2(\Homeo_+(S^1);\bZ)\] corresponds to the universal central extension 
\[
\xymatrix{
1\ar[r]& \form{T} \ar[r]& \Homeo_\bZ(\bR)\ar[r]^p& \Homeo_+(S^1)\ar[r]& 1}\]

For each $\rho\in\Hom(L,\Homeo_+(S^1))$, we 
have the pull-back of the Euler cocycle
\[
\rho^*\eu^x(a,b)=\eu^x(\rho(a),\rho(b)),\]
which in turn yields cohomology classes
\[
\rho^*\eu\in H^2(L;\bZ),\quad
\rho^*\eu_b\in H^2_b(L;\bZ).\]
The rotation number 
is determined by the Euler cocycle:
\begin{lem}\label{lem:eu-rot}
For each $g\in \Homeo_+(S^1)$, we have
\[ 
\rot(g)=\lim_{n\to\infty}\frac1n\sum_{k=1}^n \eu^0(g,g^k) \mod \bZ.\]
\end{lem}

The proof is given in Lemma~\ref{lem:eu-tau}.

\medskip
\noindent {\bf Semi-conjugacy}\\
The following diagram  
\[
\xymatrix{
L\ar[r]\ar[d] & A\ar[d]^i\\
B\ar[r]^j & C}\]
will be used to define the \emph{fiber product}\index{fiber product} as follows. Given maps $i\co A\to C$ and $j\co B\to C$, the fiber product $L$ of $i$ and $j$ is defined by:
\[
L = \{(a,b)\in A\times B\co i(a) = j(b)\}.\]
If we define $p : L \to A$ and $q : L \to B$ as the restriction of the projections to factors,
then we have $i \circ p = j \circ q$. Conversely, if there is a group $L'$ and maps $p' : L' \to A$
and $q' : L' \to B$ such that $i \circ p' = j \circ q'$, then there is a map $\phi : L' \to L$ such that
$p' = p \circ \phi$ and $q' = q \circ \phi$. Note that $\phi$ may not be an isomorphism.
 
Let $X$ and $Y$ be topological spaces,
and let
 \[\rho_0\in\Hom(L, \Homeo(X)),\quad \rho_1\in\Hom(L, \Homeo(Y)).\]
We will write 
$\rho_0\succcurlyeq_h\rho_1$
if there exists a (possibly discontinuous) map
$h\co X\to Y$ such that
the following diagram commutes for all $g\in L$:
\[
\xymatrix{
	X\ar[d]^<<<<h \ar[r]^{\rho_0(g)} & X\ar[d]^<<<<h\\
	Y \ar[r]^{\rho_1(g)} & Y
}\]

We say $h$ \emph{semi-conjugates}\index{semi-conjugacy!semi-conjugating map} $\rho_0$ to $\rho_1$.

The following definition is due to Ghys~\cite{Ghys1999}, with a slight modification by Mann~\cite{Mann-hb}.

\bd\label{defn:r-semi}
\be
\item
We say two actions \[\rho_0,\rho_1\in\Hom(L,\Homeo_+(\bR))\] are \emph{semi-conjugate}\index{semi-conjugacy} 
if there exists a proper nondecreasing map $h\co \bR\to\bR$ such that
 $\rho_0\succcurlyeq_h\rho_1$.
\item 
We say two actions \[\rho_0,\rho_1\in\Hom(L,\Homeo_+(S^1))\] are \emph{semi-conjugate},
and write $\rho_0\sim_{\mathrm{semi}}\rho_1$,
if there exist
\begin{itemize}
\item
a group $\tilde L$ and a homomorphism $q\co \tilde L\to L$, 
\item 
two semi-conjugate actions in the sense of part (1):
\[\tilde\rho_0,\tilde\rho_1\co \tilde L\to\Homeo_\bZ(\bR),\] 
\item
a commutative diagram
\[
\xymatrix{
&& \tilde L\ar[r]^q\ar[d]^{\tilde\rho_i} &  L\ar[d]^{\rho_i}\ar[dr]\\
1\ar[r]&\bZ=\form{T}\ar[r]\ar[ru] & \Homeo_\bZ(\bR) \ar[r]^p &\Homeo_+(S^1)\ar[r] & 1
}\]
\end{itemize}
for each $i=0,1$  such that the bottom row is the natural universal central extension and such that the middle square is a fiber product.
\ee
\ed

\begin{lem}\label{lem:equiv}
The semi-conjugacy is an equivalence relation on $\Hom(L,\Homeo_+(\bR))$ and also on
$\Hom(L,\Homeo_+(S^1))$.\end{lem}

\bp
The statement is an immediate consequence of Lemma~\ref{lem:real-me}.\ep

\bd\label{defn:monotone}
Let $M=\bR$ or $M=S^1$, and let $h\co M\to M$ be a (possibly discontinuous) map.
\be
\item
For $M=\bR$, we say $h$ is a \emph{monotone degree one map}\index{monotone degree one map} if
$h$ is nondecreasing and if $h(x+1)=h(x)+1$ for each $x\in \bR$.
\item
For $M=S^1$, we say $h$ is a \emph{monotone degree one map} if 
$h$ has a lift to $\bR$ which is monotone degree one. \ee
\ed

\begin{rem}
\be
\item
A nondecreasing map $h\co \bR\to\bR$ is proper if and only if
\[ \lim_{x\to\pm\infty}h(x)=\pm\infty.\]
In particular, we require $h$ to be non-constant.
\item In Definition~\ref{defn:r-semi} (2), a proper nondecreasing map $h$ satisfying $\tilde \rho_0\succcurlyeq_h\tilde \rho_1$ must be  monotone degree--one.  For, we should have \[ h\circ\tilde\rho_0(T) =\tilde\rho_1(T) \circ h.\]
\item Some authors use the terminology that $\rho_0$ is ``semi-conjugate'' to $\rho_1$ if there exists a \emph{continuous} monotone degree one map $h\co M\to M$ such that $\rho_0\succcurlyeq_h \rho_1$,
where $M=\bR$ or $M=S^1$. This notion of ``semi-conjugacy'' is not an equivalence relation. 
\ee
\end{rem}

We let $H^2_{\{0,1\}}(L;\bZ)$ be the set of cohomology classes in $H^2_b(L;\bZ)$
that can be represented by $\{0,1\}$--valued cocycles.
In particular, we have 
\[\eu_b\in H^2_{\{0,1\}}(\Homeo_+(S^1);\bZ).\]
An algebraic characterization of a semi-conjugacy class
is given as follows.

\begin{thm}[\cite{Ghys1999,Ghys2001}]\label{t:ghys2} The bounded Euler class is a semi-conjugacy
invariant. More precisely, 
there exists a one-to-one correspondence:
$$\Hom( L,\Homeo_+(S^1))/\!\sim_{\mathrm{semi}} \to H^2_{\{0,1\}}( L;\bZ),$$ given by
	$$ \quad [\rho]\mapsto \rho^* \eu_b.$$
	
\end{thm}

\noindent {\bf Limit set of a circle action}\\
A topological action of a group $L$ on a space $X$ is called \emph{minimal}\index{minimal!topological action} if every $L$--orbit is dense in $X$.
If $L\le\Homeo_+(S^1)$, then $L$ acts minimally on its \emph{limit set},\index{limit set}
which is defined as
\[
\Lambda(L)\co =\bigcap_{x\in S^1}\overline{L.x}.\]

There is a trichotomy for $\Lambda_L$ as below; see \cite[Theorem 2.1.1]{Navas2011}.
\be[(i)]
\item $\Lambda(L)=S^1$; in this case, $L$ is minimal.
\item $\Lambda(L)$ is a Cantor set, which is called an \emph{exceptional minimal set}\index{exceptional minimal set}.
\item $\Lambda(L)$ is finite (possibly empty); in this case, there exists a finite $L$--orbit.
\ee

\begin{rem}
When $\Lambda(L)$ is nonempty, then it is the \emph{smallest} nonempty minimal $L$--invariant closed subset of $S^1$,
and coincides with the notion of the \emph{minimal set} in the literature.
If $\Lambda(L)$ is empty, then a \emph{minimal} set means a finite orbit of $L$.
Conversely, if $L$ has a finite orbit, then we have the case (iii), and $\Lambda(L)$ may or may not be empty.\end{rem}

For a representation $\rho\co L\to\Homeo_+(S^1)$, we will also use the notation
\[\Lambda(\rho) =\Lambda\circ\rho(L)=\Lambda(\rho(L)).\]

\begin{lem}\label{lem:min-normal}
Let $H\le G\le\Homeo_+(S^1)$ be groups
such that $\Lambda(H)$ is nonempty. Assume either
\be[(a)]
\item $H\unlhd G$, or
\item  $[G:H]<\infty$.
\ee
Then $\Lambda(H)=\Lambda(G)$.
\end{lem}
\bp
Note that  $\Lambda(H)\sse\Lambda(G)$, and $\Lambda(G)$, being non-empty, is the unique minimal closed non-empty invariant set.
Suppose (a). Then for each $g\in G$ we have
\[g\Lambda(H)=\bigcap_{x\in S^1}\overline{gH.x}
=\bigcap_{x\in S^1}\overline{gHg^{-1}.gx}=\Lambda(H).\]
By the uniqueness of the minimal set $\Lambda(G)$, we see $\Lambda(G)\sse\Lambda(H)$.

We now assume (b). There exists a finite index normal subgroup $H_0$ of $G$ such that $H_0\le H$. By part (a), we have
\[
\Lambda(G)=\Lambda(H_0)\sse\Lambda(H).\qedhere\]\ep

\noindent {\bf Blow-up and minimalization}\\
In Definition~\ref{defn:r-semi}, it is not required to have a continuous semi-conjugating map
 $h\co\bR\to\bR$ such that $\tilde\rho_0\succcurlyeq_h\tilde\rho_1$.
It is often desirable to have notions of semi-conjugacy defined by \emph{surjective} (hence, continuous) semi-conjugating maps. These notions can be defined by \emph{blowing-up} and \emph{minimalization}. 

Recall our notation that $T(a)$ is the translation by $a$ on $\bR$ or on $S^1$. We sometimes use $T$ for the translation by $1$.

\bd\label{defn:minimalization}
Let $\rho$ and $\bar\rho$ be circle actions of $L$.
\be
\item
We say $\rho$ is a \emph{blow-up}\index{blow-up} of $\bar\rho$ 
if $\rho\succcurlyeq_h\bar\rho$ for some surjective monotone degree one map $h$ on $S^1$.
\item
We say $\bar\rho$ is a \emph{minimalization}\index{minimalization} of $\rho$
if one of the following holds:
\be[(i)]
\item
$\rho(L)$ has a finite orbit,
and $\bar\rho = T (\rot\circ\rho)$;
\item
$\rho$ is a blow-up of $\bar\rho$,
and $\bar\rho$ is minimal.
\ee\ee
\ed
\begin{rem}
The cases (i) and (ii) are mutually exclusive. 
\end{rem}

A minimalization can be regarded as a ``standard form'' of a circle action. The proof of the following lemma is given in the Appendix,
after Lemma~\ref{lem:fin-orb}.

\begin{lem}\label{lem:minimal-exists}
Every circle action admits a minimalization, which is unique up to conjugacy.\end{lem}

The most interesting case of the above lemma is when the limit set is a Cantor set. In this case, the semi-conjugating map $h$ is the Cantor function and also called as the \emph{devil's staircase function}\index{devil's staircase function}.

For $f,g\in\Homeo_+(S^1)$ and their arbitrary lifts $\tilde f,\tilde g\in\Homeo_\bZ(\bR)$,
Matsumoto~\cite{Matsumoto1986}  defined the \emph{canonical Euler cocycle}\index{canonical Euler cocycle} as
\[
\tau(f,g)=\rot^\sim(\tilde f\tilde g)-\rot^\sim(\tilde f)-\rot^\sim(\tilde g).\]
Then $\tau$ is a continuous real-valued 2-cocycle and independent of the choice of lifts.
Using our notation for a section
\[
s=s^0\co \Homeo_+(S^1)\to\Homeo_\bZ(\bR),\]
we can more succinctly write
\[\tau = -s^*(\partial\rot^\sim).\]

The following Theorem reconciles some of the notions of ``semi-conjugacy'' in the literature. 

\begin{thm}\label{thm:equiv}
Let $L$ be a countable group.
For two circle actions $\rho_0,\rho_1$ of $L$,
the following statements are all equivalent.
\be
\item\label{part:r-semi}
$\rho_0$ and $\rho_1$ are semi-conjugate.
\item\label{part:monotone}
$\rho_0$ and $\rho_1$ have a common blow-up.
\item\label{part:min}
$\rho_0$ and $\rho_1$ have a common minimalization.
	\item\label{part:eu}
$\rho_0$ and $\rho_1$ have the same bounded Euler class.
\item\label{part:matsumoto}
We have $\rot\circ\rho_0=\rot\circ\rho_1$
and $\rho_0^*\tau=\rho_1^*\tau$ as maps on $L$ and on $L\times L$, respectively.
\item\label{part:mann-path}
There exists a continuous path \[\{\rho_t\co t\in [0,1] \}\sse\Hom(L,\Homeo_+(S^1))\] such that for each $g\in L$, the value $\rot\circ\rho_t(g)$ is constant on $t\in  [0,1] $.
\item\label{part:symmetric}
$\rho_0\succcurlyeq_u\rho_1$ and $\rho_1\succcurlyeq_v\rho_0$
for some monotone degree one maps $u$ and $v$ on $S^1$.
\ee
\end{thm}

Theorem~\ref{thm:equiv} now provides various equivalent definitions of semi-conjugacy. In the Appendix, we will give a self-contained proof for
\[
\xymatrix{
(\ref{part:symmetric})&
(1)\ar@{=>}[l]\ar@{=>}[r] & (2)\ar@{=>}[r]\ar@{=>}[ld] & (3)\ar@{=>}[d] \\
&(6)\ar@{=>}[r] & (5) & (4)\ar@{=>}[l]
}\]
which are the most relevant parts of the theorem for this monograph. 
We will then indicate references for elementary proofs of 
(5)$\Rightarrow$(1) and (\ref{part:symmetric})$\Rightarrow$(\ref{part:eu}),
thus establishing all the equivalence. 
See also the references in the following remarks.
We emphasize again that in (\ref{part:r-semi}) and (\ref{part:symmetric}), semi-conjugating maps are not required to be continuous.

\begin{rem}
Mann~\cite{Mann-hb} suggested the condition (\ref{part:r-semi}), in order to modify Ghys' original definition of semi-conjugacy to become an equivalence relation. She also used the condition (\ref{part:mann-path}) while studying connected components of surface groups representations~\cite[Corollary 7.5 and Proposition 7.6]{Mann2015IM}.
The condition (\ref{part:min}) naturally stems from an exposition of Ghys~\cite{Ghys2001}. Matsumoto~\cite{Matsumoto1986} gave an essentially equivalent definition to~(\ref{part:matsumoto}).
Extensive study on the properties of $\tau$ and the rotation numbers of longer words has been done in the literature; see~\cite{JN1985MA,Naimi1994CMH,CW2011}.
The condition (\ref{part:symmetric}) is probably the simplest definition among above, and due to Takamura~\cite{Takamura} and Bucher--Frigerio--Hartnick~\cite{BFH2014}. 
\end{rem}

\begin{rem}\label{rem:semi-conj-history}
One can define a \emph{weak monotone equivalence}\index{monotone equivalence!weak} as the equivalence relation \emph{generated} by a $\succcurlyeq$--relation where the semi-conjugating monotone degree one map is required to be surjective. 
This equivalence was originally called  \emph{monotone equivalence}\index{monotone equivalence} in~\cite{CD2003IM}. The agreement of these two notions of monotone equivalence was noted in~\cite[Lemma 2.22]{Calegari2007}, in which book the equivalence of (\ref{part:monotone}) and (\ref{part:eu}) is also shown.
The agreement between weak monotone equivalence and the condition (\ref{part:symmetric}) can be found in \cite{BFH2014}. 
\end{rem}

\begin{cor}\label{cor:fin-orbit}
Let $L$ be a countable group
and let $\rho_0,\rho_1$ be two semi-conjugate circle actions of $L$.
If $\rho_0$ has a finite orbit of cardinality $N>0$, then so does $\rho_1$.
\end{cor} 

For a proof, see Lemma \ref{lem:exponent} in the Appendix.
The equivalence of (\ref{part:r-semi}) and (\ref{part:mann-path}) in Theorem~\ref{thm:equiv} implies the following.

\begin{cor}[{\cite[Proposition 7.6]{Mann2015IM}}]\label{cor:semi-cnt}
In the space of the circle actions of a finitely generated group, each semi-conjugacy class is contained in some path component.  
\end{cor} 

We lastly record a simple corollary for $\bZ=\form{1}$.
\begin{cor}\label{cor:eu-rot}
For circle actions $\rho_0,\rho_1$ of $\form{s}\cong\bZ$,
we have $\rho_0\sim_{\mathrm{semi}}\rho_1$
if and only if 
\[\rot\circ\rho_0(s)=\rot\circ\rho_1(s).\] 
\end{cor}

\noindent {\bf Minimal quasimorphisms}\\ 
For a group $ G $, let $\mathrm{QM}( L ;\bZ)$ denote the abelian group of integer--valued quasimorphisms on $ G $. 
Define 
\[\mathrm{HQM}( L ;\bZ)=\mathrm{QM}( L ;\bZ)/\left(B( L ;\bZ)\oplus H^1( L ;\bZ)\right)\] where $B( L ;\bZ)$ consists of bounded maps $ G \to\bZ$.

The \emph{defect}\index{defect} $D(\alpha)$ of  $\alpha\in\mathrm{QM}( L ;\bZ)$ is defined by \[D(\alpha)=\sup|\partial\alpha|
=\sup\{|\alpha(g)+\alpha(h)-\alpha(gh)|\co g,h\in L \}.\]
If $\alpha\in\mathrm{QM}( L ;\bZ)$ has defect--one,
then $\alpha$ has the smallest possible defect among nontrivial integral quasimorphisms.
\bd\label{defn:minimal-qm}
Let $L$ be a group.
A map $\alpha\co L\to\bZ$ is called a \emph{minimal quasimorphism}\index{quasimorphism!minimal}
if $\alpha(g)+\alpha(h)-\alpha(gh)\in\{0,1\}$ for all $g,h\in L$ and $\alpha$ not a homomorphism.
\ed
In other words, a minimal quasimorphism on $L$ is an subadditive, integer--valued defect--one quasimorphism.
For example, the quasimorphism $n\mapsto -\lceil \pi n\rceil$ is minimal.

Recall from the introduction that
we call a circle action $\rho$ of $L$  \emph{liftable}\index{liftable} if $\rho$ factors as
\[L\to  \Homeo_\bZ(\bR)\to\Homeo_+(S^1).\]
This is equivalent to the condition \[\rho^*\eu_b\in\ker(H^2_b(L;\bZ)\to H^2(L;\bZ)).\]
So if $\rho$ is liftable, then we have an equality of cocycles
\[\rho^*\eu^0=\partial\alpha\] for some map $\alpha\co L\to\bZ$.  
Since $\rho^*\eu^0$ is $\{0,1\}$-valued, we see that 
\[\partial\alpha(x,y)=\alpha(x)+\alpha(y)-\alpha(xy)\in\{0,1\}.\] 
In other words, $\alpha$ is a minimal quasimorphism 
or a homomorphism. 
Note \[[\alpha]\in\mathrm{HQM}( L ;\bZ)\] maps to $\rho^*\eu_b$ by the 
well-known exact sequence~\cite{Calegari2007}:
\[0\to \mathrm{HQM}( L ;\bZ)\to H^2_b( L ;\bZ)\to H^2( L ;\bZ).\]
By Theorem~\ref{t:ghys2}, we have the following.
\begin{lem}\label{lem:qm}
Suppose we have liftable 
circle actions $\rho_0,\rho_1$ of $L$. 
If $\rho_0$ and $\rho_1$ are not semi-conjugate,
then the quasimorphisms corresponding to $\rho_0$ and $\rho_1$ are distinct in $\mathrm{HQM}(L;\bZ)\sse H^2_b(L;\bZ)$.
\end{lem}

\subsection{Hyperbolic Geometry}\label{ss:hyp}
Throughout this monograph, we reserve the letter $\bK$ to denote $\bR$ or $\bC$.
We summarize some required facts from hyperbolic geometry of two and three dimensions.

\noindent {\bf Maximal abelian subgroups}\\ 
By identifying $\PSL_2(\bR)$ with the group of M\"obius transformations on the upper half-plane $\bH^2$ of $\bC$, we have 
 \[\PSL_2(\bR)\cong \Isom^+(\bH^2)\le\Homeo_+(S^1).\]
In particular, we have a restriction \[\rot\co \PSL_2(\bR)\to \bR/\bZ.\]
We also have
 \[\PSL_2(\bC)\cong \Isom^+(\bH^3)\le\Homeo_+(S^2).\]
A group $\mu\le\PSL_2(\bK)$ is \emph{elementary}\index{elementary subgroup} if $\mu$ admits a finite orbit on $\bH^d\cup\partial\bH^d$ for $d=2$ or $3$.

For $t\in\bK$, let us define the following elements in $\PSL_2(\bK)$:
\[
\Rot(t)=
\begin{pmatrix}
\cos(t/2) & \sin(t/2)\\
-\sin(t/2) & \cos(t/2)
\end{pmatrix},\quad
\exp(t)=
\begin{pmatrix}
e^{t/2} & 0\\
0 & e^{-t/2}
\end{pmatrix},\quad
p(t)=
\begin{pmatrix}
1 & t\\
0 & 1
\end{pmatrix}.
\]

Let $t\in \bR$. Then $\Rot(t)$ is an elliptic rotation about $i\in\bH^2$ by the angle $t$.
The elements $\exp(t)$ and $p(t)$ are the hyperbolic translation by distance $t$ along the geodesic  $[0,\infty]$ and a parabolic isometry fixing $\infty$, respectively.
Each $g\in \PSL_2(\bR)$ is conjugate to one of these three types.
In particular,
\begin{equation}\label{eq:trace}
|\tr g |= \begin{cases} |2\cos(\pi\rot( g ))|, &\text{ if } g \text{ is elliptic;}\\ 2\cosh(\ell( g )/2), &\text{ if } g \text{ is hyperbolic;}\\
2 &\text{ if } g \text{ is parabolic.} \end{cases}\end{equation}
Here, $\ell(g)$ denotes the translation length of $g$.
We always regard the identity as an elliptic, parabolic and hyperbolic isometry simultaneously.

Let us now suppose $g\in\PSL_2(\bC)$.
Then $g$ is conjugate to either $\exp(t)$ or $p(t)$ for some $t\in \bC$. We note that the following element is a rotation with the $z$--axis:
\[
\exp(it)=\begin{pmatrix}
e^{it/2} &  0 \\
0 & e^{-it/2}
\end{pmatrix}.\]

Let $C\le G$ be groups. We define the \emph{normalizer} and the \emph{centralizer} of $C$ as 
\begin{align*}
N_G(C)&= N(C) = \{g\in G \co g^{-1}Cg = C\},\\
Z_G(C)&= Z(C) = \{g\in G \co [g,c]=1\text{ for all }c\in C\} =\bigcap_{c\in C}Z_G(c).
\end{align*}

\begin{prop}\label{prop:maxabel}
\be
\item
A maximal abelian subgroup of $\PSL_2(\bR)$ is conjugate to one of the following groups:
\[
\SO(2)=\Rot(\bR),\quad
\exp(\bR),\quad
p(\bR).\]
\item
A maximal elementary subgroup of $\PSL_2(\bR)$ is conjugate to one of the following groups:
\[
\SO(2)=\Rot(\bR),\quad
N(\exp(\bR))=\exp(\bR)\rtimes\Rot(\pi),\quad
p(\bR)\rtimes\exp(\bR).\]
\item
A maximal abelian subgroup of $\PSL_2(\bC)$ is conjugate to one of the following groups:
\[
\exp(\bC),\quad
p(\bC).\]
\item
A maximal elementary subgroup of $\PSL_2(\bC)$ is conjugate to one of the following groups:
\[
N(\exp(\bC))=\exp(\bC)\rtimes\Rot(\pi),\quad
p(\bC)\rtimes\exp(\bC).\]
\ee
\end{prop}

We let $\bK=\bR$ or $\bC$. 
The following groups are called \emph{elliptic, hyperbolic} and \emph{parabolic} maximal abelian groups, respectively:
\[
\SO(2)=\Rot(\bR),\quad
\exp(\bK),\quad
p(\bK).\]

For a maximal abelian subgroup $\mu$ of $\PSL_2(\bK)$, 
we let $\Fix\mu$ denote the set of points $x\in\bH^d\cup\partial\bH^d$ which is fixed by all the elements in $\mu$.
The set $\Fix\mu$ consists of one or two points for $d=2$ or $3$.
The following are simple exercises in hyperbolic geometry.

\begin{lem}\label{lem:parabolic-test}
Let $\mu$ be a maximal abelian subgroup of ${\PSL_2(\bK)}$.
\be
\item
For $g\in\PSL_2(\bK)$, the following are equivalent:
\be[(i)]
\item $\Fix\mu = g\Fix\mu$;
\item $g\in N(\mu)$;
\item $[c,gcg^{-1}]=1$ for all $c\in \mu\setminus1$.
\ee
\item
For $g\in\PSL_2(\bK)$, 
we have $\Fix\mu\cap g\Fix\mu\ne\varnothing$
if and only if one of the following hold:
\be[(i)]
\item
$\mu$ is not hyperbolic and $g\in\mu$;
\item
$\mu$ is hyperbolic and
$\tr^2[c,gcg^{-1}]=4$  for all $c\in \mu\setminus1$.\ee
\ee
\end{lem}

\noindent {\bf Finite type hyperbolic $2$--orbifolds}\\ 
A \emph{finite type $2$--orbifold}\index{$2$--orbifold!finite type} $S$ is a geometric structure on an oriented surface of finite topological type~\cite{Thurston1997book},
which is prescribed by its  \emph{signature}\index{$2$--orbifold!signature}
\[(g;m_1,m_2,\ldots,m_k).\] 
Here, $g$ is the genus of $S$ and the integers $2\le m_1,\ldots,m_k\le\infty$ are the orders of cone points for some $k\ge0$. The order $m_i=\infty$ corresponds to a puncture or a boundary. If $S$ is a surface, then $k=0$.

The Euler characteristic $\chi(S)$ and the \emph{complexity}\index{$2$--orbifold!complexity} 
$\xi(S)$ are then  defined as:
\begin{align*}
\chi(S) &= 2-2g - \sum_{i=1}^k\left(1-\frac1{m_i}\right)\\
\xi(S) &= 3g -3 + k.
\end{align*}
There is a complete hyperbolic structure on $S$ if and only if $\chi(S)<0$.

A discrete subgroup $L$ of $\PSL_2(\bR)$ is called a \emph{Fuchsian group}\index{Fuchsian group}.
In this case, the quotient space $S=\bH^2/L$ is a complete hyperbolic $2$--orbifold and we write \[L=\pi_1^{\mathrm{orb}}(S).\] 
If $L$ is finitely generated, then $S$ is of finite type.
If, furthermore, $L$ has a finite co-volume, then we say $L$ is a \emph{lattice}. If $S$ is compact, then $L$ is said to be \emph{uniform}. Every finitely generated Fuchsian group is isomorphic to some lattice, since one can contract a geodesic boundary of a complete hyperbolic $2$--orbifold to a cusp without changing its fundamental group.

\noindent {\bf Commutative--transitive groups}\\ 
We say a group $G$ is \emph{commutative--transitive}\index{commutative--transitive group} if $[a,b]=[a,c]=1$ implies $[b,c]=1$ for $a,b,c\in G\setminus 1$.
Equivalently,  all centralizers are abelian. The group $\Isom^+(\bH^n)$ is an example of a commutative--transitive group.
Every limit group is also commutative--transitive~\cite{BF2009book}.

\begin{lem}\label{lem:malnormal}
Let $G$ be a commutative--transitive group,
and let $C$ be a nontrivial abelian subgroup of $G$.
\be
\item
Then $Z(C)$ is the unique maximal abelian subgroup of $G$ containing $C$.
\item If $C$ is malnormal in $G$, then $C$ is maximal abelian in $G$.
\item If $C$ is a maximal abelian subgroup of $G$, then we have
\begin{align*}
N(C)
&=\left\{g\in G \co [c^g,c]=1\text{ for all }c\in  C\right\}\\
&=\left\{g\in G \co [c^g,c]=1\text{ for some }c\in C\setminus1\right\}\\
&=\left\{g\in G \co C\cap C^g\ne 1\right\}.
\end{align*}
\item\label{tr-normalizer}
Let $A, \mu$ be subgroups of $G$ such that $1\ne C\le A\cap \mu$.
Assume that $\mu$ is maximal abelian in $G$ and that $C$ is malnormal in $A$.
Then 
\[C=A\cap \mu = A\cap N(\mu).\]
\ee
\end{lem}
\bp
For part (1), 
let $C\le G$ be a nontrivial abelian group and $c\in C\setminus 1$.
Then commutative--transitivity of $G$ implies that the centralizer group $Z(C)=Z(c)$ is the desired maximal abelian group containing $C$.

Part (2) is a good exercise. For part (3),  let us write the given statement as
\[N(C)=C_1=C_2=C_3.\]
Then the inclusions
\[N(C)\sse C_1\sse C_2\sse C_3\]
are immediate by definition and the maximality. To see $C_3\sse N(C)$, let us suppose
$C\cap C^g\ne1$ for some $g\in G$. 
We can find $c_0,c_1\in C\setminus1$ such that $c_0=c_1^g$.
Then 
\[C = Z(c_1) = Z(c_0)^g = C^g\]
and hence, $g\in N(C)$.

For part (\ref{tr-normalizer}), let $g\in A\cap N(\mu)$.
Since $C^g\le \mu^g=\mu$, for each $c\in C$ we have 
$[c,c^g]=1$.
By part (2), the group $C$ is maximal abelian in $A$.
So, part (3) implies that 
\[
g\in N_A(C)=C,\]
as desired.
\ep

\noindent {\bf Minimalization of Fuchsian groups}\\ 
The following characterizes minimalization of Fuchsian groups.

\begin{lem}\label{lem:min-fuchs}
Suppose we have a non-elementary Fuchsian representation 
\[\rho_0\co L\to\PSL_2(\bR)\]
for some finitely generated torsion-free group $L$
such that $\Lambda(\rho_0)\ne S^1$.
Then there exists a path $\{\rho_t\}$ of Fuchsian representations of $L$ such that $\Lambda(\rho_1)=S^1$ and such that  $\rho_1(L)$ contains a nontrivial parabolic element. Furthermore, $\rho_1$ is a minimalization of $\rho_0$.
\end{lem}

\bp
Let $X$ denote the space of the Fuchsian representations of $L$.
Put $S_0=\bH^2/\rho_0(L)$. The surface $S_0$ contains at least one geodesic funnel, represented by say, $h\in L$.

There exists a continuous deformation $\{S_t\}$ of the hyperbolic surface $S_0$ to a finite-area hyperbolic surface $S_1$. This is done by deforming all the geodesic funnels to cusps; the details of this construction are worked out by Floyd in~\cite{Floyd}.

Let $\{\rho_t\}\sse X$ be a path corresponding to $\{S_t\}$.
Then $\rho_1(h)$ must be a nontrivial parabolic element.
Part (\ref{part:mann-path}) of Theorem~\ref{thm:equiv} implies that $\rho_1$ is semi-conjugate to $\rho_0$. Hence, $\rho_1$ is a minimalization of $\rho_0$.\ep

\subsection{Indiscrete Subgroups of $\PSL_2(\bR)$}\label{sec:indiscretepsl2r}
It will be useful for us to have a dynamical characterization of subgroups of $\PSL_2(\bR)$ which are not discrete. The following lemma is well-known to experts in algebraic groups and Lie groups; see~\cite[Theorem 8.4.1]{Beardon1983GTM},~\cite{KatokBook} and \cite{Raghunathan1972} for instance. We include a proof for the convenience of the reader.

\begin{lem}\label{lem:dense}
For a finitely generated non-elementary subgroup $L$ of $\PSL_2(\bR)$, the following are equivalent:
\begin{enumerate}
\item
The rotation spectrum of $L$ is infinite;
\item
Some element of $L$ is an elliptic element of infinite order;
\item
The group $L$ is indiscrete;
\item
The group $L$ is dense.
\end{enumerate}
\end{lem}

\bp
Note that $L$ is Zariski dense, since a subgroup of $\PSL_2(\bR)$ which is not virtually solvable cannot be contained in a Borel subgroup and hence is Zariski dense.

We note $L$ is dense if and only if it is Zariski dense and indiscrete. Indeed, if $L$ is dense then its Zariski closure is a closed subgroup of $\PSL_2(\bR)$ containing $L$, and hence is all of $\PSL_2(\bR)$.  Conversely, suppose $L$ is Zariski dense and indiscrete. Then the closure of $L$ in the usual topology on $\PSL_2(\bR)$ contains a component of positive dimension which is a closed subgroup and is hence a Lie subgroup of $\PSL_2(\bR)$. Since $L$ is Zariski dense, this subgroup must be all of $\PSL_2(\bR)$. Thus, we establish the equivalence of (3) and (4).

The density of $L\le\PSL_2(\bR)$ implies that $\tr^2(L)$ is dense in $[0,\infty)$. This implies that $\rot(L)$ is dense in $\bR/\bZ$. So, we see the implication (4)$\Rightarrow(1)$.

To see that (1) implies (2), suppose $L\le\PSL_2(\bR)$ has infinite rotation spectrum but $L$ does not contain an irrational rotation.
Then $L$ must contain elliptic elements of arbitrarily high order. 
Moreover, being a finitely generated linear group, we have that $L$ contains a finite index subgroup $H$ which is torsion-free by Selberg's Lemma. 
If $n$ is the index of $H$ in $L$ and $g \in L$ is arbitrary, then $g^{n! }\in H$. It follows that $L$ must have a bounded exponent, a contradiction.

The implication (2)$\Rightarrow$(3) is obvious. 
\ep

We remark that Lemma \ref{lem:dense} really is specific to $\PSL_2(\bR)$. It fails already in $\PSL_2(\bC)$, where ``infinite rotation spectrum" is replaced by ``infinite order elliptic element" (see \cite{MaloniPalesiTan}).

To illustrate Lemma \ref{lem:dense} a little, we give an example of a dense subgroup of $\PSL_2(\bR)$ with infinite purely arithmetic rotation spectrum, and we explicitly find elliptic elements (here by purely arithmetic rotation spectrum we mean that the rotation spectrum belongs to multiples of $2\pi$ by elements of $\bZ[a_1,\cdots,a_n]$ for some algebraic integers $a_i$).
Consider the ring  $\bZ[\sqrt{2}]$. For an arbitrary $\epsilon > 0$, choose  elements \[a,b,c,d \in \bZ[\sqrt{2}]\] such that
\begin{enumerate}[(i)]
\item $0<a+d<\epsilon$ (this is possible as $\bZ[\sqrt{2}]$ is dense in $\mathbb R$),
\item $b=1, c=ad-1$.
\end{enumerate}

 The matrix (projected to a projective transformation) \[\begin{pmatrix} a& b\\ c&d\end{pmatrix}\in\PSL_2(\bR)\] is easily checked to lie in the finitely generated dense subgroup \[\PSL_2(\bZ[\sqrt{2}])<\PSL_2(\bR),\] and to be an  elliptic element of $\PSL_2(\bR)$. 
The rotation number $\theta$ is given by 
\[2\cos\frac\theta2 = a+d,\] which in turn lies in the interval 
\[[2\arccos(\epsilon/2),\pi],\] thus providing infinitely many distinct arithmetic rotation numbers.
More generally, such examples can be constructed from Hilbert modular groups $\PSL_2(\mathcal{O})$, where $\mathcal{O}$ is the ring of integers in a totally real number field.

The following observation is an immediate consequence of Lemma \ref{lem:dense}:

\begin{cor}\label{cor:discrete}
If $L\le \PSL_2(\bR)$ is a finitely generated group such that $\rot(L)=\{0\}$,
then $L$ is torsion-free and discrete.
\end{cor}


\section{Topological Baumslag Lemmas}\label{sec:baumslag}
This section deals with one of the principal technical tools of the monograph, namely an extension of Baumslag's Lemma. The following are some of the main ingredients in this section:

\begin{enumerate}
	\item A Topological Baumslag Lemma, which gives sufficient conditions to guarantee nontriviality of a word \[w (t_1,\cdots , t_k) = g_1 \mu_1(t_1) \cdots g_k \mu_k (t_k)\] for large  values of the parameters $t_i$. Here the
	$\mu_j(t_j)$s are one-parameter subgroups of a continuous (possibly analytic) group. The proof is quite general and reminiscent of Tits'
	proof \cite{tits-jalg} of the Tits' alternative for discrete linear groups in that the underlying idea consists of a  ping-pong argument. 
	\item The one-parameter subgroups are often (generalizations of) parabolic and hyperbolic one-parameter subgroups of $\PSL_2(\bR)$. To handle elliptic subgroups we use a complexification and Zariski density trick by embedding $\PSL_2(\bR)$ in	$\PSL_2(\mathbb{C})$ and reduce the elliptic case to the hyperbolic one (Lemma \ref{lem:baumslag-psl}).
\end{enumerate}
	
 We focus primarily on {\it projective} group actions on the circle, i.e. representations into $\PSL_2(\bR)$; but many of the proofs go through
 in the more general context of representations into linear algebraic semi-simple groups, including $\PSL_2(\bC)$.

\subsection{Topological Setting}
We will often denote a generator of $\bZ$ by $s$.
Recall our notation from the introduction:
\[A\ast_C = A\ast_C(C\times\form{s}).\]

Let $F_n$ denote a free group of rank $n$.
Half a century ago, Baumslag used the following lemma to deduce the residual freeness of groups of the form $F_n\ast_{\form{w}}F_n$ for $w\in F_n$.

\begin{lem}[\cite{Baumslag1962}]\label{lem:baumslag}
Suppose we have \[u, g_1,g_2,\ldots,g_k\in F_n\] such that $[u,g_i]\ne1$ for each $i$.
Then there exists $M>0$ such that 
\[g_1 u^{t_1} g_2 u^{t_2}\cdots g_k u^{t_k}\ne1\]
whenever $|t_i|\ge M $ for each $i$.\end{lem}

Sometimes called Baumslag's Lemma, this result plays a fundamental role in the theory of limit groups; see~\cite{Wilton2009solutions} and references therein.
Baumslag's Lemma generalizes for multiple ``twisting words''~\cite{Kim2010,BG2010JAM},
and also for torsion-free word-hyperbolic groups~\cite{GW2010JT}. We will present a continuous version of Baumslag's lemma in this section that will imply both of these generalizations.

Throughout this subsection, we let $T$ be a set (of \emph{parameters}) equipped with a decreasing sequence of subsets 
\[
T\supseteq T_0\supseteq T_1 \supseteq T_2\supseteq\cdots.\]

\bd
\label{defn:attracting}
Let $X$ be a topological space and let $\phi\co T\to\Homeo(X)$ be a map.
\be
\item We say $\phi$ is \emph{attracting (to $C$)}\index{attracting (homeomorphism)} if we have a nonempty proper compact set $C\sse X$ such that for each open neighborhood $U\supseteq C$ there exists $M>0$ satisfying \[\phi(t)^{\pm1}(X\setminus U)\sse U\text{ for all }t\in T_M.\]
\item\label{item:doubly}
We say $\phi$ is \emph{doubly attracting (to $(C^+,C^-)$)}\index{attracting (homeomorphism)!doubly attracting}
if we have a disjoint pair of nonempty compact sets $(C^+,C^-)$ in $X$ such that
$X\ne C^+\cup C^-$
and such that
for each disjoint pair of open neighborhoods $(U^+,U^-)$ of $(C^+,C^-)$
there exists $M>0$ satisfying the following for all $t\in T_M$:
\[\phi(t)(X\setminus U^-)\sse U^+,
\quad
\phi(t)^{-1}(X\setminus U^+)\sse U^-.\]
\ee
\ed

\begin{rem}
\be
\item We do not assume $\phi$ is continuous or a homomorphism.
\item We forbid the possibility of $X=C^+\cup C^-$ in (\ref{item:doubly}),
since we do not want to call the constant map 
\[\phi\co T\to\{\Id\}\]
as doubly attracting.
\item
A doubly attracting map is attracting: we simply set $C=C^+\cup C^-$. 
\ee
\end{rem}

\begin{exmp}\label{exmp:attr}
Let us list some examples of attracting maps.
We let $\bK=\bR$ or $\bC$. 
\be
\item\label{exmp:attr-psl-1}
Suppose $\mu$ is a maximal parabolic subgroup of $\PSL_2(\bK)$.
Up to conjugacy, we can parametrize $\mu$ as
\[ \mu(z)=\begin{pmatrix} 1&z \\ 0& 1\end{pmatrix}. \]
Set $T=\bK$ and $T_M=\{z\in T\co |z|\ge M\}$.
Then for $d=2$ or $3$, the map
\[\mu\co T\to\PSL_2(\bK)\cong\Isom^+(\bH^d)\le\Homeo_+(S^{d-1})\] is attracting to $\Fix \mu$.
\item\label{exmp:attr-psl-2}
A maximal hyperbolic subgroup $\mu\le\PSL_2(\bK)$
can be parametrized as
\[
\mu(z)=\begin{pmatrix} z&0 \\ 0& 1/z\end{pmatrix}
\]
up to conjugacy.
If we let $T=\bK^\times$ and $T_M=\{z\in T\co |z|\ge M\}$,
then
\[\mu\co T\to\PSL_2(\bK)\le\Homeo_+(S^{d-1})\] is again attracting to $\Fix \mu$.
\item 
If $\mu$ is a maximal elementary non-elliptic subgroup of $\Isom^+(\bH^n)$ for some $n\ge3$,
then the natural embedding $\phi\co T=\mu\hookrightarrow\Homeo_+(S^{n-1})$ is attracting to $\Fix \mu$,
for a suitable stratification $T=T_0\supseteq T_1 \supseteq\cdots$.
\item
For $1\le k\le\infty$, we have a $k$--fold cover:
\[\PSL_2^{(k)}(\bR)\to\PSL_2(\bR),\]
Let $\mu\le\PSL_2(\bR)$ be a maximal non-elliptic abelian 
subgroup and $\mu^\sim\le\PSL^{(k)}(2,\bR)$ be the lift of $\mu$. Using the parametrization in~(\ref{exmp:attr-psl-1}) and (\ref{exmp:attr-psl-2}), we see 
\[\phi\co T=\mu^\sim\hookrightarrow\Homeo_+(S^1)\] is attracting to the periodic set $\operatorname{Per}\mu^\sim$.
\item
Let $g\in\Isom^+(\bH^n)$ and define $\phi\co T=\bZ\to\Homeo_+(S^{n-1})$ as $\phi(t)= g^{t}$. We set $T_M=\{t\in\bZ\co |t|\ge M\}$.
Then $\phi$ is attracting to $\Fix g$ if and only if $g$ is non-elliptic.
\item
Let $S$ be a closed hyperbolic surface and let $g\in\Mod(S)$ be  pseudo-Anosov.
We have an action of $\Mod(S)$ on $\partial\mathrm{Teich}(S)=\mathrm{PML}(S)$.
The
 map $\phi\co t\mapsto g^t$ is attracting to the union of the stable and the unstable laminations.
In addition to these actions, we have an action of $\Mod(S,x)$ on $S^1=\partial\bH^2$
such that the map $t\mapsto g^t$ is attracting to the set consisting of even number of points~\cite{CB1988}.
Here, $x$ denotes a base point of $S$.
More detailed account of $\Mod(S)$ action on $S^1$ is given in  Section~\ref{sec:mcg}.
\item
Let $G$ be a torsion-free non-elementary word-hyperbolic group and let $1\ne h\in G$.
The group $G$ acts on $\partial G$ by homeomorphisms \cite{Gromov1987}, 
Then the map $\phi\co t\mapsto h^t$ is attracting to $\{h^\infty,h^{-\infty}\}$.
Note that $\partial G\ne\{h^\infty,h^{-\infty}\}$,
as required by Definition~\ref{defn:attracting} (\ref{item:doubly}).
\ee
\end{exmp}

\begin{lem}\label{lem:doubly}
\be
\item
If $\phi\co T\to\Homeo(X)$ is doubly attracting
for some Hausdorff space $X$,
then for all sufficiently large $M$ and for each $t\in T_M$,
the group $\form{\phi(t)}$ is an infinite cyclic subgroup of $\Homeo(X)$.
\item
Let $\phi\co T\to\Isom^+(\bH^d)\le\Homeo_+(S^{d-1})$ be doubly attracting for some $d\ge2$.
Then for all sufficiently large $M$ and for each $t\in T_M$,
the element $\phi(t)$ is hyperbolic.
\ee
\end{lem}

\bp
(1)
Let $\phi$ doubly attract to $(C^+,C^-)$.
Since $X\ne C^+\cup C^-$ by assumption,
we can find disjoint open neighborhoods $U^\pm$ of $C^\pm$ 
and \[x_0\not\in U^+\cup U^-.\]
If $M$ is sufficiently large and $t\in T_M$,
then 
\[\phi(t)^{\pm1}.x_0\in U^\pm \sse X\setminus U^\mp.\]
So for each $n>0$, we have
 \[\phi(t)^{\pm n}.x_0\in U^{\pm}.\]
In particular, $\phi(t)^{\pm n}\ne\Id$. Now (2) is immediate.
\ep

The following is the main result of this section, and we call it the Topological Baumslag Lemma. Some of the ideas go back to the ping-pong arguments of Klein~\cite{Klein1883MA}, Maskit~\cite{Maskit1988Springer} and Tits \cite{tits-jalg}. 

\begin{lem}[Topological Baumslag Lemma]\label{lem:baumslag-gen}
Let $X$ be a Hausdorff space and let $k\ge1$.
Define  $\psi\co T\to\Homeo(X)$ by
\[\psi(t) = g_1 \phi_1(t)g_2\phi_2(t)\cdots g_k \phi_k(t),\] 
such that the following hold for each $i=1,2,\ldots,k$:
\be[(i)]
\item
 $g_i\in \Homeo(X)$;
\item
$\phi_i\co T\to\Homeo(X)$ is attracting to $C_i$;
\item
$C_i \cap g_{i+1}C_{i+1}=\varnothing$ where indices are taken modulo $k$ (thus including $C_k \cap g_{1}C_{1}=\varnothing$);
\item $X\ne g_1 C_1\cup C_k$.
\ee
Then $\psi$ is doubly attracting to $(g_1 C_1,C_k)$.
\end{lem}

\bp 
Choose an open neighborhood $U_i$ of $C_i$ such that
 \[U_i\cap g_{i+1}(U_{i+1})=\varnothing\] for each $1\le i \le k$ (again indices are taken modulo $k$).
There exists $M>0$ such that \[\phi_i(T_M)^{\pm1}(X\setminus U_i)\sse U_i.\]
Fix $t$ such that $t\in T_M$.
It suffices to verify the ping-pong condition of Definition~\ref{defn:attracting} (2)
for $\psi(t)$ with respect to the pair of disjoint open sets $(g_1(U_1),U_k)$.
We first note
\[
g_i\;\phi_i(t)(X\setminus U_i)\sse g_i(U_i)\sse X\setminus U_{i-1},\]
\[
\phi_i(t)^{-1}g_i^{-1}(X\setminus g_i(U_i))=
\phi_i(t)^{-1}(X\setminus U_i)\sse U_i\sse X\setminus g_{i+1}(U_{i+1}).\]
%
So we have that
\begin{align*}
\psi(t)(X\setminus U_k)
&=\prod_{i=1}^k g_i \;\phi_i(t) (X\setminus U_k)
\sse
\prod_{i=1}^{k-1} g_i \;\phi_i(t) (X\setminus U_{k-1})\\
&\sse\cdots\sse g_1 \phi_1(t)(X\setminus U_1)
\sse g_1(U_1),
\end{align*}
\begin{align*}
\psi(t)^{-1}(X\setminus g_1(U_1))
&=\prod_{i=k}^1 \phi_i(t)^{-1}g_i^{-1}
(X\setminus g_1(U_1))
\sse \prod_{i=k}^{2} \phi_i(t)^{-1}g_i^{-1}(X\setminus g_2(U_2))
\\
&\sse\cdots\sse   \phi_k(t)^{-1}g_k^{-1}(X\setminus g_k(U_k))
=  \phi_k(t)^{-1}(X\setminus U_k)
\sse U_k.
\end{align*}
Thus we establish the lemma.
\ep

\subsection{Projective and Discrete Settings}
We will mainly employ Lemma~\ref{lem:baumslag-gen} for the projective setting. For a subgroup $H\le G$, we let $N(H)$ denote the normalizer group of $H$.
For each matrix $g$, we let $\|g\|$ denote its $\ell^2$-norm.
Recall our notation that  $\bK=\bC$ or $\bK=\bR$.

\begin{lem}[Projective Baumslag Lemma]\label{lem:baumslag-psl}
Let $\mu\le\PSL_2(\bK)$ be a maximal abelian subgroup.
We define $\phi\co \mu\to\PSL_2(\bK)$ by
\[\phi(\nu)=g_1 \nu^{m_1} g_2\nu^{m_2}\cdots g_k \nu^{m_k},\]
such that $k\ge1$ and such that for each $i=1,\ldots,k$ we have
\be[(i)]
\item
 $g_i\in\PSL_2(\bK)$
 and
 $m_i\in\bZ\setminus\{0\}$;
\item
$\Fix\mu\cap g_i\Fix\mu=\varnothing$.
\ee
\be
\item
Then the following set is finite
for each $x\in\bK$:
\[\{\nu\in\mu\co \tr^2\circ \phi(\nu)=x\}.\]
\item
If $\mu$ is not purely elliptic, then we have
\[\lim_{\|\nu\|\to\infty, \nu\in\mu} |\tr\circ\phi(\nu)|=\infty.\]
\ee
\end{lem}

\bp
Let us first consider the case $\bK=\bC$.
We will prove (2) and then (1).
Since the condition $\|\nu\|\to\infty$ is invariant under conjugation, we may parametrize $\mu$ so that
\[
\mu(z)=
\begin{pmatrix} 1&z \\ 0& 1\end{pmatrix}
\text{ or }
\mu(z)=\begin{pmatrix} z&0 \\ 0& 1/z\end{pmatrix}.\]
In order to apply the Topological Baumslag Lemma (Lemma~\ref{lem:baumslag-gen}), we set 
 \[T=\bC, \quad \phi_i(z)=\mu(m_i z)\] in the former case and \[T=\bC^\times, \quad \phi_i(z)=\mu(z^{m_i})\] in the latter case.
 Further, set \[ C_i=\Fix \mu\sse S^2,\quad \psi(z)=\phi\circ\mu(z).\]
Note $g_i\Fix\mu\cap\Fix\mu=\varnothing$ on $S^2$.
So $\psi$ is doubly attracting to $(g_1\Fix\mu,\Fix\mu)$.
As the hyperbolic distance between $g_1\Fix\mu$ and $\Fix\mu$ is infinite, we have 
\[\lim_{|z|\to\infty} |\tr^2\circ\psi(z)|=\lim_{|z|\to\infty} |\tr^2\circ\phi\circ\mu(z)|=\infty.\]
So part (2) follows. For part (1) we note that
 $\tr^2\circ\phi\circ\mu(z)=x$ is an algebraic equation which is not an identity, and hence, the solution set is finite.
 
Now we set $\bK=\bR$. 
If $\mu$ is non-elliptic, then the proofs of (1) and (2) are almost identical to the previous paragraph. 
So we assume $\mu$ is elliptic
and prove (1). 

We use a complexification trick, embedding $\PSL_2(\bR)$ into $\PSL_2(\mathbb{C})$ and using Zariski density of $\SO(2)$ in $\bC^\times$. 
That is, we regard $\bH^2$ as a hemisphere properly embedded in $\bH^3$ with the upper half-space model. 
We may assume $\mu$ is the rotation group about the $z$--axis of $\bH^3$, so that
\[ \mu(t)= \begin{pmatrix} e^{it/2} & 0 \\ 0& e^{-it/2}\end{pmatrix}
\text{ and }
\phi\circ\mu(2t) = \psi(e^{it}),
\]
where we define
\[ \psi(z)=\prod_{i=1}^k g_i \begin{pmatrix}z & 0\\
0 & z^{-1}\end{pmatrix}^{m_i}.\]
By assumption, we have 
\[g_i.\{0,\infty\}\cap\{0,\infty\}=\varnothing.\]
We define an algebraic set
\[
Y=\{z\in\bC^\times\co \tr^2\circ\psi(z)=x\}.\]
If some $Y\cap S^1$ is infinite, then $Y=\bC^\times$ since every infinite set is Zariski dense in $\bC^\times$.
This is a contradiction to the case $\bK=\bC$.
\ep

\begin{rem}
Let us illustrate why one cannot weaken the condition (ii) to
\be[(i)]
\item[(ii)'] $\Fix\mu\ne g_i\Fix\mu$
\ee
in the above lemma.
Consider $\mu(t)=\exp(t)$ and $g_1,g_2$ are two distinct nontrivial parabolic elements fixing $\infty$.
In the upper half-space model, we see 
\[
[\mu(t),g_i\mu(t)g_i^{-1}]\]
is a translation along $\bR$, and so
\[f(t) = \left[[\mu(t),g_1\mu(t)g_1^{-1}],[\mu(t),g_2\mu(t)g_2^{-1}]\right]\]
is constantly the identity. Note that $f(t)$ satisfies the condition (ii)' and also all the conditions of the lemma except for (ii).
\end{rem}

A version of Baumslag's Lemma holds for other groups with north-south dynamics. 
Recall an isometric action of a non-virtually cyclic group $G$ on a metric space $(X,d)$ is \emph{acylindrical}\index{acylindrical (action)} if for all $r\ge0$ there exists $R,N\ge0$ such that whenever $x,y\in X$ satisfy $d(x,y)\ge R$
the set 
\[
\{g\in G\co d(x,gx)\le r\text{ and }d(y,gy)\le r\}\]
has cardinality at most $N$. If $G$ acts acylindrically on a hyperbolic space $X$, then each element $g\in G$ is either \emph{elliptic}\index{isometry!elliptic} (that is, $\form{g}$ has a bounded orbit) or \emph{loxodromic}\index{isometry!loxodromic} (that is, $g$ acts by translation on a quasi-geodesic axis); see~\cite{Bowditch2008}
and~\cite{Hamenstadt2008JEMS}.
Word-hyperbolic groups, mapping class groups and right-angled Artin groups admit acylindrical actions on hyperbolic spaces~\cite{Bowditch2008,KK2013b,DGO2017}; see also~\cite{BF2002GT}.
Using a local-to-global principle, we can strengthen the Topological Baumslag Lemma for acylindrically hyperbolic groups (i.e.\  groups that admit  acylindrical actions on hyperbolic spaces) to draw a geometric conclusion.

\begin{thm}[Discrete Baumslag Lemma]\label{thm:baumslag-acyl}
Let $G$ be a finitely generated group acting acylindrically on a hyperbolic space $Y$ and let
\[g_1,\ldots,g_k,u_1,\ldots,u_k\in G\]
such that each $u_i$ is loxodromic.
For each $1\le i\le k$, we assume
\[\form{u_i}\cap g_{i+1} \form{u_{i+1}} g_{i+1}^{-1}=\{1\},\]
where the indices are taken modulo $k$.
For $t=(t_1,\ldots,t_k)\in\bZ^k$, 
we define
\[\phi(t) =g_1 u_1^{t_1} g_2 u_2^{t_2} \cdots g_k u_k^{t_k}.\]
Then $\phi(t)$ is loxodromic whenever all $|t_i|$'s are sufficiently large.
\end{thm}
\bp 
The Gromov boundary $\partial Y$ is a complete metric space~\cite[1.8B]{Gromov1987}.
Each loxodromic $u\in G$ admits a quasi-geodesic axis $\ell_u\sse Y$,
and acts by north-south dynamics on $\partial Y$ with exactly two fixed points which comprise $\partial\ell_u$~\cite{Hamenstadt2008JEMS}.
For $E(u)=\stab(\partial \ell_u)$,
we have $[E(u):\form{u}]<\infty$.
See~\cite[Proposition 6]{BF2002GT} and \cite[Lemma 6.5 and Theorem 6.8]{DGO2017}
for relevant details.
So for each $1\le i\le k$, we have
\[E(u_i)\ne g_{i+1}E(u_{i+1})g_{i+1}^{-1},\quad
\partial\ell_{u_{i}}\cap g_{i+1}\partial\ell_{u_{i+1}}=\varnothing.\]

Let us fix $y\in Y$. For each $t\in \bZ^k$, we naturally regard $\phi(t)$ as a word of length 
\[k+|t_1|+\cdots+|t_k|\]
and define a bi-infinite coarse path
\[\gamma\co \bZ\to Y\]
which is the orbit of the bi-infinite words $\cdots\phi(t)\phi(t)\cdots$ based at $y$.

It follows from \cite[III.H.1.13]{BH1999} 
or \cite[Theorem 21, Chapter 5]{GH1990} that for each $\ell \geq 1$, there exists (large) $K$ such that every $(K,\ell)$-local quasi-geodesic in $Y$
 is a global quasi-geodesic. We shall show that $\gamma$ is such a $(K,\ell)$-local quasi-geodesic whenever each $|t_i|$ is sufficiently large;  hence $\phi(t)$ is loxodromic.

It remains to choose each $|t_i|$  large enough so that $\gamma$ is a $(K,\ell)$-local quasi-geodesic with appropriate $K,\ell$.
First note that 

\begin{enumerate}
\item[(1)] There exists $\ell_1$ such that for all $i$ and $n \in \bZ$, the word $u_i^{n}$ is an $(\ell_1,\ell_1)$--quasi-isometric embedding of $[0,|n|]$. This follows from the hypothesis that $u_i$'s are loxodromic.
\item[(2)]  There exists $\ell_2$ such that for all $i$ and $n_1, n_2 \in \bZ$, the word $u_i^{n_1}g_{i+1}u_{i+1}^{n_2}$ is an $(\ell_2,\ell_2)$--quasi-isometric embedding of $[0,|n_1|+|n_2|+1]$. This follows from Item (1) above and the hypothesis that $\partial\ell_{u_{i}}\cap g_{i+1}\partial\ell_{u_{i+1}}=\varnothing$ (see for instance Corollary 1.37 of \cite{mahan-sardar}).
\end{enumerate}

Choose $\ell=\max(\ell_1, \ell_2)$ and let $K$ be large enough so that every $(K,\ell)$-local quasi-geodesic is a global  quasi-geodesic. We now choose
 each $|t_i| \geq K$. Then each subpath  $\gamma (N, N+K)$ of $\gamma$ is of type (1) or (2) above and is therefore an   $(\ell,\ell)$-quasi-geodesic.
 Hence $\gamma$ is a global quasigeodesic and we are done.
\ep
\begin{rem}
In the above proof, we can apply Lemmas \ref{lem:doubly} and~\ref{lem:baumslag-gen}
for the setting \[T=\bZ^k,\quad X=\partial Y,\quad \phi_i(t)=u_i^{t_i},\quad  \phi(t)=\prod_{i=1}^k g_i\phi_i(t),\]
to conclude that $\phi\co\bZ^k\to\Homeo(\partial Y)$ is doubly attracting as a map.
However, this does not directly imply that the individual isometry $\phi(t)$ is loxodromic whenever $|t_i|$'s are large.
\end{rem}


\begin{exmp}\label{exmp:tor-free-hyp}
Suppose $G$ is a torsion-free word-hyperbolic group.
Define $\phi\co\bZ^k\to G$ as in Theorem~\ref{thm:baumslag-acyl}, where $g_i,u_i\in G$ satisfy
\[[u_i,g_{i+1} u_{i+1}g_{i+1}^{-1}]\ne1.\]
Then $\phi(t)\ne1$ whenever all $|t_i|$'s are sufficiently large by Theorem~\ref{thm:baumslag-acyl}. 
If \[u_1=\cdots=u_k=u,\] the hypothesis can be weakened to
\[[u,g_i]\ne1\] for each $i$.
So we have a direct generalization of the original Baumslag's Lemma.
See also \cite[Lemma 5.4]{GW2010JT}, \cite[Lemma 3.6]{Kim2010} and  \cite[Lemma 3.5]{BG2010JAM} for other special  cases of Theorem~\ref{thm:baumslag-acyl}.
\end{exmp}

\begin{cor}\label{cor:mcg}
Let $S$ be a closed oriented hyperbolic surface
and $\Mod(S)$ denote its mapping class group.
Define $\phi\co\bZ^k\to \Mod(S)$ as in Theorem~\ref{thm:baumslag-acyl}, where $g_i,u_i\in G$ satisfy that
\[[u_i,g_{i+1} u_{i+1}g_{i+1}^{-1}]\ne1,\]
and that each $u_i$ is pseudo-Anosov.
Then  $\phi(t)$ is pseudo-Anosov whenever all $|t_i|$'s are sufficiently large. 
\end{cor}
\bp
Since $G$ acts on the curve graph acylindrically~\cite{Bowditch2008}, we can apply Theorem~\ref{thm:baumslag-acyl}.
\ep

\begin{rem}
If $u_1=\cdots=u_k$ in the above corollary, then we can also weaken the hypothesis to 
\[
[u,g_i]\ne1\]
for each $i$, since mapping class groups do not admit Baumslag--Solitar relations that are not commutators~\cite[Theorem 8.2]{KL2007}.
\end{rem}

The following is an immediate consequence of Corollary~\ref{cor:mcg}.
\begin{prop}
Let $G$ be the mapping class group of a closed orientable hyperbolic surface, and let $\form{f}\le G$ be a maximal cyclic pseudo-Anosov subgroup. 
Define $\phi_n\co G\ast_{\form{f}} G\to G$
by extending the map $\Id_G$ on the left factor and 
the map
\[g\mapsto f^n g f^{-n}\] on the right factor.
Then the family $\{\phi_n\}$ is stably injective; that is,
we have
\[
\bigcap_{M\ge1}\bigcup_{n\ge M}\ker\phi_n=\{1\}.\]
\end{prop}

\section{Splittable Fuchsian Groups}\label{sec:fuchs}
In this section, we study deformations of (possibly indiscrete) faithful representations of Fuchsian groups such that almost all points on the deformations are still faithful.
Let $L$ be a \emph{splittable}\index{Fuchsian group!splittable} Fuchsian group, which includes all Fuchsian groups with Euler characteristic at most $-1$; see Definition~\ref{defn:split}.
We will prove that an arbitrarily small deformation of a given representation can be chosen so that the new trace spectrum is almost disjoint from the original one (Theorem~\ref{thm:fuchs-flex}). Then we show $X_{\mathrm{proj}}(L)$ contains at least one indiscrete representation (Lemma~\ref{lem:split-indiscrete}). Moreover, if an open set $U$ contains at least one indiscrete representation in $X_{\mathrm{proj}}(L)$, then $U$ contains uncountably many pairwise inequivalent indiscrete representation in $X_{\mathrm{proj}}(L)$ (Theorem~\ref{thm:uncountable}).

The main engine behind this study is the Pulling-Apart Lemma (Lemmas~\ref{lem:free} and~\ref{lem:pull-apart-psl}). Roughly, this lemma means that ``a generic deformation of a representation remembers obvious group relations only''. 
The two ingredients of this lemma are put together here. 
First, we use in an essential way, the algebro-geometric structure of representation varieties to establish that a generic point is faithful {\it provided} no word (varying according to suitably chosen parameters) is  identically trivial. The Topological Baumslag Lemma and the complexification trick from the previous section furnish this sufficient condition. While building faithful representations of free products with amalgamation $G = A\ast_C B$, or HNN extensions $A\ast_\phi$ from faithful representations of $A, B$, we will use genericity or {\it Baire category argument}, which is the second ingredient. This argument makes precise the notion of a property satisfied by a generic point in a real or complex algebraic variety. 

Throughout this section, we will let $\bK=\bC$ or $\bK=\bR$.

\subsection{Very General Points, Abundance and Stable Injectivity}\label{sec:baire}
Our techniques involve both the  topological category as well as that of (real or complex)
algebraic varieties. Thus a choice of terminology needs to be made, which we set out to do here.

First we set up the terminology in the topological category.
Recall that a $G_\delta$ set in a topological space $X$ is the countable intersection of open subsets.
We shall refer to a property as being satisfied by a \emph{very general}\index{very general!in a topological space} point in $X$
if it is satisfied by a point in a (suitably chosen) dense $G_\delta$ subset of $X$.

The corresponding notion in algebraic geometry is more restrictive.
We will regard linear algebraic groups as affine algebraic sets
and equip them with the Zariski topology.
We shall refer to a property as being satisfied by a \emph{very general}\index{very general!in an algebraic set} point in an algebraic set $X$ if it is satisfied by a point in the complement of a countable union of subsets \[\{X_1,X_2,\ldots\}\] of $X$ such that each $X_i$ is a proper algebraic subset of an irreducible component of $X$.

We gather together some facts for use later. All algebraic sets will be defined over $\bR$ or $\mathbb C$.
Note that $\PSL_2(\bR)$ and $\PSL_2(\bC)$ are irreducible algebraic groups.
The following is a standard fact; see ~\cite{borel83,bourabee} or ~\cite[Lemma 2.4]{BGGT} for instance.

\begin{lem}\label{lem:baire}
Let $X$ be an algebraic set, and let
 \[X_1\sse X_2\sse \cdots\] be a countable chain of proper algebraic subsets of $X$. Then the set
 \[\DD=X\setminus\bigcup_{i=1}^\infty X_i\]
contains very general points of some irreducible component of $X$.
Furthermore, if $\dim X>0$, then
$\DD$ is  uncountable.
\end{lem}

\begin{proof} 
Each $X_i$ is closed in the Zariski topology, 
and hence is either nowhere dense in $X$ or contains an irreducible component of $X$. 
Since $X$ has finitely many components
and each $X_i$ is a proper subset,
there exists an irreducible component $V$ of $X$ such that $V\cap X_i$ is proper in $V$ for each $i$.
By definition, \[Y=V\setminus\bigcup_i X_i\] is very general in $V$.
For the second part, note
$Y$ contains a dense $G_\delta$ set.\end{proof}

We indicate here a general method of extracting a large (uncountable)
family of faithful representations by applying
Lemma \ref{lem:baire}. The main  idea in this comes from translating a concept
from the (discrete) context of limit groups into the (algebraic geometry)
context of representations into algebraic groups.

\bd[\cite{Sela2001PIHES}]\label{defn:stab-inj}
We say a sequence of group homomorphisms \[\{f_n\}_{n\ge1}\sse\Hom(G,H_n)\] is \emph{stably injective}\index{stably injective} if 
for each $g\in G\setminus 1$ there exists $n_0>0$ such that $f_n(g)\ne1$ for all $n\ge n_0$.
\ed

For terminological convenience, we will also say a sequence of maps is \emph{stably injective} if the sequence has a stably injective subsequence.  As a warm-up, we now give a simple illustration of the Baire category argument 
in connection with stably injective maps.

\begin{lem}\label{lem:baire-simple}
Let $L$ be a  finitely generated group, and ${G}$ be an algebraic group.
If there exists a stably injective sequence of representations \[\{f_n\co L\to{G}\}_{n\in\bN}\] 
which contains infinitely many distinct representations, 
then faithful representations are very general in some irreducible component of $\Hom(L,G)$.
\end{lem}

\bp
Write $L=\form{S}$. 
Note that \[X=\Hom(L,{G})\sse {G}^S\] is an algebraic set (by the Hilbert Basis Theorem).
Since $X$ is infinite, it has a positive dimension.

For each finite subset $Q\sse L\setminus 1$, 
 we define an algebraic set
 \[Y(Q)=\{ \rho\in X \co \rho(g)=1\text{ for some }g\in Q\}.\]
The stable injectivity implies that $Y(Q)$ is proper for all $n\in\bN$.
We can now apply Lemma~\ref{lem:baire} to the set of faithful representations, which can be written as
 \[\DD=X\setminus \bigcup_{Q\sse L,\text{ finite} } Y(Q).\qedhere\]
\ep

\bd\label{defn:limit group}
A finitely generated group $L$ is a \emph{limit group}\index{limit group}
if there exists a stably injective sequence \[\{\rho_n\co L\to F\}_{n\in\bN}\] for some fixed nonabelian free group $F$.\ed
Limit groups are torsion-free and finitely presented~\cite{BF2009book}. 
Surface groups and free groups are limit groups. 
An obvious consequence of Lemma~\ref{lem:baire-simple} 
is that if an algebraic group ${G}$ contains a nonabelian free group $F$, then ${G}$ contains an isomorphic copy of every limit group; see \cite{BF2009book,Wilton2009solutions,BG2010JAM}.

The following group--theoretic lemma will come handy when finding stably injective sequences.
For a group $L$, a subset $P\sse L$ and an integer $n>0$, 
we denote
\[ {P^n}={\{g^n\co g\in P\}}.\]
\begin{lem}\label{lem:sep}
For a finitely generated residually finite group $L$ and a finite subset $P \sse L$,
the following sequence of the quotient maps is stably injective:
\[\{\rho_n\co L\to L/\fform{P^{n!}}\co n\ge1\}.\]
\end{lem}
\bp
If $N$ is a normal subgroup of $L$ with index $k$ and $n\geq k$, then $\fform{P^{n!}}\le N$.
So for any sequence $n_i \to \infty$, we have 
\[\bigcap_{i\ge1}\fform{P^{n_i!}}=\{1\},\] or equivalently $\rho_n$ is stably injective.
\ep

\subsection{Pulling-Apart Lemma}
Let us now describe a more general method to continuously ``pull-apart'' subgroups to become (amalgamated) free products of those subgroups. Similar ideas can be found in the literature, e.g.~\cite{BF2009book} and~\cite{Mann-hb}. McCammond and Sulway coined the term ``pull-apart'' in a different (but loosely relevant) sense from ours; see~\cite{MS2017IM}. 

Recall our convention that $\bK=\bR$ or $\bK=\bC$.
For a group $C\le \PSL_2(\bK)$, we let $N(C)$ denote its normalizer group in ${\PSL_2(\bK)}$. 
Suppose $C\le {\PSL_2(\bK)}$ is nontrivial and abelian,
so that the centralizer $Z(C)$ is the maximal abelian group containing $C$; see Lemma~\ref{lem:malnormal}.
In case $K = \bR$, if $C$ is hyperbolic, then $[N(C):Z(C)]=2$, if elliptic $N(C) = Z(C)$, and if $C$ is parabolic, $[N(C) : Z(C)] = 1$.
In case $K = \bC$, if $C$ is loxodromic, then $[N(C) : Z(C)] = 2$, and if $C$ is parabolic,
$[N(C) : Z(C)] = 1$.

Using the multiplicative notation, we let $\form{s}\cong\bZ$.
For $g,t\in{\PSL_2(\bK)}$, we define homomorphisms
$\Inn(t)\co \PSL_2(\bK)\to \PSL_2(\bK)$ and $\delta_t\co \bZ\to\PSL_2(\bK)$ by
\[
\Inn(t)(g)=g^t,\quad
\delta_t(s)=t^s.\] 
If $L$ is a group and $V\sse L$ is a set, we define
\[
V^L=\{v^g\co v\in V, g\in L\}.
\]

The following can be regarded as a warm-up case for the Amalgamated Pulling-Apart Lemma.

\begin{lem}[Free Product Pulling-Apart Lemma]\label{lem:free}
\be
\item
Let $A$ and $B$ be countable subgroups of ${\PSL_2(\bK)}$,
and let $W\sse\bK$ be a countable set.
Then for a very general $\nu\in{\PSL_2(\bK)}$, 
the group $L_\nu=\form{A,B^\nu}$ is isomorphic to $A\ast B$
and furthermore, 
\[\tr^2\left(L_\nu\setminus (A\cup B^\nu)^{L_\nu}\right)\cap W=\varnothing.\]
\item
Let $A$ be a countable subgroup of ${\PSL_2(\bK)}$,
and let $W\sse\bK$ be a countable set.
Then for a very general $\nu\in{\PSL_2(\bK)}$, 
the group $L_\nu=\form{A,\nu}$ is isomorphic to $A\ast \bZ$
and furthermore, 
\[\tr^2\left(L_\nu\setminus A^{L_\nu}\right)\cap W=\varnothing.\]
\ee
\end{lem}

\bp
In Case (1), we let \[L=A\ast B,\quad V=A\cup B\sse \PSL_2(\bK).\]
In Case (2), we let \[L=A\ast\form{s},\quad V=A\sse \PSL_2(\bK)\]
for $\form{s}\cong\bZ$.
Note we allow the possibility that $A=B$.

For each $h\in \form{V}\setminus\{1\}$, there exists a hyperbolic element $c=c(h)\in\PSL_2(\bK)$ such that $\Fix c\cap h\Fix c=\varnothing$. By Lemma~\ref{lem:parabolic-test} (2), we see that 
\[\tr^2[c,hch^{-1}]\ne 4.\]
Since $\form{V}\setminus\{1\}$ is countable
and since $\tr^2[c,hch^{-1}]=4$ is an algebraic equation of $c$ and $h$,
there exists some hyperbolic element $c_0\in\PSL_2(\bK)$ 
such that whenever $h\in\form{V}\setminus\{1\}$ 
we have that
\[\tr^2[c_0,hc_0h^{-1}]\ne 4.\]
Let $\mu$ be the hyperbolic maximal abelian group containing $c_0$. 
Then for all $h\in\form{V}\setminus\{1\}$, we have \[h\Fix\mu\cap \Fix\mu=\varnothing.\]

We will define a natural map below
\[\rho_\nu\co L\to L_\nu.\]
For $T\sse \bK$ and $Q\sse L\setminus V^{L}$, we then define
\[\CC(Q,T):=\{\nu\in \PSL_2(\bK)
\co \tr^2\circ\rho_\nu(Q)\cap T\ne\varnothing\}.\]
The crucial step of the proof is the following claim.

\begin{claim*}
If $T\sse \bK$ 
and $Q\sse L\setminus V^{L}$ are finite subsets,
then $\CC(Q,T)\ne  \PSL_2(\bK)$.
\end{claim*}

Let us prove the claim in the cases (1) and (2) separately, after defining $\rho_\nu$ precisely.

(1) For each $\nu\in{\PSL_2(\bK)}$, we define $\rho_\nu$ by the following diagram:
\[
\xymatrix{
& A\ar[rd]\ar[rrrd]^{\Id_A}\\
1\ar@{^(->}[ru] \ar@{^(->}[rd]& & 
L=A\ast B\ar[rr]^<<<<<<<{\rho_\nu} && {\PSL_2(\bK)}\\
& B\ar[ru]\ar[rrru]_{\Inn(\nu)\circ\Id_B}
}\]

Each $g\in Q$ can be written as
\[
g=\prod_{i=1}^k a_i b_i\]
for some $k\ge1, a_i\in A\setminus1$ and $b_i\in B\setminus1$
up to conjugacy. It follows that, after reducing cyclically,
\[
\rho_\nu(g)=\prod_{i=1}^\ell g_i\nu^{m_i}\]
for some $g_i\in (A\cup B)\setminus1$ and $m_i=\pm1$.
By Lemma~\ref{lem:parabolic-test} and Projective Baumslag Lemma (Lemma~\ref{lem:baumslag-psl}),
whenever $\nu\in \mu$ and $\|\nu\|\gg0$ we have
\[\tr^2\circ\rho_\nu(Q)\cap T=\varnothing.\]

(2) For each $\nu\in{\PSL_2(\bK)}$, we define $\rho_\nu$ as
\[
\xymatrix{
& A\ar[rd]\ar[rrrd]^{\Id_A}\\
1\ar@{^(->}[ru] \ar@{^(->}[rd]& & L=A\ast \form{s}\ar[rr]^<<<<<<<{\rho_\nu} && {\PSL_2(\bK)}\\
& \form{s}\ar[ru]\ar[rrru]_{s\mapsto \nu}
}\]

Let $g\in Q$. 
If $g=s^m$ for some $m\ne 0$ (up to conjugacy),
then we have
\[\rho_\nu(g) = \nu^m.\]
Otherwise, we have, after reducing cyclically,
\[
\rho_\nu(g)=\prod_{i=1}^\ell g_i \nu^{m_i}\]
for some $k\ge1, g_i\in A\setminus1$ and $m_i\in\bZ\setminus0$
up to conjugacy. 
The rest of the proof is identical to (1).

From the claim above and Lemma~\ref{lem:baire}, we see that 
 the following set is very general:
\[\PSL_2(\bK)\setminus \CC(L\setminus V^L,W\cup\{4\}).\]
Note that if $1\ne g\in V^L$, then $\rho_\nu(g)\ne1$.
Hence we can further require that $\rho_\nu$ is faithful.
\ep

For $C\le\PSL_2(\bK)$, recall our notation for the centralizer:
\[ Z(C) = \{ g\in\PSL_2(\bK)\co [g,c]=1\text{ for all }c\in C\}.\]

\begin{lem}[Amalgamted Pulling-Apart Lemma]\label{lem:pull-apart-psl}
For a countable subgroup $L\le \PSL_2(\bK)$ and a countable set $W\sse\bK$, assume one of the following.
\be
\item
For some subgroups $A,B,C\le L$, we have that $L=\form{A,B}$
and that $C$ is a malnormal nontrivial abelian subgroup of $A$ and  also of $B$. 
For each $\nu\in Z(C)$ we let $\rho_\nu$ be the map uniquely determined by the following commutative diagram
\[\xymatrix{
& A\ar[rd]\ar[rrrd]^{\Id_A}\\
C\ar@{^(->}[ru] \ar@{^(->}[rd]& & L^*=A\ast_C B\ar[rr]^<<<<<<<{\rho_\nu} && {\PSL_2(\bK)}\\
& B\ar[ru]\ar[rrru]_{\Inn(\nu)\circ\Id_B}}\]
We further assume $V=A\cup B$ is parabolic-free.
\item
For some subgroups $A,C\le L$ and an element $s\in L$,
we have that $L=\form{A,s}$ and that $C$ and $C^s$ are malnormal nontrivial abelian subgroups of $A$. Furthermore, we assume that 
either 
\begin{itemize}
\item $s\in Z(C)$, or
\item $C$ and $C^s$ are not conjugate in $A$.
\end{itemize}
For each $\nu\in Z(C)$ we let $\rho_\nu$ be the map uniquely determined by the following commutative diagram
\[
\xymatrix{
& A\ar[rd]\ar[rrrd]^{\Id_A}\\
C\ast C^s\ar[ru] \ar[rd]& & L^*=A\ast_{\Inn(s)\co C\to C^s}\ar[rr]^<<<<<<<{\rho_\nu} && {\PSL_2(\bK)}\\
& C\ast\form{ s^*}\ar[ru]\ar[rrru]_{  s^*\mapsto \nu s}
}\]
Here $s^*$ denotes the stable generator of $L^*$,
and the map $C\ast C^s\to C\ast\form{s^*}$ is defined as the extension of $g\mapsto g$ and $s^{-1}gs\mapsto (s^*)^{-1}gs^*$ for $g\in C$.
We further assume $V=A$ is parabolic-free.
\ee
Under the assumption (1) or (2), for all but countably many $\nu\in Z(C)$ the map $\rho_\nu$ is faithful  and  \[\tr^2 \circ\rho_\nu\left(L^*\setminus V^{ L^*}\right)\cap W=\varnothing.\]
\end{lem}
\bp
We let $\mu=Z(C)$, which is a non-parabolic maximal abelian group.
Also, put $V=A\cup B$ in (1), and $V=A$ in (2), as subsets of $L^*$. As in the proof of Lemma~\ref{lem:free}, it suffices to show the following claim.
\begin{claim*}[Key Claim]
If $T\sse \bK$ 
and $Q\sse L^*\setminus V^{ L^*}$ are finite subsets,
then we have
\[\CC(Q,T):=\{\nu\in \mu\co \tr^2\circ\rho_\nu(Q)\cap T\ne\varnothing\}\ne \mu.\]
\end{claim*}
From the Key Claim,
we see that the following set is very general, as desired:
\[\mu\setminus \CC\left( L^*\setminus V^{L^*},  \{4\}\cup W\right).\]

Let us now prove the Key Claim through a sequence of steps.

\begin{claim}\label{claim:nmu}
\be[(i)]
\item
We have $N(\mu)\cap V=C\le \mu$, as subsets of $\PSL_2(\bK)$.
\item
For $h\in V\setminus C$, we have $\Fix\mu\cap h\Fix\mu=\varnothing$.
\item
In Case (2), if $s\not\in Z(C)$, then we have
 $\Fix\mu\cap s\Fix\mu=\varnothing$.
\ee
\end{claim}

If $C$ is not hyperbolic in (i), then $N(\mu)\cap V = \mu\cap V=C$ by Proposition~\ref{prop:maxabel}. So we may assume $C$ is hyperbolic and in particular, torsion-free. Then for each $c\in C$ and $h\in N(\mu)\cap V$, we have $hch^{-1}=c^{\pm1}$. Hence $h$ normalizes $C$ and this proves part (i) of Claim~\ref{claim:nmu}. For part (ii), let us fix $c\in C\setminus1$. Note that 
\[\tr^2[c,hch^{-1}]\ne4\] by part (i)
and our hypothesis that $V$ is parabolic-free. 
By Lemma~\ref{lem:parabolic-test} the proof of part (ii) is complete.
For part (iii), let us fix $c\in C\setminus 1$.
Since $C^s\ne C$, we see $s\not\in N(\mu)$
and hence, 
\[
[c,scs^{-1}]\ne 1.\]
As $A$ is parabolic-free and $\form{C,C^s}\le A$, we have 
\[ \tr^2[c,scs^{-1}]\ne 4.\]
This completes the proof of Claim~\ref{claim:nmu}.

Let us now consider the cases (1) and (2) of the lemma separately.

\noindent\textbf{Case (1)} $ L^*=A\ast_C B$.

Let $g\in Q$. Up to conjugacy, we can write \[\rho_\nu(g) =\prod_{i=1}^k g_i \nu^{m_i}\] 
for some $k\ge1, m_i=\pm1$ and
\[ g_i\in A\cup B\setminus C=A\cup B\setminus N(\mu).\]
Lemma~\ref{lem:parabolic-test} implies that
\[[c,g_icg_i^{-1}]\in A\cup B\setminus 1\]
for all $i$ and $c\in C\setminus1$.
If $\Fix\mu\cap g_i\Fix\mu\ne\varnothing$, then $\mu$ must be hyperbolic and 
\[\tr^2[c,g_icg_i^{-1}]=4.\]
This contradicts the parabolic-free hypothesis.
Lemma~\ref{lem:baumslag-psl} implies the Key Claim.

\noindent\textbf{Case (2)-1.}  $ L^*=A\ast_{C}$ and $C\le Z(s)$.

In this case, we have \[s\in  \left((L\setminus A)\cup C\right)\cap\mu.\]
If $g=c(s^*)^m\in Q$
 for some $c\in C$ and $m\in\bZ\setminus0$,
then we have \[\tr^2\circ\rho_\nu(g)=\tr^2(c\nu^ms^m)\in\tr^2(\mu)\]
for $\nu\in\mu$.
Then $\tr^2\circ\rho_\nu(g)\not\in T$ whenever $\|\nu\|\gg0$.

So up to conjugacy, we consider
\[g=\prod_{i=1}^k g_i (s^*)^{m_i}\in Q\] for some $k\ge1$, some $g_i\in A\setminus C$ and some $m_i\in \bZ\setminus0$.
We have
\[\rho_\nu(g) =\prod_{i=1}^k \left(g_i  s^{m_i}\right)\nu^{m_i}.\]
Note that \[g_i s^{m_i}\Fix\mu= g_i\Fix\mu.\]
So, by Claim~\ref{claim:nmu} and Lemma~\ref{lem:baumslag-psl} again, 
the Key Claim follows.

\noindent\textbf{Case (2)-2.}  $ L^*=A\ast_{\Inn(s)\co C\to C^s}$,
and the subgroups $C$ and $C^s$ are not conjugate in $A$.

If $g=(s^*)^m\in Q$, then 
\[\rho_\nu(g) = (\nu s)^m\]
Claim~\ref{claim:nmu} (iii) implies that 
 $\tr^2\circ\rho_\nu(g)\not\in T$ whenever $\|\nu\|\gg0$.

So we may assume  $g\in  Q$ satisfies
\[ g = \prod_{i=1}^k g_i ( s^*)^{m_i} \]
for some $k\ge 1, m_i\in \bZ\setminus\{0\}$ and $g_i\in V\setminus\{1\}$.
Using the normal form theorem for HNN-extensions, we may assume that $g$ has no subwords of the form
$ s^*c( s^*)^{-1}$ or $( s^*)^{-1}c' s^*$ for $c\in C$ or $c'\in C^s$. We can write
\[\rho_\nu(g) =\prod_{i=1}^k g_i ( \nu s)^{m_i}= \prod_{i=1}^\ell h_j\nu^{n_j}\]
for some 
\[k,\ell\ge1,\quad g_i\in A\setminus1, \quad h_j\in L\setminus1,
\quad m_i\in\bZ\setminus0,\ n_j = \pm 1.\]
Suppose $\Fix\mu\cap h_j\Fix\mu\ne\varnothing$ for some $j$.
There are five cases. 
\be[(I)]
\item
$h_j=s^{\pm1}$.
\item
$\cdots \nu (s g_i)\nu\cdots=\cdots\nu h_j\nu\cdots$.
\item
$\cdots \nu^{-1} ( g_i s^{-1} )\nu^{-1}\cdots=\cdots\nu^{-1} h_j\nu^{-1}\cdots$.
\item
$\cdots \nu^{-1} (g_i)\nu\cdots=\cdots\nu^{-1} (h_j)\nu\cdots$.
\item
$\cdots \nu (s g_i s^{-1})\nu^{-1}\cdots=\cdots\nu (h_j)\nu^{-1}\cdots$.
\ee
By Claim~\ref{claim:nmu} (iii), the case (I) is excluded.
If (II) occurs, then we have $s g_i=h_j$.
If $s g_i\in N(\mu)$, then
\[
C^{s g_i}= C,\]
which contradicts the assumption that $C$ and $C^s$ are not conjugate in $A$. 
So we have a nontrivial parabolic element
\[[ c, g_i^{-1}s^{-1} c s g_i]\]
for each $c\in C\setminus 1$. This contradicts the parabolic-free assumption on $A$, since $s^{-1}cs\in A$.
Suppose (III) occurs and $h_j=g_i s^{-1}$.
As in the case (II), we have a  nontrivial parabolic element 
\[[ c, g_is^{-1} c s g_i^{-1}]\in A\] for each $c\in C\setminus 1$, which is a contradiction.

Suppose we have (IV).
Assuming the minimality of $k$, we have $h_j=g_i\in A\setminus C$;
this is a contradiction to Claim~\ref{claim:nmu} (ii).
In (V), we have $ g_i=h_j^s\in A\setminus C^s$.
If $h_j\in N(\mu)$,
then we would have
\[
C^{s g_i}= C^{h_js} = C^s,\]
which contradicts the hypothesis that $C^s$ is malnormal in $A$. So for each $c\in C$, we have a nontrivial parabolic element 
\[
[c,h_jch_j^{-1}]^s=[ c, sg_is^{-1} c s g_i^{-1}s^{-1}]^s
=[ c^s, g_i c^sg_i^{-1}]\in A.\] 
This is a contradiction, and the Key Claim is proved.
\ep

\subsection{Almost Faithful Paths}\label{sec:almost faithful}
Let $L$ be a finitely generated group and let $\rho_0\in X_{\mathrm{proj}}$.
We say a path $\{\rho_t\}_{t\in \bR}$ in $\Hom(L,\PSL_2(\bR))$ is an \emph{almost faithful path}\index{almost faithful} if $\rho_t\in X_{\mathrm{proj}}(L)$ for all but countably many $t$. 
In this section, we will describe a method to find an almost faithful path starting from an arbitrary $\rho_0\in X_{\mathrm{proj}}(L)$.
In general, it is not be possible to require that such a path consists of only faithful representations (Proposition~\ref{prop:faithful-arc}).

\bd\label{defn:split}
Let $S$ be a finite type hyperbolic 2--orbifold. We say a simple closed curve $\alpha$ is \emph{splitting}\index{splitting (curve)} if the Euler characteristic of each component of $S\setminus\alpha$ is negative. If there exists at least one splitting curve, we say $S$ (and also $\pi_1^{\mathrm{orb}}(S)$) is \emph{splittable}\index{$2$--orbifold!splittable}.
\ed
For instance, every non-separating essential simple closed curve is splitting.

\begin{lem}\label{lem:fuchs-analytic}
Let $L=\pi_1^{\mathrm{orb}}(S)$ for some finite type hyperbolic 2--orbifold $S$, 
and let $\alpha$ be a splitting curve on $S$.
If $\rho_0\in X_{\mathrm{proj}}(L)$, then there exists an analytic path 
\[\{\rho_t\co t\in\bR\}\sse\Hom(L,{\PSL_2(\bR)})\] satisfying the following:
\begin{quote}
for each $b\in L$, the analytic function $\tr^2\circ\rho_t(b)$ is constant on $t$ if and only if the free homotopy classes of $[\alpha]$ and $b$ are disjointly realized.
\end{quote}
\end{lem} 

Note we allow the representation $\rho_0$ to be discrete or indiscrete.
\bp[Proof of Lemma~\ref{lem:fuchs-analytic}]
For $C=\form{[\alpha]}\le L$, we have two cases.

\be[(i)]
\item
The curve $\alpha$ is separating on $S$ and \[L=\form{A,B}\cong A\ast_C B\] for some $A,B\le L$. In this case, we let $V=A\cup B$.
\item
The curve $\alpha$ is non-separating on $S$ and \[L=\form{A,s}\cong A\ast_{C\to s^{-1}Cs}\] for some $A,C\le L$ and $s\in L$. In this case, we let $V=A$.
\ee
For $b\in L$, we have $b\in V^L$ if and only if $[\alpha]$ and $b$ can be disjointly realized. Hence the conclusion follows from Lemma~\ref{lem:pull-apart-psl}.
\ep

Splittable Fuchsian groups are characterized as follows.

\begin{lem}\label{lem:split}
Let $S$ be a finite type hyperbolic $2$--orbifold.
Then $S$ is splittable if and only if 
the signature of $S$ is not one of the following types:
\be[(i)]
\item $(0;p,q,r)$ for some $2\le p,q,r\le \infty$ satisfying $1/p+1/q+1/r<1$;
\item $(0;2,2,2,p)$ for some $3\le p\le\infty$;
\item $(0;2,2,2,2,2)$.
\ee
\end{lem}
\bp
When an orbifold $S$ of the types (i) through (iii) is cut along an arbitrary non-null-homotopic simple closed curve,
then at least one component has the type $(0;p,\infty)$ or $(0;2,2,\infty)$. This implies $S$ is non-splittable. 

In order to show the converse, 
let us assume $S$ is a finite type hyperbolic $2$--orbifold.
If the genus of $S$ is zero and $S$ is not one of (i) through (iii), then one can find an essential separating simple closed curve $\alpha$ such that neither component of $S\setminus\alpha$ has the type $(0;p,\infty)$, $p\geq 2$ or $(0;2,2,\infty)$. So $S$ is splittable. If the genus of $S$ is at least one, we choose a non-separating curve $\alpha$ so that $\chi(S\setminus\alpha)=\chi(S)<0$. 
\ep

A set of essential simple closed curves on a 2--orbifold is called \emph{filling} if the complement of those curves is the union of open disks with at most one puncture or at most one cone point.

\begin{lem}\label{lem:split-scc}
If $S$ is a finite type hyperbolic $2$--orbifold that is splittable,
then there exists a filling set of splitting simple closed curves.
\end{lem}
\bp
Let $g$ be the genus of $S$. If $g\ge1$, then one can choose a filling set $C$ consisting of non-separating (hence splitting) simple closed curves. If $g=0$, then we first pick a splitting simple closed curve $\gamma$. Note that the existence of a splitting curve forces the total number of punctures and cone points of $S$ to be at least four. Note further that in this case, the pure mapping class group of $S$ (i.e. the subgroup of the mapping class group which fixes the cone points and the punctures pointwise) contains a nontrivial pseudo-Anosov mapping class $\psi$. It is a standard fact then that the isotopy classes of curves $\{\psi^n(\gamma)\}_{n\in\bZ}$ are all distinct and that their union fills $S$ (see~\cite{FM2012}, for instance), which furnishes a finite, filling collection of such splitting curves. For each $n$, we have that the orbifold types of the components of $S\setminus\gamma$ and of $S\setminus\psi^n(\gamma)$ are the same. This completes the proof.
\ep

For a group $L$, we used $T_L$ to denote the set of torsion elements in $L$. 
Recall  that two projective representations $\rho,\lambda$ of $L$ are \emph{tracially disjoint}\index{tracially disjoint} if
\[
\tr^2\circ\rho(L\setminus T_L)\cap
\tr^2\circ\lambda(L\setminus T_L)=\varnothing,\]
where here we remind the reader that $T_L$ denotes the subset of torsion elements.
We can make an arbitrarily small faithful tracially disjoint deformation of each faithful parabolic-free representation of a splittable Fuchsian group:

\begin{thm}\label{thm:fuchs-flex}
Let $L$ be a  splittable Fuchsian group and let $\rho_0\in X_{\mathrm{proj}}(L)$.
Then for each open neighborhood  $U\sse\Hom(L,\PSL_2(\bR))$ of $\rho_0$
there exists an uncountable, tracially disjoint set $\Lambda$ satisfying
\[\rho_0\in\Lambda\sse U\cap X_{\mathrm{proj}}(L)\]
such that
every pair of representations in  $\Lambda$ are joined by an almost faithful path.
\end{thm}

\bp
Let us fix $\rho_0$ and $U$ as in the hypothesis.
For an arbitrary countable subset  $W\sse\bR$,
we claim there exists an almost faithful path $\{\rho_t\}\sse U$ such that \[\tr^2\circ\rho_1(L\setminus T_L)\cap W = \varnothing.\]
Once this claim is established, we may set $\Lambda$ as the maximal tracially disjoint set satisfying that
\[\rho_0\in\Lambda\sse U\cap X_{\mathrm{proj}}(L)\]
and that 
every pair of representations in  $\Lambda$ are joined by an almost faithful path.

For some finite type hyperbolic 2--orbifold $S$,
we can write
\[
L=\pi_1^{\mathrm{orb}}(S).\]

Suppose first that $S$ has no punctures.
Since $S$ is splittable, Lemma~\ref{lem:split-scc} implies that
there exists a filling set $\underline{\gamma}=\{\gamma_1,\ldots,\gamma_m\}$ of essential splitting simple closed curves on $S$. 
Let $C_i$ be the infinite cyclic group generated by $\gamma_i$ (with an arbitrary choice of the base). 

For each $i=1,\ldots,m$, we have a decomposition $L=\GG_i$
such that one of the following holds:
\be[(i)]
\item
$\gamma_i$ is separating and 
\[\GG_i=A_i\ast_{C_i} B_i\] for some $A_i$ and $B_i$.
We let $V_i=A_i\cup B_i$ for this $i$.
\item
$\gamma_i$ is non-separating and \[\GG_i=A_i\ast_{C_i\to t_i C_i t_i^{-1}}\]
where the amalgamating map is a conjugation by some $t_i\in L$.
We let $V_i=A_i$ for this $i$.
\ee

Each nontrivial element of the set $\bigcap_{i=1}^m V_i^L$
is realized by a loop freely homotopic into the complement of $\underline{\gamma}$.
Since $\underline{\gamma}$ is filling and since $S$ has no punctures, it follows that 
\[\bigcap_{i=1}^m V_i^L =  T_L .\]
By applying Lemmas~\ref{lem:pull-apart-psl}
and~\ref{lem:fuchs-analytic} successively, we obtain the desired path.

Now assume $S$ has punctures. By cutting out neighborhoods of cusps, we may choose $S$ to be compact.  We double $S$ to obtain
\[
S' = S\cup_{\partial S}S\]
and $L'=\pi_1^{\mathrm{orb}}(S')$. We extend
 $\rho_0$ to $\rho'\in X_{\mathrm{proj}}(L')$; one can see this by successively applying Lemma~\ref{lem:pull-apart-psl},
as $L'$ can be obtained from $L$ by an amalgamated free product, and then by a sequence of HNN-extensions. Now we can apply the closed case to have a path $\rho'_t$ starting from $\rho'$. The restriction of $\rho_t'$ to $S$ is the desired.
\ep

We note that it is not possible to strengthen  Theorem~\ref{thm:fuchs-flex} so that all $\rho_t$'s are faithful:

\begin{prop}\label{prop:faithful-arc}
Let $L$ be a finitely generated group and let $\{\rho_t\}_{t\in [0,1]}$ be a path in $\Hom(L,\PSL_2(\bR))$ such that $\rho_t$ is faithful for each $t$.
Then $\{\rho_t\}_{t\in [0,1]}$ consists of a single semi-conjugacy class.
\end{prop}

\bp
Let $g\in L\setminus1$. Then for any $t$, $\rot\circ\rho_t(g)\in\bQ/\bZ\setminus 0$  if and only if $g$ is a non-trivial element of $  T_L    $.
Since $\rot\circ\rho_t(g)$ is continuous on $t$, and the set $\bQ/\bZ \setminus 0$ as well as its complement is
dense in $\bR/\bZ$, it follows that $\rot\circ\rho_t(g)$ is constant on $t$. By Theorem~\ref{thm:equiv}, we see the semi-conjugacy class of $\rho_t$ is constant on $t$. 
\ep

For the rest of this section, we will establish that the semi-conjugacy classes of faithful parabolic-free projective actions by a splittable Fuchsian group are abundant.
Note that Theorem~\ref{thm:fuchs-flex} does not directly imply the abundance of such semi-conjugacy classes, since a tracially disjoint set of representations can belong to one semi-conjugacy class; see Lemma~\ref{lem:teichmueller}.

\begin{lem}\label{lem:nonuniform}
For $2\le p<\infty$ and $3\le q<\infty$,
the group $\bZ/p\bZ\ast\bZ/q\bZ$ admits at least one indiscrete faithful parabolic-free projective representation.
\end{lem}
\bp
We put
\begin{align*}
L&=\bZ/p\bZ\ast\bZ/q\bZ=\form{a,b\mid a^p=1=b^q},\\
Y &= \{\rho\in\Hom(L,\PSL_2(\bR))\co
[\rho(ab),\Rot(1)]=1\}.\end{align*}
(Recall that $\Rot(t)$ is an elliptic rotation about $i \in {\mathbb{H}}^2$ by the angle $t$.)
Then $Y$ is algebraic.  Also since $\rho(ab)$ commutes with an irrational rotation ($1$ is an irrational multiple of $2\pi$),  it must be 
 elliptic for each $\rho\in Y$.
In particular, each faithful $\rho\in Y$ is indiscrete.
For all sufficiently large integer $r$, we have a (non-faithful) discrete parabolic-free representation
$\rho_r\in Y$ defined as the composition
\[
L\to L/\fform{(ab)^r} \to \PSL_2(\bR).\]
Note that $L/\fform{(ab)^r} $ is a triangle group and the second map in the previous line and gives a faithful Fuchsian representation of $L/\fform{(ab)^r} $.
For each finite set $Q\sse L\setminus1$, let us define
\[
Y(Q) = \{\rho\in Y \co 4\in\tr^2\circ\rho(Q)\}.\]
By Lemma~\ref{lem:sep}, we see $\rho_r\in Y\setminus Y(Q)$ for some $r$.
Lemma~\ref{lem:baire} implies that for a very general $\rho$ in $Y$, we have $\rho\in X_{\mathrm{proj}}(L)$ is indiscrete.
\ep

\begin{lem}\label{lem:split-indiscrete}
Every finite type splittable Fuchsian group admits at least one indiscrete faithful parabolic-free projective representation.\end{lem}
\begin{rem}
For a surface group, this lemma is immediate from~\cite{Goldman1988IM} and~\cite{DK2006Duke}.
We give an independent proof for general Fuchsian groups.
\end{rem}
\bp[Proof of Lemma~\ref{lem:split-indiscrete}]
Let $S$ be a compact splittable $2$--orbifold and $L=\pi_1^{\mathrm{orb}}(S)$. 
If $\partial S\ne\varnothing$, then $L$ splits as a free product, and a repeated application of Lemmas~\ref{lem:pull-apart-psl} and~\ref{lem:nonuniform} implies $X_{\mathrm{proj}}(L)$ contains an indiscrete representation. So let us assume $S$ is closed and has the signature  \[(g;m_1,\ldots,m_k).\]

\textbf{Case 1. $S$ is splittable by a separating simple closed curve.} 

This case occurs when $g=0$, in particular.
We have a nontrivial splitting
\[
L=A\ast_{\form{c}}B\]
where $c$ is represented by a separating simple closed curve on $S$. Since $A$ and $B$ are (once-punctured) hyperbolic 2--orbifold groups, we see that 
\[
N_A(C)=C=N_B(C).\]
For $Q\sse L\setminus 1$, let us define algebraic sets
\begin{align*}
Y &= \{\rho\in\Hom(L,\PSL_2(\bR))\co
[\rho(c),\Rot(1)]=1\},\\
Y(Q) &= \{\rho\in Y\co 4\in \tr^2\circ\rho(Q)\}.
\end{align*}

Fix a finite set $Q\sse L\setminus1$. Each $q\in Q$ is either $q=d\in C\setminus1$ or 
\[q=a_1b_1a_2b_2\cdots\text{ or }q=b_1a_1b_2a_2\cdots,\] 
for some $a_i\in A\setminus C$ and $b_i\in B\setminus C$.
Let us denote by $C_0,A_0,B_0$ as the finite sets of 
such $d$'s, $a_i$'s or $b_i$'s, respectively.
We have the natural quotient maps
\[
\alpha_p\co A\to A/\fform{c^p},\quad
\beta_p\co B\to B/\fform{c^p}\]
for each $p>0$. By applying Lemma~\ref{lem:sep} for the elements
in $A_0,B_0,C_0$ and also for the elements $[a_i,c],[b_i,c]$, we see that there exists some $p\gg0$ such that
\begin{align*}
\alpha_p(A_0)\cap \alpha_p\form{c}=\varnothing,\quad&
1\not\in\alpha_p(C_0),\\
\beta_p(B_0)\cap \beta_p\form{c}=\varnothing,\quad&
1\not\in\beta_p(C_0).\end{align*}
We fix cocompact Fuchsian representations
\[
\xymatrix{
A/\fform{c^p}\ar[r]^{\bar\alpha} &\PSL_2(\bR) & B/\fform{c^p}\ar[l]_{\bar\beta},}\]
so that $c$ maps to $\Rot(2\pi/p)$ (these come from  complete hyperbolic structures on the corresponding orbifolds of negative Euler characteristic).
We have a commutative diagram as below
\[
\xymatrix{
& A\ar[rd]\ar[r]& A/\fform{c^p}\ar[rrd]^{\bar\alpha} \\
\form{c}\ar@{^(->}[ru] \ar@{^(->}[rd]& & L=A\ast_C B\ar@{-->}[rr]^<<<<<<<{\exists\rho_p} && \PSL_2(\bR)\\
& B\ar[ru]\ar[r]
 & B/\fform{c^p}\ar[rru]_{\bar\beta}}\]
We put \[L_p=\rho_p(L) = \form{\rho_p(A),\rho_p(B)}.\]
Note that $\rho_p(C)$ is a maximal abelian subgroup of $\rho_p(A)$, as is the case of $A/\fform{c^p}$. 
By applying Lemma~\ref{lem:pull-apart-psl} to $L_p$, we have
a faithful parabolic-free representation 
\[\rho_p'\co \rho_p(A)\ast_{\rho_p(C)}\rho_p(B)\to\PSL_2(\bR)\]
such that 
\[\tr^2\circ\rho_p'\circ\rho_p(g)=\tr^2\circ\rho_p(g)\]
for $g\in A$ or $g\in B$.
Define $\rho_p^*$ as the composition 
\[ L\to \rho_p(A)\ast_{\rho_p(C)}\rho_p(B)\to\PSL_2(\bR).\]
Then $4\not\in\tr^2\circ\rho_p^*(Q)$
and $\rho_p^*(c)=\Rot(2\pi/p)\in\SO(2)$, up to conjugacy.
So $Y(Q)$ is proper in $Y$.
By Lemma~\ref{lem:baire}, 
a very general point $\rho\in Y$ belongs to $\in X_{\mathrm{proj}}(L)$
and is indiscrete.

\textbf{Case 2. $S$ is splittable by a non-separating simple closed curve.} 

This case occurs exactly when $g\ge1$.
We proceed in a similar way to the Case 1.
We have a nontrivial splitting
\[
L=A\ast_{C\to C'}=\form{A,s}\]
where $C=\form{c}$ is represented by a non-separating simple closed curve on $S$, and $C'=s^{-1}Cs$. 
Note $A$ is a (twice-punctured) hyperbolic 2--orbifold group and
\[
N_A(C)=C.\]
We define $Y$ and $Y(Q)$ as in Case 1.

Consider a finite set $Q\sse L\setminus1$. Each $q\in Q$ can be written as
\[q=a_1 s^{m_1}a_2 s^{m_2}\cdots\text{ or }
q=s^{m_1}a_1 s^{m_2}a_2 \cdots,\]
for some $a_i\in A\setminus1$ and $m_i\in\bZ\setminus0$.
By Lemma~\ref{lem:sep}, if $p$ is sufficiently large
then
the natural quotient map
\[
\alpha_p\co A\to A/\fform{c^p, (c')^p}\]
sends each $a_i$ to a nontrivial element.
We can further require that
whenever $a_i\not\in\form{c}$, we have $\alpha_p(a_i)\not\in\form{\alpha_p(c)}$,
and 
whenever $a_i\not\in\form{c'}$, we have $\alpha_p(a_i)\not\in\form{\alpha_p(c')}$,

We fix cocompact Fuchsian representations
\[\bar\alpha\co A/\fform{c^p, (c')^p} \to\PSL_2(\bR)\] 
so that $c$ and $c'$ map to rotations of angle $2\pi/p$;
in particular, we have some $\bar s\in\PSL_2(\bR)$ such that 
\[\bar\alpha(c')=\bar s^{-1}\bar\alpha_p(c)\bar s.\]

We have a commutative diagram as below
\[
\xymatrix{
& A\ar[rd]\ar[r]& A/\fform{c^p, (c')^p} \ar[rrd]^{\bar\alpha} \\
C\ast C'\ar[ru] \ar[rd]& & L=A\ast_{C\to C'} \ar@{-->}[rr]^<<<<<<<{\exists\rho_p} && \PSL_2(\bR)\\
& C\ast\form{s}\ar[ru]\ar[r]
 & C/\fform{c^p}\ast\form{s}\ar[rru]_{s\mapsto\bar s}}\]
We put \[L_p=\rho_p(L) = \form{\rho_p(A),\rho_p(s)}.\]
Since $\rho_p(C)$ is a maximal abelian subgroup of $L_p$,
we can apply Lemma~\ref{lem:pull-apart-psl} to obtain
a faithful parabolic-free representation 
\[\rho_p'\co \rho_p(A)\ast_{\rho_p(C)\to\rho_p(C')}\to\PSL_2(\bR)\]
such that 
\[\tr^2\circ\rho_p'\circ\rho_p(g)=\tr^2\circ\rho_p(g)\]
for $g\in A$.
Define $\rho_p^*$ as the composition 
\[ L\to \rho_p(A)\ast_{\rho_p(C)\to\rho_p(C')}\to\PSL_2(\bR).\]
Then $4\not\in\tr^2\circ\rho_p^*(Q)$
and $\rho_p^*(c)=\Rot(2\pi/p)\in\SO(2)$.
So $Y(Q)\ne Y$.
By Lemma~\ref{lem:baire},
 a very general point $\rho\in Y$
is  indiscrete 
and belongs to
$X_{\mathrm{proj}}(L)$.
\ep

Let $S_g$ denote a closed oriented hyperbolic surface of genus $g$.
It follows from~\cite{Mann2015IM} that each discrete faithful representation
\[\rho_0\co \pi_1(S_g)\to \PSL_2(\bR)\] is \emph{locally rigid}\index{locally rigid} in $\Hom(\pi_1(S_g),\Homeo_+(S^1))$. This means that an open neighborhood $U$ of $\rho_0$ is a single semi-conjugacy class.
Here we prove that each indiscrete $\rho_0\in X_{\mathrm{proj}}(L)$ of a splittable Fuchsian group $L$ is \emph{locally flexible}\index{locally flexible} in the following strong sense:

\begin{thm}\label{thm:uncountable}
Let $L$ be a splittable Fuchsian group.
\be
\item
Then there exists uncountably many pairwise-inequivalent 
indiscrete faithful parabolic-free projective representations of $L$.
\item
Let $\rho$ be an indiscrete faithful parabolic-free projective representation of $L$.
Then every open neighborhood $U\sse \Hom(L,\PSL_2(\bR))$ of $\rho$
contains an uncountable set $\Lambda\sse U\cap X_{\mathrm{proj}}(L)$ of pairwise-inequivalent indiscrete representations such that each pair of points in $\Lambda$ can be joined by an almost faithful path in $U$.
\ee
\end{thm}

\bp
By Lemma~\ref{lem:split-indiscrete}, there exists at least one  $\rho$ as in part (2).
The conclusion follows from Theorem~\ref{thm:fuchs-flex}.
\ep

\begin{rem}\label{rem:DK2006}
For a surface group $L$,
there exists a dense set of faithful representations in
\[\Hom(L,\PSL_2(\bK)),\]
as shown in~\cite{DK2006Duke}.  Let $\rho_0\co L\to\PSL_2(\bK)$ be a representation. A small modification of a Baire category argument in the same paper implies that $\rho_0$ admits an arbitrarily small deformation $\{\rho_t\}_{t\in [0,1]}$ such that the hypothesis of Theorem~\ref{thm:fuchs-flex} holds,
and from this, one can then deduce each indiscrete representation in $X_{\mathrm{proj}}(L)$ in Theorem~\ref{thm:uncountable} is accumulated.
However, such an approach does not directly yield almost faithful paths 
or work with the orbifold case.
\end{rem}

\begin{rem}
Let $L=\pi_1(S_g)$.
For each $\rho\in\Hom(L,\Homeo_+(S^1))$,
the Milnor--Wood Inequality asserts that $\eu(\rho)$ takes integer values between $2-2g$ and $2g-2$.
Moreover, $|\eu(\rho)|=2g-2$ if and only if $\rho$ is injective and semi-conjugate to a discrete faithful representation of $L$ into $\PSL_2(\bR)$~\cite{Goldman1988IM,Matsumoto1987IM,Mann2015IM}.
On the other hand, a connected component of \[\Hom(\pi_1(S_g),{\PSL_2(\bR)})\] is precisely determined by its Euler number~\cite{Goldman1988IM}. Each component contains a faithful parabolic-free (possibly indiscrete) representation by~\cite{DK2006Duke}.
So Theorem~\ref{thm:uncountable} implies that each component where the Euler number is not $\pm(2-2g)$ contains uncountably many distinct semi-conjugacy classes of faithful parabolic-free actions which are joined by almost faithful paths.
\end{rem}

\begin{exmp}\label{exmp:triangle}
Theorem~\ref{thm:uncountable} is not true without the splittability hypothesis.
For example, consider a hyperbolic triangle group
\[L=\form{a,b,c\mid a^p=b^q=c^r=abc=1}\]
such that $2\le p,q,r<\infty$ and $1/p+1/q+1/r<1$.
If $\rho\co L\to\PSL_2(\bR)$ is a representation,
then the elliptic centers of $\rho (a),\rho (b),\rho (c)$ form a hyperbolic triangle with internal angles 
\[
{u\pi}/{p},\quad{v\pi}/{q},\quad{w\pi}/{r},\]
for some $u,v,w\in\bN$.
Since $(u,v,w)$ determines the isometry type of the triangle formed by the centers of those rotations,
there are only finitely many possible conjugacy classes of a  representations $\rho\co L\to\PSL_2(\bR)$. 
We point out that not only are there finitely many conjugacy classes of
projective representations, but further that the are also finitely many semi-conjugacy classes of such representations in
$\Homeo^+(S^1)$ \cite{CalegariForcing}.

By a similar (and easier) reasoning, we see that an infinite dihedral group does not admit an indiscrete representation.

\end{exmp}

\subsection{Simultaneous Control of Rotation Numbers}\label{ss:simul}
So far we have developed techniques of finding uncountably many inequivalent faithful projective actions of a given group. 
In this subsection, we describe an approach to simultaneously control the rotation numbers of given (finitely many) group elements. 

To state the main result of this section more precisely, let us introduce a notation.
Suppose $L$ is a finitely generated group and $Z$ is a set of elements (or, conjugacy classes) in $L$. 
We write
\[E(L,Z)=\{\rho\in X_{\mathrm{proj}}(L) \co \rot\circ\rho(Z)\text{ is a singleton}\}.\]

A \emph{multi-curve} $Z$ on a $2$--orbifold $S$ means a finite collection of disjoint simple closed curves on $S$, which can be identified with a certain set of conjugacy classes in $\pi_1^{\mathrm{orb}}(S)$. We assume further that each component of $S \setminus Z$ has negative orbifold Euler number.
For each representation $\rho\co\pi_1^{\mathrm{orb}}(S)\to\PSL_2(\bR)$ we define
\[\rot\circ\rho(Z)=\{\rot\circ\rho(g)\co [g]\in Z\}\sse\bR/\bZ.\]

\begin{thm}\label{thm:simul}
Let $S$ be a finite type hyperbolic $2$--orbifold and let $Z$ be a multi-curve on $S$. If each element of $Z$ is either splitting or peripheral (i.e.\ corresponding to a puncture),  then  the following set is uncountable:
\[\{\rot\circ\rho(Z)\co \rho\in E(\pi_1^{\mathrm{orb}}(S),Z)\}.\]
\end{thm}
Hence, there exist uncountable set $\Lambda$ of pairwise-inequivalent indiscrete faithful parabolic-free projective representations of $\pi_1^{\mathrm{orb}}(S)$
such that each $\rho\in\Lambda$ maps $Z$ to elliptic elements of the same nonzero rotation number.
So we have the following corollary.
Recall
\[M(L) = \Out(L)\backslash \Hom(L,\Homeo_+(S^1))/\text{semi-conjugacy}.\]

\begin{cor}\label{cor:simul}
If $S$ is a finite type hyperbolic surface
and if $Z$ is a multi-curve on $S$,
then the image of $E(\pi_1(S),Z)$ in $M(\pi_1(S))$ is uncountable.
\end{cor}
We will prove Theorem~\ref{thm:simul} in this subsection.
Recall our notation
\[\fform{Z^p}=\fform{\{g^p\co [g]\in Z\}}\unlhd\pi_1^{\mathrm{orb}}(S).\]
\begin{lem}
Let $S$ and $Z$ be as in Theorem~\ref{thm:simul}.
Then for a sufficiently large $p>0$, the following hold.
\be
\item
If $[c]\in Z$ is peripheral on $S$, then the group $\pi_1^{\mathrm{orb}}(S)/\fform{Z^p}$
contains $\form{c}\cong\bZ/p\bZ$ as a maximal abelian subgroup.
\item
For distinct peripheral $c_1, c_2\in Z$, the images of $\form{c_1}$ and $\form{c_2}$ are non-conjugate in $\pi_1^{\mathrm{orb}}(S)/\fform{Z^p}$.
\item
There exists
a faithful parabolic-free representation $\phi_p\co \pi_1^{\mathrm{orb}}(S)/\fform{Z^p}\to\PSL_2(\bR)$ such that
\[\rot\circ\phi_p(Z)=\{1/p\}.\]
\ee
\end{lem}
\bp
Put $L=\pi_1^{\mathrm{orb}}(S)$, and define $L_p=L/\fform{Z^p}$.
The group $L_p$ has a graph of groups structure such that  the edges groups $\{E_j\}$ are all isomorphic to $\bZ/p\bZ$. 
Since we are assuming $p$ is sufficiently large, the vertex groups $\{V_i\}$ are hyperbolic 2--orbifold groups.
Moreover, each edge group $E_j$ corresponds to a cone point of an incident vertex group $V_i$. By the property of a Fuchsian group, we have an embedding $E_j\le V_i$. By repeated applications of amalgamated free products and HNN--extensions, we see that each vertex group $V_i$ embeds into $L_p$.

In part (1), we let $\bar c$ denote the image of $c$ in $L_p$. 
Let $V_i$ be the vertex group containing $\form{\bar c}$ as the cyclic group corresponding to a cone point of order $p$.
Again by the property of a hyperbolic 2--orbifold group,
the element $\bar c$ is not conjugate in $V_i$ into any of the edge groups incident with $V_i$.
Suppose $g\bar c g^{-1}=\bar c$ for some $g\in L_p$. 
By the normal form theorem for graph of groups~\cite{LS2001},
we see that $g\in V_i$. This implies $g\in \form{\bar c}$, as desired.
The proof of part (2) is very similar, by considering the equation of the form $g\bar c_1 g^{-1}=\bar c_2$.

Let us consider part (3).
If all elements in $Z$ are peripheral, then we can find a parabolic-free Fuchsian representation $\phi_p$ such that the curves in $Z$ are mapped to rotations of angle $2\pi/p$,
while the other peripheral elements are mapped to hyperbolic elements.
So we may assume some element $c\in Z$ is non-peripheral.
We will induct on the complexity $\xi(S)>0$.

Suppose $c$ is separating. 
We have a nontrivial decomposition $L=A\ast_C B$ for $C=\form{c}$.
Possibly after conjugation, we may write 
\[\Sigma_A\cup \Sigma_B=Z,\quad \Sigma_A\cap \Sigma_B=\{c\}.\]
By the inductive hypothesis on each component of $S\setminus c$,
we have faithful parabolic-free representations (for sufficiently large $p$)
\[
\xymatrix{
A_p=A/\fform{\Sigma_A^p}\ar[r]^{\alpha_p}&
 \PSL_2(\bR)
 &
B_p=B/\fform{\Sigma_B^p}.\ar[l]_{\beta_p}
}\]
such that \[\alpha_p(c)=\beta_p(c)=\Rot(2\pi/p).\]
Here, we abuse the notation so that $c$ also means the images of $c$ in the quotients.

By part (1), we can apply the Pulling-Apart Lemma
to the group
\[
\form{A_p,B_p}\]
to have the desired representation $L_p\to\PSL_2(\bR)$;
see Lemma~\ref{lem:pull-apart-psl} (i).

If $c$ is non-separating, the proof is almost identical using part (2) and the Pulling-Apart Lemma (iii).
\ep

The following composition yields a stably injective sequence of parabolic-free representations:
\[
\xymatrix{
\pi_1^{\mathrm{orb}}(S)\ar[r] & \pi_1^{\mathrm{orb}}(S)/\fform{Z^{p!}}\ar[r]^{\phi_p} & \PSL_2(\bR)
}\]

Theorem \ref{thm:simul}  follows from the next lemma.
\begin{lem}\label{lem:uv1}
Let $L$ be a finitely generated group and let $Z\sse L$ be a nonempty finite set.
Suppose that for all sufficiently large $p$ 
there exists $c_p\in(0,1/p]$
and
 a parabolic-free projective representations $\rho_p$
such that
\[\rot\circ\rho_p(Z)=\{c_p\}\]
and
such that
the sequence $\{\rho_p\}$ is stably injective.
Then the following set is uncountable:
\[
\{\rot\circ\rho(Z)\co \rho\in E(L,Z)\}.\]
\end{lem}

\bp
Let us fix a finite set $A\sse L\setminus Z$ so that $L=\form{Z\cup A}$.
Define an algebraic set
\[
Y^* =\SO(2)\times\PSL_2(\bR)^Z\times \PSL_2(\bR)^A.\]
Consider an arbitrary
\[\chi = (\alpha_0,\{\alpha_t\}_{t\in Z},\{\beta_a\}_{a\in A})\in Y^*.\]
We will denote the projection map to the first component
\[
\chi\mapsto \alpha_0\]
as $\Pi_0$.
Given $\chi \in  Y^\ast$, there exists at most one homomorphism \[\phi_\chi\in\Hom(L,\PSL_2(\bR))\] satisfying the following:
\[
\phi_{\chi}(x) =
\begin{cases}
\alpha_0^{\alpha_x},&\text{ if }x\in Z\\
\beta_x,&\text{ if }x\in A.
\end{cases}
\]
So we have an algebraic set 
\[Y=\{\chi\in Y^*\co \phi_\chi\text{ defines a homomorphism}\},\]
For any $\rho_p$ as in the lemma, there exists $\chi_p\in Y$ such that 
\[
\rho_p=\phi_{\chi_p},\quad\rot\circ\rho_p(Z)=\Pi_0(\chi_p).\]

For each finite subset $Q\sse L\setminus 1$, we have an algebraic set 
\[Y(Q) = \{ \chi \in Y \co \phi_\chi(Q)\text{ contains a (possibly trivial) parabolic element}\}.\]
We may identify 
\[E(L,Z)=Y\setminus \bigcup_{Q\sse L\setminus1, \text{ finite }} Y(Q).\]

Let $Z\sse \SO(2)$ be an arbitrary countable set.
For each finite subset $U\sse Z$, we define 
\[
W(U) = \Pi_0^{-1}(U)\cap Y.\]
Since $\{\rho_p\}$ is stably injective and parabolic-free, 
each $Y(Q)$ is proper in $Y$. Moreover, the condition on the rotation number implies that $W(U)$ is proper as well. 
By Lemma~\ref{lem:baire}, the following set is uncountable:
\[E(L,Z)\setminus\bigcup_{U\sse Z,\text{ finite}}W(U).\]
It follows that $\Pi_0(E(L,Z))$ is uncountable.
\ep

\begin{exmp}\label{exmp:genus-two}
Let $L$ be the fundamental group of a genus--two surface.
For an illustrative purpose, let us prove that the image of $X_{\mathrm{proj}}(L)$ is uncountable in $M(L)$, without resorting to Theorem~\ref{thm:uncountable}.
See Lemma~\ref{lem:uv-alg} for a generalization.

We have a sequence of maps as follows.
\[\xymatrix{
L=\form{a,b,c,d| [a,b]=[c,d]}\ar@{^(->}[r]^<<<<<{\iota}
& H=\form{a,b,t|[[a,b],t]=1}\ar[d]_{f_p}\\
\iota(c)=a^t,\ \iota(d)=b^t&F_2=\form{a,b} \ar[d]_{g_q}\\
f_p(t)=[a,b]^p& \pi_1^{\mathrm{orb}}(\Sigma_q) = \form{a,b| [a,b]^q=1}
\ar@{^(->}[r]^<<<<{\rho_q}& \PSL_2(\bR).
}\]
Here,  $g_q$ is the natural quotient map,
and 
$\Sigma_q$ is a compact genus--one hyperbolic orbifold with exactly one cone point, whose order is $q$.
The group $\pi^{\mathrm{orb}}_1(\Sigma_q)$ admits a 
 cocompact Fuchsian
 representation $\rho_q$
such that \[\rho_q[a,b]=\Rot(2\pi/q).\]

Recall the definition of stable injectivity from Definition~\ref{defn:stab-inj} and from the paragraph following it.
The sequence $\{f_p\}$ is stably injective by applying Lemma~\ref{lem:baumslag} to the case $u=[a,b]\in F_2$.
Also, Lemma~\ref{lem:sep} implies that  $\{g_q\}$ is stably injective.
For each $p\ge1$ we can choose a sufficiently large \[q=q(p)\ge p\] such that  \[\left\{\phi_p = g_{q(p)}\circ f_p\co H\to \pi_1^{\mathrm{orb}}(\Sigma_{q(p)})\right\}\] 
is a stably injective sequence. Note that \[\rot\circ \phi_p[a,b]=1/q.\]
So we can apply Lemma~\ref{lem:uv1} and see that
the following set is uncountable:
\[
\left\{\rot\circ\rho\left([a,b]\right)\co \rho\in E(H,\left\{[a,b]\right\})\right\}.\]
In particular, $L$ admits uncountably many pairwise-inequivalent indiscrete faithful parabolic-free representations.
\end{exmp}

\section{Combination Theorem for Flexible Groups}\label{sec:flex}
In this section, we establish a combination theorem (Theorem~\ref{thm:main-comb}) for the class of flexible and liftable-flexible groups. 

\subsection{Statement of the result}
Recall that a set 
\[ \Lambda\sse\Hom(L,\PSL_2(\bR))\]
 is tracially disjoint if for all distinct $\rho$ and $\lambda$ in $\Lambda$ we have \[\tr^2\circ\rho(L\setminus T_L)\cap\tr^2\circ\lambda(L\setminus T_L)=\varnothing.\]
The following definition is given in the introduction.

\bd\label{defn:flex}
A finitely generated group $L$ is called \emph{flexible}\index{flexible} if there exists an uncountable, tracially disjoint subset $\Lambda$ of indiscrete faithful parabolic-free projective representations.
If, furthermore, $\Lambda$ can be chosen so that each representation in $\Lambda$ is liftable, then we say $L$ is \emph{liftable-flexible}\index{flexible!liftable-flexible}.
\ed

We will often use the following alternative characterization of flexibility. The proof is immediate by Zorn's lemma.

\begin{lem}\label{lem:flex-count}
A finitely generated group $L$ is {flexible} if and only if 
for each countable set $W\sse\bR$ there exists an indiscrete faithful representation $\rho\co L\to \PSL_2(\bR)$ such that 
\[\tr^2\circ\rho(L\setminus T_L)\cap W=\varnothing.\]
Moreover, $L$ is liftable-flexible
if and only if such $\rho$ can be chosen to be liftable.
\end{lem}

\begin{exmp}\label{exmp:flex}
\be
\item
The infinite cyclic group is liftable-flexible.
This is because we can map a generator to an arbitrary irrational rotation.
\item
Infinite dihedral groups and triangle groups are not flexible (Example~\ref{exmp:triangle}).
\ee
\end{exmp}

Let us note several restrictions for flexible groups.
The proofs are trivial.

\begin{prop}
\be
\item 
A flexible group is commutative--transitive.
\item
Every finite subgroup of a flexible group is cyclic.
\item
A liftable-flexible group is torsion-free.
\ee
\end{prop}

We will denote by $\FF$ and $\FF^\sim$ the classes of flexible groups and liftable-flexible groups, respectively. Note that $\FF^\sim\sse\FF$. The main result of this section is the following.

\begin{thm}[Combination Theorem for Flexible Groups]\label{thm:main-comb}
Let $\HH=\FF$ or $\FF^\sim$.
\be
\item\label{main:inf-cyc}
Infinite cyclic groups are in $\HH$.
\item\label{main:sub}
If $H\le L\in\HH$ and $[L:H]<\infty$, then $H\in\HH$.
\item\label{main:free} If $A$ and $B$ are in $\HH$, then so is $A\ast B$.
\item\label{main:amalgam}
Let $A,B\le L$ and $C\le A\cap B$ satisfy that \[L=\form{A,B}\in\HH.\]
If $C$ is nontrivial malnormal abelian in $A$ and in $B$,
then $A\ast_C B\in \HH$.
\item\label{main:hnn}
Let $A,C\le L$ and $s\in L$ satisfy that  \[L=\form{A,s}\in\HH\] and that $C$ and $C^s$ are nontrivial malnormal abelian subgroups of $A$. Assume either 
\begin{itemize}
\item $s\in Z(C)$, or
\item $C$ and $C^s$ are not conjugate in $A$.
\end{itemize}
Then  $A\ast_{\Inn(s)\co C\to C^s}\in\HH$.
\item\label{main:fin-cyc}
If $A$ is in $\FF$ and $B$ is a finite cyclic group, then $A\ast B\in\FF$.
\ee
\end{thm}

\begin{rem}
The malnormality condition of $C$ can be replaced by a seemingly weaker (but equivalent in our case) condition that
\[ N_L(C)\cap V=C,\]
where $V=A\cup B$ in (1) and $V=A$ in (2).
\end{rem}

\subsection{Proof}
\bp[Proof of Theorem~\ref{thm:main-comb}]
Throughout this proof, we will fix a countable set $W\sse\bR$.
Our goal in each part is, to find $\rho$ satisfying the conditions in Lemma~\ref{lem:flex-count}.
We will first consider the case $\HH=\FF$.

Part (1) is shown in Example~\ref{exmp:flex}. For part (2), suppose $\rho_0\co L\to \PSL_2(\bR)$ is indiscrete faithful and satisfies
\[
\rho_0(L\setminus T_L)\cap W=\varnothing.\]
By Lemma~\ref{lem:dense}, we can pick $g_0\in L$ such that $\rot\circ\rho(g_0)$ is irrational.
There is $N>0$ such that
\[h_0=g_0^N\in H.\] 
So $\lambda\restriction_H$ is the desired indiscrete faithful representation. 

In the Pulling-Apart Lemmas,
we note that if $g$ is an irrational rotation in $L\le\PSL_2(\bR)$,
and if $\nu$ is sufficiently close to the identity, 
then very generally $\rho_\nu(g)$ is also an irrational rotation. This is because $\rho_\nu$ is continuous on $\nu$ and $g$ has infinite order as an element of $L$. So the parts (3) through (5) are immediate from the Pulling-Apart Lemma (Lemma~\ref{lem:pull-apart-psl}).

Part (6) follows from Lemma~\ref{lem:free} (1).

Now, the case $\HH=\FF^\sim$ follows from the lemma below.
\ep

\begin{lem}\label{lem:lift}
Let $G,\tilde G$ be groups and $p\co \tilde G\to G$ be an epimorphism.
Choose an abelian subgroup $\mu\le G$ such that $\tilde\mu=p^{-1}(\mu)$ is abelian,
and fix \[\nu\in\mu,\quad\tilde\nu\in \tilde\mu, \] such that $p(\tilde{\nu}) = \nu$.
For some group $L$, suppose we have maps \[\rho\co L\to G,\quad\tilde\rho\co L\to \tilde G,\] 
such that $p\circ\tilde\rho = \rho$.
\be
\item
Let $A,B\le L$ and $C\le A\cap B$ satisfy $L=\form{A,B}$ and $\rho(C)\le\mu$.
If we put \[L^*=A\ast_C B,\]
then there uniquely exist homomorphisms \[\rho^*\co L^*\to G,\quad
\tilde\rho^*\co L^*\to \tilde G\] such that the following hold:
\be[(i)]
\item $\rho^*\restriction_A=\rho\restriction_A$ and $\rho^*\restriction_B=\Inn(\nu)\circ\rho\restriction_B$;
\item $\tilde\rho^*\restriction_A=\tilde\rho\restriction_A$ and $\tilde\rho^*\restriction_B=\Inn({\tilde\nu})\circ\tilde\rho\restriction_B$;
\item $p\circ\tilde\rho^*=\rho^*$.
\ee
\item
Let $A,C\le L$ and $s\in L$ satisfy $C\cup C^s\le A$,
$L=\form{A,s}$ and $\rho(C)\le\mu$.
If we put \[L^*=A\ast_{C\to s^{-1}Cs},\]
then there uniquely exist homomorphisms \[\rho^*\co L^*\to G,\quad\tilde\rho^*\co L^*\to \tilde G\] such that the following hold:
\be[(i)]
\item $\rho^*\restriction_A=\rho\restriction_A$ and $\rho^*(s)=\nu\rho(s)$;
\item $\tilde\rho^*\restriction_A=\tilde\rho\restriction_A$ and $\tilde\rho^*(s)=\tilde\nu\tilde\rho(s)$;
\item $p\circ\tilde\rho^*=\rho^*$.
\ee
\ee
\end{lem}

\bp
The proof is immediate from the universality in the following diagrams.
\[
\xymatrix{
& A\ar[rrrrrrd]^{\rho}\ar[rd]\ar[rrrd]^>>>>>>>>{\tilde\rho}\\
C\ar[ru] \ar[rd]& & L^*\ar[rr]^<<<<<{\tilde\rho^*} &&  \tilde G\ar[rrr]^<<<<<<<<p
&&& G\\
& B\ar[rrrrrru]_>>>>>>>>>>>>{\Inn(\nu)\circ\rho}\ar[ru]\ar[rrru]_>>>>>>>>{\Inn(\tilde\nu)\circ\tilde\rho}
}\]
\[
\xymatrix{
& A\ar[rrrrrrd]^{\rho}\ar[rd]\ar[rrrd]^>>>>>>>>{\tilde\rho}\\
C\ast C^s\ar[ru] \ar[rd]& & L^*\ar[rr]^<<<<<{\tilde\rho^*} &&  \tilde G\ar[rrr]^<<<<<<<<p
&&& G\\
& C\ast\form{s}\ar[rrrrrru]_>>>>>>>>>>>>{s\mapsto \nu\rho(s)}\ar[ru]\ar[rrru]_>>>>>>>>{s\mapsto \tilde\nu\tilde\rho(s)}
}\]\ep

For the rest of this section, we will investigate consequences of the combination theorem and the properties of (liftable-)flexibility.

A particularly simple cases of the combination theorem are as below.
Let $A$ be a group and $C$ be a subgroup.
Then we simply write 
\[A\ast_C = A\ast_{\Id\co C\to C} = A\ast_C(C\times\bZ).\]
For a group $A$ and a subgroup $C$, we denote by $A\ast_C\bar A$ the amalgamation of two copies of $A$ along $\Id_C$.

\begin{cor}\label{cor:double}
Let $A$ be a group and let $C$ be an abelian and malnormal subgroup of $A$.
If $A$ is flexible or liftable-flexible, then so are  $A\ast_C$ and $A\ast_C \bar A$.
\end{cor}
\bp
Apply Combination Theorem (\ref{main:free}), (\ref{main:amalgam}), (\ref{main:hnn}) for $L=A=B$ and $s=1$.
\ep

Let us observe a consequence for Fuchsian groups.
\begin{cor}\label{cor:fuchsian-flexible}
\be
\item
Finite type splittable Fuchsian groups are flexible.
\item
Finite type torsion-free Fuchsian groups are liftable-flexible.
\ee\end{cor}

\bp
Part (1) follows from Theorem~\ref{thm:fuchs-flex} and Theorem~\ref{thm:uncountable}.
For part (2), let $L$ be a torsion-free Fuchsian group of finite type. 
If $L$ is free, then  Theorem~\ref{thm:main-comb}  (\ref{main:inf-cyc}) and (\ref{main:free}) imply that $L$ is liftable-flexible.
If $L$ is not free, then since every surface group is a finite index subgroup of the genus 2 surface
group, we may assume
\[L = \pi_1(S_2)=F_2\ast_\bZ F_2=\form{a,b}\ast_{[a,b]=[c,d]}\form{c,d},\]
using Theorem~\ref{thm:main-comb} (\ref{main:sub}).
In this case, Corollary~\ref{cor:double} implies  $L$ is liftable-flexible. 
\ep

\subsection{Exotic Circle Actions}
In this section we provide uncountably many distinct exotic circle actions of flexible groups, and make a connection to the existence of quasimorphisms on those groups. We will prove that most Fuchsian groups and their iterated generalized doubles are flexible. Limit groups will be shown to be liftable-flexible.

Recall from Theorem \ref{t:ghys2} that there is an embedding
\[\Hom(L,\Homeo_+(S^1))/\text{semi-conjugacy} \to H^2_b(L;\bZ)\]
defined by $\rho\mapsto \rho^*\eu_b$.

\begin{lem}\label{lem:indep}
Let $L$ be a countable group and $g_0\in L$.
Suppose  \[\Lambda\sse \Hom(L,\PSL_2(\bR))\] satisfies that the family
\[\{\rot\circ\rho(g_0)\co\rho\in\Lambda\}\] is $\bZ$--linearly independent
(as numbers in
$\bR/\bZ$).
Then \[\{\rho^*\eu_b\co\rho\in\Lambda\}\] is 
linearly independent in $H^2_b(L;\bZ)$.
\end{lem}

\bp
Let us suppose that \[k_0\rho_0^*\eu_b=\sum_{i=1}^n k_i\rho_i^*\eu_b\] for some $\rho_i\in\Lambda$ and nonzero integers $k_i$. Then there exists a bounded map $\beta\co G\to\bZ$ such that
 \[k_0\rho_0^*\eu^x=\sum_{i=1}^n k_i\rho_i^*\eu^x+\partial\beta\] as two--cocycles.
By Lemma~\ref{lem:eu-rot}, we have
\begin{align*}
\rot \rho_0(  g_0 )
&=\lim_{N\to\infty}
 \frac1N \sum_{j=1}^{N}
\rho_0^*\eu^x( g_0 , g_0 ^{j-1})
\\
&=\lim_{N\to\infty}
 \frac1{k_0 N} \sum_{j=1}^{N}
\left(
\sum_{i=1}^n
k_i\rho_i^*\eu^x( g_0 , g_0 ^{j-1}) +\beta( g_0 )+\beta( g_0 ^{j-1})-\beta( g_0 ^{j})
\right)
\\
&=\frac1{k_0}\beta( g_0 )+\sum_{i=1}^n \frac{k_i}{k_0}\rot\rho_i( g_0 ). 
\end{align*}
This contradicts the $\bZ$-linear independence hypothesis (since the hypothesis requires $\bZ$-linear independence  as numbers in
$\bR/\bZ$).
\ep

\begin{cor}\label{cor:consequence-psl}
If $L$ is a flexible group, then there exists 
an uncountable set $\Lambda$ of
faithful parabolic-free projective representations of $L$
such that the following hold.
\be[(i)]
\item All representations in $\Lambda$ stay distinct in the double coset space
\[M(L) = \Out(L)\backslash \Hom(L,\Homeo_+(S^1))/\text{semi-conjugacy}.\]
\item
The image of $\Lambda$ is $\bZ$--linearly independent
in $H^2_b(L;\bZ)$.
\item
If, furthermore, $L$ is liftable-flexible, then $\Lambda$ maps to the kernel of 
\[H^2_b(L,\bZ)\to H^2(L,\bZ).\]
\ee
\end{cor}

\bp
Let $\Lambda_0$ be a maximal tracially disjoint set of indiscrete representations in 
$X_{\mathrm{proj}}\sse\Hom(L,\PSL_2(\bR))$.
By flexibility (see the remark after Definition \ref{defn:flex-intro}), the set $\Lambda_0$ is uncountable and $L\ne T_L$. 
(Alternately, H\"older's theorem gives us that if a group acts on $S^1$ and every element has torsion, then the group is
semi-conjugate to a group of rotations.)

For two distinct $\rho,\lambda\in\Lambda_0$, the two sets
\[\rot\circ\rho(L)\setminus(\bQ/\bZ),\quad
\rot\circ\lambda(L)\setminus(\bQ/\bZ)\]
are a disjoint pair of infinite subsets of $\bR/\bZ$ by Lemma \ref{lem:flex-count}. Hence $\Lambda_0$ satisfies the condition (i). 

For each $\rho\in\Lambda_0$, there exists some $g\in L$ such that $\rho(g)$ is an elliptic element of infinite order. Since $\Lambda_0$ is uncountable and $L$ is countable,
we have some $g_0\in L$ such that the following set is uncountable:
\[\Lambda_1=\{\rho\in\Lambda_0\co \rot\circ\rho(g_0)\not\in\bQ/\bZ\}.\]
By our assumption, for all distinct $\rho,\lambda\in\Lambda_1$, we have
\[\rot\circ\rho(g_0)\ne\rot\circ\lambda(g_0).\]
Since the $\bZ$--span of a countable set in $\bR/\bZ$ is countable,
we can find an uncountable subset $\Lambda\sse\Lambda_1$ such that the following set is $\bZ$--linearly independent:
\[\{\rot\circ\rho(g_0)\co \rho\in\Lambda\}.\]
By Lemma~\ref{lem:indep}, we see that $\Lambda$ satisfies (ii). 

If $L$ is liftable-flexible as in (iii),
we will further require that each representation in $\Lambda_0$ is liftable.
In this case, $\Lambda_0$ (and its subset $\Lambda$) maps to the kernel of 
\[H^2_b(L,\bZ)\to H^2(L,\bZ).\]
\ep

\subsection{Quasimorphisms}
It is known that $\mathrm{HQM}(L;\bZ)$ is  infinite dimensional if $L$ is a nontrivial free group, 
a  non-elementary word-hyperbolic group or a mapping class group~\cite{Brooks1981,EpsteinFujiwara1997,BBF2016IHES}.
Our construction shows that for a liftable-flexible $L$, 
the abelian group $\mathrm{HQM}(L;\bZ)$ has
 uncountable dimension
even when we count only defect--one subadditive quasimorphisms:

\begin{cor}\label{cor:qm}
If $L$ is liftable-flexible,  then there exists a set $\Lambda$ of integer--valued subadditive defect--one quasimorphisms
such that \[\{[\lambda]\in \mathrm{HQM}(L;\bZ)\co \lambda\in \Lambda\}\]
is uncountable and linearly independent. 
\end{cor}

\bp
This is immediate from 
Corollary~\ref{cor:consequence-psl} (ii), along with Lemmas~\ref{lem:qm} and~\ref{lem:indep}.
\ep

\subsection{Limit Groups}
\bd[cf.~\cite{CG2005IJM}]\label{defn:igd}
\be
\item
An epimorphism $\phi\co L\to L'$ is called a \emph{generalized double over $L'$}\index{double (of a group)!generalized} if 
$L=A\ast_C B$ for some finitely generated nontrivial groups $A$ and $B$ such that $C$ is maximal abelian both in $A$ and in $B$,
and such that $\phi\restriction_A$ and $\phi\restriction_B$ are both injective.
\item
An epimorphism $\phi\co L\to L'$ is called a \emph{centralizer extension}\index{centralizer extension} if 
$L=A\ast_C$ for some finitely generated nontrivial group $A$ such that $C$ is maximal abelian in $A$ and such that $\phi\restriction_A$ is injective.
\item\label{p:igd}
If there is a sequence of epimorphisms
\[\xymatrix{L=L_n\ar@{->>}[r]^{f_n} & L_{n-1}\ar@{->>}[r]^{f_{n-1}} & \cdots \ar@{->>}[r] 
&  L_1\ar@{->>}[r]^>>>>>{f_1}  & L_0,}\]
such that each $f_i$ is a generalized double or a centralizer extension,
then we say $L$ is an \emph{iterated generalized double over $L_0$}\index{double (of a group)!iterated generalized double}
and write $L\in \mathrm{IGD}(L_0)$.
\item
In part (\ref{p:igd}), we say a group $L$ is an \emph{iterated malnormal generalized double over $L_0$}\index{double (of a group)!iterated malnormal generalized double}
if the amalgamating subgroup of $L_i$ in each step is 
further required to be malnormal in the vertex groups of $L_i$.
We write $L\in \mathrm{IMGD}(L_0)$ in this case.
\ee
\ed

The following characterizes limit groups as iterated generalized doubles of free groups.

\begin{lem}\label{lem:igd}
\be
\item
Every limit group is in $\mathrm{IMGD}(F_r)$ for some $r$.
\item
Suppose $L$ and $H$ are groups such that $L$ is in $\mathrm{IMGD}(H)$.
If $H$ is flexible or liftable-flexible, then so is $L$.
\ee
\end{lem}

\bp
(1)
It is well-known that every limit group $L$ is torsion-free and in $\mathrm{IGD}(F_r)$ for some $r$;
see~\cite{Sela2001PIHES,CG2005IJM}.
So, we have a sequence as shown in part (\ref{p:igd})  of Definition~\ref{defn:igd}.
Since every limit group embeds into $\PSL_2(\bR)$
\cite{BF2009book,Wilton2009solutions},
so does each $L_i$.
For a torsion-free group $L\le\PSL_2(\bR)$, a subgroup $C$ is maximal abelian if and only if $C$ is malnormal (Remark~\ref{rem:main-comb-intro}).
Hence each $L_i$ is actually an iterated \emph{malnormal} generalized double of $F_r$.

Part (2) is immediate from the Combination Theorem \ref{thm:main-comb}.
\ep

Recall $F_r$ is liftable-flexible (Theorem~\ref{thm:main-comb}). 
By Lemmas~\ref{lem:pull-apart-psl} and~\ref{lem:igd}, we have the following general result.
\begin{cor}\label{cor:limit-gen}
\be
\item
Every nontrivial limit group is liftable-flexible.
\item
Every nontrivial limit group admits a set $\Lambda$ of integer--valued subadditive defect--one quasimorphisms such that \[\{[\lambda]\in \mathrm{HQM}(L;\bZ)\co \lambda\in \Lambda\}\]
is uncountable and linearly independent. 
\ee
\end{cor}

Consider the special case of part (1) in Corollary~\ref{cor:limit-gen}, when $L$ is non-Fuchsian. Then we can give a shorter proof of the flexibility of $L$ as follows; see~\cite{Wilton2009solutions,BG2010JAM} for similar ideas.  Let $Q\sse L\setminus 1$ and $W\sse\bR$ be finite sets such that $4\in W$. We can find homomorphisms
\[\xymatrix{L\ar[r]^\lambda& F_2\ar[r]^\rho &\PSL_2(\bR)}\]
such that $1\not\in\lambda(Q)$ and such that $\tr^2\circ\rho\circ\lambda(Q)\cap W=\varnothing$, using the Projective Baumslag Lemma \ref{lem:baumslag-psl}. A Baire Category argument as in Lemma~\ref{lem:baire-simple} shows that for all countable $W\sse \bR$
there exists a faithful parabolic-free projective representation $\rho$ of $L$ such that
\[\tr^2\circ\rho(L\setminus 1)\cap W=\varnothing.\]
Since $L$ is non-Fuchsian, we see that $\rho$ is necessarily indiscrete, and so, $L$ is flexibile. However, this argument neither yields directly the \emph{liftable}--flexibility of $L$, 
nor applies to the case when $L$ is Fuchsian.

\section{Axiomatics}\label{sec:axiom}
In this section, the phrase `very general' will always be used in the topological sense.
\subsection{Tracial Structure}\label{ss:tracial}
A map defined on a group is called a \emph{class function}\index{class function}
if the map is constant on each conjugacy class.
\begin{defn}\label{defn:tracial}
Let $\FG$ be a commutative--transitive topological group,
which is either a complete metric space or a locally compact Hausdorff space.
Suppose $\varphi\co \FG\to\CS$ is a continuous surjective class function for some Hausdorff space $\CS$.
Under this setting, we say the pair $(\FG,\varphi)$ is a \emph{tracial structure}\index{tracial structure} if both of the following hold.
\be[(A)]
\item \textbf{Baumslag I.} \label{alg1} 
Let $f\co \FG\to\FG$ be a map defined by
\be[(i)]
\item
$f(\nu)=\nu^m$ for some $m\in\bZ\setminus0$, or
\item
$f(\nu) = g_1 \nu^{m_1} g_2 \nu^{m_2}\cdots g_k \nu^{m_k}$
for some $k\ge1, m_i\in\bZ\setminus0$ and $g_i\in\FG\setminus 1$.
\ee
Then for each $x\in\CS$, the set
$(\varphi\circ f)^{-1}(x)$ is nowhere dense in $\FG$.
\item \textbf{Baumslag II.} \label{alg2} 
Let $\mu$ be a maximal abelian subgroup of $\FG$ such that $\varphi(\mu)\ne\{\varphi(1)\}$.
Suppose a map  $f\co \mu\to\FG$ is defined by
\be[(i)]
\item $f(\nu)=\nu^m$ for some $m\in\bZ\setminus0$, or
\item
$f(\nu) = g_1 \nu^{m_1} g_2 \nu^{m_2}\cdots g_k \nu^{m_k}$
for some $k\ge1, m_i\in\bZ\setminus0$ and $g_i\in\FG$ such that 
for each $i$, there exists $c\in \mu$ satisfying $\varphi[c,g_ic{g_i}^{-1}]\ne\varphi(1)$.\ee
Then for each $x\in\CS$, the set
$(\varphi\circ f)^{-1}(x)\cap \mu$ is nowhere dense in $\mu$.
\ee
\ed

\begin{exmp}\label{ex:psl}
The reader is advised to keep in mind the key examples of this monograph, which are
$\FG=\PSL_2(\bK)$ for $\bK=\bR$ or $\bK=\bC$.
We let \[\varphi=\tr^2\co \PSL_2(\bK)\to \bK,\]
The conditions \textbf{Baumslag I, II} follow from Baumslag's Lemma for $\PSL_2(\bK)$, that is Lemma~\ref{lem:baumslag-psl}.  It is routine to verify the rest of the conditions.
\end{exmp}

\begin{exmp}\label{ex:pslk}
The group $\PSL_2^{(k)}(\bR)$ for $2\le k\le\infty$ admits a natural tracial structure:
\[
\tr^2\co \PSL_2^{(k)}(\bR)\to\PSL_2(\bR)\to\bR.\]
\end{exmp}

\bd\label{defn:flex-gen}
Given a tracial structure $(\FG,\varphi\co \FG\to\CS)$,
a finitely generated group $L$ is in the class $\FF(\FG,\varphi)$
if for each nonempty open set $U\sse\CS$
there exists an uncountable subset $\Lambda\sse\Hom(L,\FG)$ of faithful representations
such that $\varphi\circ\rho(L)\cap U\ne\varnothing$ for each $\rho\in \Lambda$
and such that 
\[\varphi\circ\rho(L\setminus T_L)\cap\varphi\circ\lambda(L\setminus T_L)=\varnothing\] for all distinct $\rho$ and $\lambda$ in $\Lambda$.
\ed
\begin{exmp}
A flexible group (Definition~\ref{defn:flex} and Lemma~\ref{lem:flex-count}) is in the class $\FF(\PSL_2(\bR),\tr^2)$.
To see this, recall that a flexible group $L$ admits an indiscrete faithful representation. 
If $L$ is virtually cyclic, then $L$ is actually cyclic (Example~\ref{exmp:triangle}) and its generator can assume a very general trace value in $\bR$. If $L$ contains $F_2$, then the the trace spectrum of an arbitrary faithful indiscrete representation is dense in $\bR$ (Lemma~\ref{lem:dense}).
\end{exmp}

For the rest of this section, we fix a tracial structure $(\FG,\varphi\co\FG\to\CS)$,
and simply write $\FF_\varphi=\FF(\FG,\varphi)$.
We set $\PP=\varphi^{-1}\circ\varphi(1)$ and call each element in $\PP$ as a \emph{parabolic}\index{tracial structure!parabolic}.
A set is \emph{parabolic-free}\index{tracial structure!parabolic-free set} if it does not contain a nontrivial parabolic element. A subgroup $H\le \FG$ will be called \emph{non-parabolic}\index{tracial structure!non-parabolic set} if $H\not\sse\PP$.
The Combination Theorem for Flexible Groups (Theorem~\ref{thm:main-comb}) generalizes as follows.

\begin{thm}\label{thm:comb-gen}
\be
\item\label{inf-cyc-gen}
Infinite cyclic groups are in $\FF_\varphi$.
\item\label{part:free-gen} If $A,B\in\FF_\varphi$, then $A\ast B\in \FF_\varphi$.
\item\label{part:amalgam-gen}
Let $A,B\le L$ and $C\le A\cap B$ satisfy that \[L=\form{A,B}\in\FF_\varphi.\]
If $C$ is nontrivial malnormal abelian in $A$ and in $B$,
then $A\ast_C B\in\FF_\varphi$.
\item\label{part:hnn-gen}
Let $A,C\le L$ and $s\in L$ satisfy that  
\[L=\form{A,s}\in\FF_\varphi\]
and that $C$ and $C^s$ are nontrivial malnormal abelian subgroups of $A$. Assume either 
\begin{itemize}
\item $s\in Z(C)$, or
\item $C$ and $C^s$ are not conjugate in $A$.
\end{itemize}
Then  $A\ast_{\Inn(s)\co C\to C^s}\in\FF_\varphi$.
\ee
\end{thm}

The proof of Theorem~\ref{thm:comb-gen} is almost identical to that of Theorem~\ref{thm:main-comb},
so will be only sketched. We first point out two key lemmas below needed for the proof.

\begin{lem}\label{lem:tr-prelim}
\be
\item\label{tr-inf}
A very general $\nu\in\FG$ is non-parabolic and has infinite order.
\item\label{tr-closed}
If $\mu$ is a maximal abelian subgroup of $\FG$,
then $\mu$ and $N(\mu)$ are closed and nowhere dense in $\FG$;
in particular, they are Baire spaces.
\ee
\end{lem}
\bp
(\ref{tr-inf}) Apply \textbf{Baumslag I} for $f(\nu)=\nu^m$ for $m\in\bZ\setminus0$.

(\ref{tr-closed}) Fix $c\in \mu\setminus1$. As was seen in Lemma~\ref{lem:malnormal}, the subgroups $\mu$ and $N(\mu)$ are algebraically defined as
\[
\mu=\{g\in\FG\co [c,g]=1\},\quad
N(\mu) = \{ g\in\FG \co [c,gcg^{-1}]=1\}.\]
By \textbf{Baumslag II}, both are closed and nowhere dense in $\FG$.
\ep

As before, the process of pulling-apart subgroups will play a crucial role for us.

\begin{lem}[General Free Product Pulling-Apart Lemma]\label{lem:free-gen}
\be
\item
Let $A$ and $B$ be countable subgroups of $\FG$,
and let $W\sse\CS$ be a countable set.
Then for a very general $\nu\in\FG$, 
the group $L_\nu=\form{A,B^\nu}$ is isomorphic to $A\ast B$
and furthermore, 
\[\varphi\left(L_\nu\setminus (A\cup B^\nu)^{L_\nu}\right)\cap W=\varnothing.\]
\item
Let $A$ be a countable subgroup of $\FG$,
and let $W\sse\CS$ be a countable set.
Then for a very general $\nu\in\FG$, 
the group $L_\nu=\form{A,\nu}$ is isomorphic to $A\ast \bZ$
and furthermore, 
\[\varphi\left(L_\nu\setminus A^{L_\nu}\right)\cap W=\varnothing.\]
\ee
\end{lem}

\bp
Let $V={A\cup B}$ in (1), and $V=A$ in (2).
We choose $\mu$ to be a non-parabolic maximal abelian group 
such that $[c,gcg^{-1}]$ is not parabolic for all $c\in\mu\setminus1$ and $g\in \form{V}\setminus1$.
here, the hypothesis \textbf{Baumslag I} is used.
The rest of the proof is very similar to the proof of Lemma~\ref{lem:free}.\ep

Recall that for $C\le\FG$ we write
\[
Z(C) = \{\nu\in\FG\co [c,\nu]=1\text{ for all }c\in C\}.\]
\begin{lem}[General Amalgamated Pulling-Apart Lemma]\label{lem:pull-apart-gen}
For a countable subgroup $L\le\FG$
and for a countable set $W\sse\CS$,
assume one of the following.
\be
\item
For some subgroups $A,B,C\le L$, we have that $L=\form{A,B}$
and that $C$ is a malnormal nontrivial abelian subgroup of $A$ and  also of $B$. 
For each $\nu\in Z(C)$ we let $\rho_\nu$ be the map uniquely determined by the following commutative diagram
\[\xymatrix{
& A\ar[rd]\ar[rrrd]^{\Id_A}\\
C\ar@{^(->}[ru] \ar@{^(->}[rd]& & L^*=A\ast_C B\ar[rr]^<<<<<<<{\rho_\nu} && {\FG}\\
& B\ar[ru]\ar[rrru]_{\Inn(\nu)\circ\Id_B}}\]
We further assume $V=A\cup B$ satisfies 
$\varphi(1)\not\in\varphi(V\setminus 1)$.
\item
For some subgroups $A,C\le L$ and an element $s\in L$,
we have that $L=\form{A,s}$ and that $C$ and $C^s$ are malnormal nontrivial abelian subgroups of $A$. Furthermore, we assume that either
\begin{itemize}
\item $s\in Z(C)$, or
\item $C$ and $C^s$ are not conjugate in $A$.
\end{itemize}
For each $\nu\in Z(C)$ we let $\rho_\nu$ be the map uniquely determined by the following commutative diagram
\[
\xymatrix{
& A\ar[rd]\ar[rrrd]^{\Id_A}\\
C\ast C^s\ar@{^(->}[ru] \ar[rd]& & L^*=A\ast_{\Inn(s)\co C\to C^s}\ar[rr]^<<<<{\rho_\nu} && {\FG}\\
& C\ast\form{ s^*}\ar[ru]\ar[rrru]_{  s^*\mapsto \nu s}
}\]
Here $s^*$ denotes the stable generator of $L^*$.
We further assume $V=A$ satisfies 
$\varphi(1)\not\in\varphi(V\setminus 1)$.
\ee
Under the assumption (1) or (2), 
for a very general $\nu\in Z(C)$
the map $\rho_\nu$ is faithful  and 
\[\varphi \circ\rho_\nu\left(L^*\setminus V^{ L^*}\right)\cap W=\varnothing.\]
\end{lem}

The proof of the above lemma is omitted, as it is identical to that of Lemma~\ref{lem:pull-apart-psl},
after substituting the phrase ``closed nowhere dense'' for ``finite''.

\bp[Proof of Theorem~\ref{thm:comb-gen}]
For the proof, let us fix  $U\sse \FG,W\sse\CS$ such that $U$ is open and $W$ is countable.
Furthermore, we assume $\varphi(1)\in W$.
We use an alternate characterization of flexibility as in Lemma~\ref{lem:flex-count}.

(\ref{inf-cyc-gen})
In Lemma~\ref{lem:pull-apart-gen} (2), we simply put $A=L=\{1\}$ and $\nu\in\varphi^{-1}(U)$.
In Lemma~\ref{lem:tr-prelim}, we saw a very general $\nu$ has infinite order.

(\ref{part:free-gen})
Suppose $\alpha\co A\to\FG$ and $\beta\co B\to\FG$ be faithful representations
satisfying the conditions (of $\rho$) in Definition~\ref{defn:flex-gen}.
Then Lemma~\ref{lem:free-gen} (1) applies with 
 $L=\form{\alpha(A),\beta(B)}$.

{(\ref{part:amalgam-gen})}
Suppose we have a faithful representation 
$\rho\co L\to\FG$
such that \[\rho(L)\cap U\ne\varnothing,\quad\rho(L\setminus T_L)\cap W=\varnothing.\]
Fix $g\in L$ such that $\rho(g)\in U$.
Let $\mu\le\FG$ be the maximal abelian subgroup containing $\rho(C)$.
For each $\nu\in\mu$, we have a map \[\rho_\nu\co L^*=A\ast_C B\to\FG\]
as in the proof of the Pulling Apart Lemma (Lemma~\ref{lem:pull-apart-psl}).
Then for a very general $\nu\in\mu$, the representation $\rho_\nu$ satisfies
\[\rho_\nu(L\setminus T_L)\cap W=\varnothing.\]
In particular, one can find $\nu\in\mu$ such that $\rho_\nu(g)\in U$.

We omit the proof of (4), as it is very similar to the above.
\ep

By Lemma~\ref{lem:igd} and Theorem~\ref{thm:comb-gen},
we have the following.
\begin{cor}
Every nontrivial limit group is in $\FF_\phi$.\end{cor}

\subsection{UV--structure}
We can also generalize  Lemma~\ref{lem:uv1} to the context of an arbitrary algebraic group in order to control certain ``spectra'' of representations.
Suppose we want to find many representations $\rho\co G\to\PSL_2(\bR)$ 
which map certain fixed elements to irrational rotations.
Note that neither the set $\EE=(\tr^2)^{-1}(0,4)\sse\PSL_2(\bR)$ nor the map \[\rot\co \PSL_2(\bR)\to \bR/\bZ\] is algebraic. So we use a parametrization
by an algebraic group:
\[\xymatrix{\SO(2)\times\PSL_2(\bR) \ar@{->>}[r]^<<<<{\xi} \ar[rd]^{\eta} &\EE\ar[d]^{\bar \eta}\ar@{^(->}[r] & \PSL_2(\bR)\\& \SO(2)}\]
where 
\[\xi(x,y)=x^y,\quad \eta(x,y)=x,\quad \bar \eta(x) = \Rot(2\pi\rot(x)).\]

\bd\label{defn:uv}
\be
\item
Let $\mathfrak{G}$ be an algebraic group and $k>0$.
For each $i=1,2,\ldots,k$, we assume there is a commutative diagram
\[\xymatrix{U_i \ar@{->>}[r]^<<<<{\xi_i} \ar[rd]^{\eta_i} & \xi_i(U_i)\ar[d]^{\bar \eta_i}\ar@{^(->}[r] & \mathfrak{G}\\& V_i},\]
such that
\be[(i)]
\item
$U_i$ and $V_i$ are algebraic sets;
\item
$\xi_i$ and $\eta_i$ are polynomial maps;
\item
$\mathfrak{G}\setminus\bigcup_{i=1}^k \xi_i(U_i)$ is an algebraic set.
\ee
In this situation, we say $\mathfrak{G}$ is equipped with a \emph{UV--structure}\index{UV--structure}.
\item
Assume (1). Suppose  $G$ is a finitely generated group,
and \[A=\{A_1,\ldots,A_k\}\] is a collection of finite subsets $A_i\sse G$.
Fix algebraic subsets  
\[W_i\sse \prod_{A_i} V_i\] for each 
$i=1,2,\ldots,k$,
and put \[W=\{W_1,\ldots,W_k\}.\]
We say a representation $\rho\co G\to\mathfrak{G}$ is \emph{UV--compatible (with respect to $A$ and $W$)}\index{UV--structure!UV--compatible} if the following hold.
\begin{itemize}
\item
$\rho(G)\sse \bigcup_{i=1}^k \xi_i(U_i)$.
\item
 $\rho(A_i)\sse \xi_i(U_i)$ 
 and
 $\bar \eta_i\circ\rho\restriction_{A_i}\in W_i$ 
 for each $i$.
\end{itemize}
\ee
\ed

\begin{rem}
\be
\item
We assume neither that $\xi_i(U_i)$ is algebraic nor that $\bar{\eta_i}$ is a polynomial map. The map $\bar{\eta_i}$ merely makes the diagram commute.
\item
If (i) and (ii) are satisfied, we can force (iii) by setting \[U_{k+1}=\xi_{k+1}(U_{k+1})=V_{k+1}=\mathfrak{G}.\]
\item
If $V_i$'s are not specified, then we will assume $V_i=\{0\}$.
\ee
\end{rem}

\begin{exmp}\label{exmp:uv1}
Let us define a UV--structure on $\mathfrak{G}=\PSL_2(\bR)$ by:
\[
U_1=\SO(2)\times\PSL_2(\bR), V_1=\SO(2), U_2=\exp(\bR)\times\PSL_2(\bR), V_2=\exp(\bR),\]
and $\eta_i(x,y)=x$ and $\xi_i(x,y)=x^y$ for $i=1,2$.
Note that $\xi_1(U_1)$ is the set of elliptics,
and $\xi_2(U_2)$ is that of hyperbolics.
Then $\bar \eta_1,\bar \eta_2$ are well-defined by
\[\bar \eta_1(x)=\Rot(2\pi\rot(x))\] and \[\bar \eta_2(x)=\exp \ell(x).\]
For a representation $\rho\in\Hom(G,\mathfrak{G})$, 
the set \[\bar\eta_1(\rho(G)\cap\xi_1(U_1))\] is the rotation spectrum of $\rho$.
We may call \[\bar\eta_2(\rho(G)\cap\xi_2(U_2))\] as the \emph{(hyperbolic) length spectrum}\index{length spectrum} of $\rho$.
The set of parabolics is \[\{1\}\cup\PSL_2(\bR)\setminus \left(\xi_1(U_1)\cup \xi_2(U_2)\right),\] which is algebraic.
\end{exmp}

Then Lemma~\ref{lem:uv1} generalizes as follows.

\begin{lem}\label{lem:uv-alg}
Let us fix an algebraic group $\mathfrak{G}$
equipped with a UV--structure,
and a finitely generated group $G$. Let $A$ and $W$ be as in Definition~\ref{defn:uv} and assume each $A_i$ is nonempty.
Let
\[\{\rho_n\co G\to\mathfrak{G}\}_{n\ge1}\] 
be 
a stably injective sequence of UV--compatible (with respect to $A$ and $W$) representations.
Suppose 
for each \[i\in\{1,\ldots,k\}\] and for each $a\in A_i$,
we have  \[\bar{\eta_i}\circ \rho_m(a)\ne \bar{\eta_i}\circ \rho_n(a)\] whenever $m\ne n$.
Then there exists an uncountable collection $\Lambda$ of faithful UV--compatible (with respect to $A$ and $W$) representations $G\to\mathfrak{G}$  such that 
 \[\{\bar{\eta_i}\circ \lambda(a)\co \lambda\in\Lambda\}\] is uncountable for each $i$ and $a\in A_i$.
\end{lem}

We omit the proof, which is an easy variation from Lemma~\ref{lem:uv1}.

\begin{exmp}
Let \[H=\form{a,b,t\mid [[a,b],t]=1}.\]
Recall the UV--structure of $\PSL_2(\bR)$ in Example~\ref{exmp:uv1}.
In Example~\ref{exmp:genus-two},
we exhibited a stably injective sequence \[\phi_p\co H\to\PSL_2(\bR)\]
such that \[\rot\circ\phi_p(t)=1/p\] for $p\ge2$.
We can further require that $\phi_p(a)$ is a hyperbolic element of length at most $1/p$ by considering the moduli space of a torus with one cone point.

So, Lemma~\ref{lem:uv-alg} implies that there exists an uncountable set \[\Lambda\sse\Hom(H,\PSL_2(\bR))\] of faithful parabolic-free representations
with 
all distinct rotation spectra and
with all distinct length spectra.
\end{exmp}

\subsection{Combination Theorem for Smooth Actions}\label{combsmooth}
The last combination theorem \ref{thm:circle baire1} below is one in the smooth category and its proof follows the lines of the usual Baire category argument for the real line.  Let $\Diff^{\infty}(M)$ denote the group of $C^{\infty}$ orientation preserving diffeomorphisms of a one-manifold $M$. Recall that $G\le\Diff^{\infty}(S^1)$ is {\it fully supported} if 
the fixed point set of each element $g\in G\setminus 1$ has empty interior.
A group \[G\sse \Diff^{\infty}(S^1)\] \emph{has a free orbit}\index{free orbit}
if $G$ acts freely on some $G$--orbit.
Following the strategy of \cite[Corollary 2.6]{Grabowski1988} and \cite{Ghys2001}, we will use the  Baire category argument in the theorem below. Note the similarity to Lemma~\ref{lem:free}.

\begin{thm}\label{thm:circle baire1}
Suppose $G$ and $H$ are countable, fully supported subgroups
of the group $\Diff^{\infty}(S^1)$.
Then for a very general choice of \[(\psi,x)\in\Diff^{\infty}(S^1)\times S^1,\]
the group $\form{G,H^\psi}$ is isomorphic to $G\ast H$
and has $\form{G,H^\psi}\cdot x$ as a free orbit.
In particular, $G\ast H$ admits a faithful $C^\infty$ action on $S^1$.
\end{thm}

Theorem \ref{thm:circle baire1} seems to be a folklore theorem, the details of whose proof we record here for the convenience of the reader.
The theorem does not hold without the hypothesis of being fully supported, 
even for $C^2$; see~\cite{KKFreeProd2017}.

In the proof below, for $w\in G\ast H$, we let $\|w\|$ denote the word--length of $w$.
\bp
For $\psi\in\Diff^{\infty}(S^1)$, 
define 
$\rho_\psi\co G\ast H\to \form{G,H^\psi}$
 by the following:
\[
\rho_\psi\restriction_G=\Id_G,
\quad 
\rho_\psi(h)=h^\psi\text{  for }h\in H.\]

Let \[X=\Diff^{\infty}(S^1)\times S^1.\] Since $\Diff^{\infty}(S^1)$ is a Frech\'et space~\cite{Herman1979}, we see $X$ is Baire.
For each \[w\in G\ast H\setminus 1\] we define a closed set
\[
X_w = \{(\psi,x)\in X \co \rho_\psi(w)(x)=x\}.\]
By the Baire Category Theorem, it suffices to show the following.

\begin{claim*}
$X_w$ is nowhere dense for $w\ne1$.\end{claim*}

Let $w$ be a shortest word such that $X_w$ has a nonempty interior.
Pick a nonempty open set $V\sse X_w$.
If \[w=g\in G\setminus 1,\] then we have \[X_g=\Diff^{\infty}(S^1)\times\Fix g\] is nowhere dense.
If \[w=h\in H\setminus 1,\] then we obtain
\[
X_h = \{(\psi,x)\in X \co \psi^{-1}h\psi(x)=x\}
=\bigcup\left\{
\{\psi\}\times \psi^{-1}\Fix h \co \psi\in \Diff^{\infty}(S^1)\right\}\]
is nowhere dense. So, we assume that $\|w\|\ge 2$.

For $u,v\in G\ast H$, we let
\[
Y_{u,v}=\{(\psi,x)\in X\co \rho_\psi(u)\rho_\psi(v)(x)=\rho_\psi(v)(x)\}.\]
There exists a continuous map $\Phi_v\co X_u\to Y_{u,v}$ defined by 
\[(\psi,x)\mapsto(\psi,\rho_\psi(v)^{-1}(x)).\] As $\Phi_{v^{-1}}\circ\Phi_v$ is the identity, we have a homeomorphism $X_u\to Y_{u,v}$.
By minimality, the set
\[Y = \bigcup\{ Y_{u,v} \co \|u\|,\|v\|<\|w\|,\text{ and }u\ne1\}\]
is a countable union of closed nowhere dense sets, and 
hence has empty interior.
So $V_0:=V\setminus Y$ is a nonempty subset of $X_w$.

For each $u\in G\ast H$, we have 
\[X_{uwu^{-1}}=Y_{w,u^{-1}}\approx X_w.\] 
So, we may assume
\[
\rho_\psi(w) = \psi^{-1}h_k\psi g_k\cdots \psi^{-1}h_1\psi g_1\]
possibly after a conjugation. In particular, $\|w\|=2k$.
Let us pick an arbitrary $(\psi,x_0)\in V_0$.
Define \[x_1,\ldots,x_{2k}\] by \[x_{2i+1}=g_i(x_{2i})\] 
and by \[x_{2i}=\psi^{-1}h_i\psi(x_{2i-1}).\] By assumption, 
$x_{2k} = 
\rho_\psi(w).x_0=x_0$. Moreover, the points \[x_0,x_1,\ldots,x_{2k-1}\] are all distinct
since $V_0\cap Y_{u,v}=\varnothing$ for $\|u\|,\|v\|<\|w\|$. 

Note that \[\psi(x_0)=\psi(x_{2k}) =h_k\psi(x_{2k-1}).\]
Choose a small closed interval $J$ containing $x_0$ such that $x_i\not\in J$ for each $0< i<2k$. 
By perturbing $\psi$, we can pick $(\phi,x_0)\in V$ such that
\[\phi\restriction_{S^1\setminus J}=\psi\restriction_{S^1\setminus J}\]
and
\[\phi(x_0)\ne h_k\psi(x_{2k-1}).\]

If $i< k$, then 
\[
\phi^{-1}h_i\phi g_i(x_{2i-2})
=\phi^{-1}h_i\phi (x_{2i-1})
=\phi^{-1}h_i\psi (x_{2i-1})
=\psi^{-1}h_i\psi (x_{2i-1})=x_{2i}\]
since $h_i\psi (x_{2i-1})\not\in \psi(J)$ and since 
\[\phi^{-1}\restriction_{S^1\setminus \psi(J)}
=\psi^{-1}\restriction_{S^1\setminus \psi(J)}.\] Hence,
\begin{align*}
\rho_\phi(w)(x_0)&=\prod_{i=k}^1 \phi^{-1}h_i\phi g_i(x_0)
=
\prod_{i=k}^2 \phi^{-1}h_i\phi g_i(x_2)=\cdots\\
&= \phi^{-1}h_k\phi g_k(x_{2k-2})= \phi^{-1}h_k\psi (x_{2k-1})\ne x_0.
\end{align*}
So $\phi\not\in X_w$. We have a contradiction since we have assumed $(\phi,x_0)\in V$.
\ep

Theorem~\ref{thm:circle baire1} implies that if $G\le \Diff^{\infty}(S^1)$ is fully supported then $G\ast \bZ$ embeds into $\Diff^{\infty}(S^1)$ as a subgroup admitting almost every point as a free orbit.
One actually has a genericity statement as below.

\begin{cor}\label{cor:circle baire1}
Let $G\le \Diff^{\infty}(S^1)$ be countable and fully supported.
Then for a very general choice of \[(\psi,x)\in\Diff^{\infty}(S^1)\times S^1,\] 
the group $\form{G,\psi}$ is isomorphic to $G\ast \bZ$
and has $\form{G,\psi}\cdot x$ as a free orbit.
\end{cor}

\bp
Recall our notation $\bZ=\form{s}$. 
If $w\in G\ast \bZ$, we let $\|w\|$ denote the word length in the following sense
\[
\|w\|=\min\{\ell\co  w=s_1s_2\cdots s_\ell,\quad s_i\in G\cup\{s\}^{\pm1}\}.\]
For $\psi\in\Diff^{\infty}(S^1)$, 
define 
\[\rho_\psi\co G\ast\form{s}\to \Diff^{\infty}(S^1)\] by
$\rho_\psi\restriction_G=\Id_G$
and $\rho_\psi(s)=\psi$.
Then we can proceed very similarly to the proof of Theorem~\ref{thm:circle baire1}, so we omit the details.\ep

\section{Mapping Class Groups}\label{sec:mcg}
In this section, we shift the focus to exotic actions of mapping class groups on the circle, where here ``exotic" means ``not conjugate to Nielsen's standard action" (see~\cite{HT1985,CB1988} for detailed discussions of Nielsen's action).
We first discuss actions of fibered, hyperbolic $3$--manifold groups on $S^1$, in relation to Nielsen's action.

\subsection{The Universal Circle and Nielsen's Action}
Let us fix a marked point $x\in S_g$,
and a distinguished lift $\tilde x$ of $x$ in the universal cover $\bH^2$ of $S_g$. $\Mod(S_g,x)$ will denote the mapping class group of the marked surface $(S_g,x)$ , i.e.\ the group of (relative) isotopy classes of diffeomorphisms of $S_g$ fixing $x$.

For each homeomorphism $\phi$ of  $S_g$ fixing $x$,  we choose the preferred lift $\tilde\phi$ of $\phi$ to $\bH^2$ which fixes $\tilde x$.  There exists a continuous extension of $\tilde\phi$ to $\partial \bH^2\approx S^1$, which depends only on the mapping class $[\phi]\in\Mod(S_g,x)$.  The resulting action is called \emph{Nielsen's action}\index{Nielsen's action} and denoted as \[\nu_g\co \Mod(S_g,x)\to\Homeo_+(S^1).\]

Let us now pick a pseudo-Anosov homeomorphism $\Psi$ on $S_g$.
We will denote  the mapping torus of $S_g$  with the monodromy map $\Psi$ as 
\[ M = S_g\yt{\times}_\Psi S^1.\] 
Then $M$ is a closed hyperbolic $3$--manifold fibering over the circle,
whose topological type is determined by the mapping class $\psi=[\Psi]\in\Mod(S_g)$. We often abuse the notation and use the mapping class $\psi$ instead of the homeomorphism $\Psi$.

The universal cover $\yt{M}$ of $M$ is naturally homeomorphic to \[\yt{S_g}\times\bR\cong\bH^2\times\bR,\] where the $\bR$ factor gives rise to a natural foliation of $\yt{M}$ by lines. 
There is a lamination $\lambda\sse S_g$ preserved by $\Psi$, 
and $\lambda$ gives a suspended codimension one foliation $\Lambda\sse M$ which is transverse to $S_g$. This lamination in turn lifts to a lamination $\yt{\Lambda}\sse\yt{M}$ which is transverse to $\yt{S_g}$.

The identification of $\yt{S_g}$ with $\bH^2$ gives rise to a natural faithful action of $\pi_1(M)$ on $S^1\cong\partial\bH^2$. This action extends the given Fuchsian action of $\pi_1(S_g)$ on $\partial\bH^2$ by $\form{\psi}\cong\bZ$; see~\cite{FM2012}, for instance.
This group $\form{\psi}$ acts on $S^1$ with a finite even number of fixed points (perhaps after replacing $\Psi$ by a nonzero power), whose dynamics are alternately attracting and repelling.
For compactness of notation, this action will be called the \emph{universal circle action}\index{universal circle action} of $\pi_1(M)$ and denoted as
\[\rho_{\mathrm{u}}\co \pi_1(M)\to\Homeo_+(S^1).\]

We have a diagram as below, the commutativity of whch follows from Proposition \ref{prop:universal circle}:
\begin{equation}\label{eq:univ}
\xymatrix{
&& \pi_1(M,x)\ar[r]\ar@{^(->}[d] & \bZ\ar@{^(->}[d]_{1\mapsto\psi}\ar[rd]\\
1\ar[r] & \pi_1(S_g,x)\ar[rd]^{\mathrm{Fuchs.}}\ar[r]^{\mathrm{Push.}}\ar[ru] & \Mod(S_g,x)\ar[r]\ar@{^(->}^{\nu_g}[d] & \Mod(S_g)\ar[r] & 1\\
&& \Homeo_+(S^1)
}
\end{equation}
Here, Push and Fuchs respectively denote the point pushing map
and the prescribed Fuchsian representation.
Group theoretically, the middle row is isomorphic to 
\[
1\to \pi_1(S_g,x)\to \Aut(\pi_1(S_g,x))\to\Out(\pi_1(S_g,x))\to 1.\]

The upper row is the split short exact sequence coming from the fibering.
The discussion in the previous paragraph shows that the square in the diagram is a fiber product (Section \ref{sec:prelimcirc}), all considered as subgroups of $\Homeo_+(S^1)$.
Hence we deduce the following, which is probably well--known (cf.~\cite{DKL14}, for instance, as well as~\cite{Ghys2001,CD2003IM,CalegariUniversal}).

\begin{prop}\label{prop:universal circle}
The universal circle action $\rho_{\mathrm{u}}$
extends to Nielsen's action $\nu_g$. 
\end{prop}



Since the Fuchsian action of $\pi_1(S_g)$ on $S^1$ is minimal, so is the action $\nu_g$.
As an element of $\pi_1(M)$, the mapping class $\psi\in\Mod(S_g,x)$ will be identified with the stable letter of the fibration of $M$ over $S^1$. In other words, we write
\[
\pi_1(M)=\pi_1(S_g)\rtimes\form{\psi}.\]
The conjugating action of the stable letter $\psi$ on $\pi_1(S_g)$ comes from 
the Dehn--Nielsen--Baer isomorphism, which says
\[
\psi\in\Mod(S_g)\cong \Out(\pi_1(S_g)).\]

\subsection{Exotic Mapping Class Group Actions}
Let $S\nought$ be a compact hyperbolic surface with \emph{nonempty} geodesic boundary, and let $x$ be a marked point in the interior of $S\nought$. We put $S'=S\nought\setminus \{x\}$.
In this section, we exhibit an inequivalent pair of \emph{faithful} representations \[\Mod(S\nought,x)\to\Homeo_+(S^1).\]

We note that the mapping class group of a non-closed surface admits many other interesting actions on $\bR$ or $S^1$, which we will not consider here. See~\cite{HKM2007IM,BowditchSakuma}, for instance. Remarkably, however, for a closed surface with a single marked point, there are no
exotic actions of the full mapping class group $\Mod(S_g,x)$ on the circle. Indeed, any such action is semi-conjugate to Nielsen's action $\nu_g$, by a recent result of Mann--Wolff~\cite{MannWolff18}. For finite index subgroups of $\Mod(S_g,x)$, no such characterization of circle actions is known.

Let us denote by $p\co \tilde S\nought\to S\nought$ the universal cover of $S\nought$. We can identify $\tilde S\nought$ with a convex subset of the hyperbolic plane.
The argument in the previous subsection also applies to $S$, and hence we have Nielsen's action:
\[\nu\co \Mod(S\nought,x)\to \Homeo_+(S^1).\]
This action is minimal on the limit set, which is a Cantor set.
In particular, $\nu$ is not semi-conjugate to the trivial action (Corollary~\ref{cor:fin-orbit}).

On the other hand,
we have Thurston's faithful action 
\[\tau\co \Mod(S')\to \Homeo_+(\bR).\]
This action is obtained by considering lifts of self-homeomorphisms of $S'$ to the universal cover $\tilde S'$. 
Such lifts are chosen to fix one distinguished lift  $\tilde\beta$ of $\beta\sse\partial S'$, so that $\Mod(S')$ acts on \[(\partial\tilde S'\setminus\tilde\beta)\cup \partial\pi_1(S',x)\approx \bR.\]  See \cite{SW2000,Ghys2001} for details, and~\cite{Dehornoy1994TAMS} for related ideas.
We have an embedding
\[i\co \Homeo_+(\bR)\to \Homeo_+(S^1)\] as a subgroup with a global fixed point.
Let $\bar\tau$ be the map obtained by composing $i\circ \tau$ with 
the isomorphism $\Mod(S\nought,x)\cong \Mod(S')$.
Since $\bar\tau$ is semi-conjugate to the trivial action,
we have the following.
\begin{prop}\label{prop:mcg}
The faithful actions 
\[
\nu,\bar\tau\circ\alpha\co \Mod(S\nought,x)\to \Homeo_+(S^1)\] are not semi-conjugate 
for all 
$\alpha\in\Aut(\Mod(S\nought,x))$.
\end{prop}

In other words, $\nu$ and $\bar\tau$ do not merely fail to be semi-conjugate, but actually are inequivalent. 
In the spirit of Section \ref{sec:flex}, we ask:

\begin{que}
Do mapping class groups of bounded compact hyperbolic surfaces admit uncountably many inequivalent actions on the circle?
\end{que}

Here, the existence of a boundary component is a necessary hypothesis in light of~\cite{MannWolff18}.

\section{Zero rotation spectrum and Teichm\"uller theory}\label{sec:rot spec}
In this section, we consider free group and surface group actions on the circle, and develop conditions under which the equivalence class of an action is determined by the rotation spectrum, and when the semi-conjugacy class of the action is determined by the marked rotation spectrum.

In the case of indiscrete representations of groups into $\PSL_2(\bR)$, there is a lack of a geometric interpretation of such representations which is as well-developed as Teichm\"uller theory in the case of discrete representations. In this section, we consider the degree to which marked rotation spectrum can supplant marked length spectrum as a (sometimes nearly complete) semi-conjugacy invariant.

Section \ref{sec:teich} (see also Theorem \ref{thm:rigid})
establish rigidity results. Here we conclude that there exist only finitely many semi-conjugacy classes of non-elementary faithful representations of a given group into $\PSL_2(\bR)$ with zero rotation spectrum (see also \cite{Mann2015IM}). 
The remaining subsections focus more on flexibility results. 
We first describe properties of a linear (that is, factoring through some finite--dimensional connected Lie group) action on the circle in Section~\ref{ss:lie}.
In Sections \ref{ss:3mfd}, we analyze the circle actions of free and surface subgroups of fibered hyperbolic $3$--manifolds.
In Section \ref{ss:smooth-free}, we construct \emph{smooth} nonlinear actions of free groups.


\subsection{Rigidity of Projective Actions}\label{sec:teich}

Let us gather some facts from Teichm\"uller Theory which allow us to analyze projective actions of free and surface groups with zero marked rotation spectrum.

\begin{lem}(cf. \cite{Mann-hb,FM2012})\label{lem:teichmueller}
Let $S$ be an orientable surface and let \[\phi_1,\phi_2\colon \pi_1(S)\to\PSL_2(\bR)\] be discrete and faithful representations of $\pi_1(S)$ corresponding to complete finite volume hyperbolic structures on $S$, with a fixed orientation. Then the corresponding actions of $\phi_1$ and $\phi_2$ on $S^1$ are conjugate in $\Homeo_+(S^1)$.
\end{lem}
\begin{proof}
The two representations $\phi_1$ and $\phi_2$ of $\pi_1(S)$ (considered up to conjugacy in $\PSL_2(\bR)$) correspond to two points $X_1$ and $X_2$ in the Teichm\"uller space $\mathcal{T}(S)$ or the Teichm\"uller space $\mathcal{T}(\overline{S})$ with the opposite orientation. 
It suffices to show that if \[X_1,X_2\in \mathcal{T}(S),\] then the corresponding representations are conjugate in $\Homeo_+(S^1)$. Between any two points in $\mathcal{T}(S)$, there is a quasi-conformal map taking one hyperbolic structure to the other. Writing $h$ for such a map between $X_1$ and $X_2$, we lift $h$ to the universal covers of $X_1$ and $X_2$ respectively. If $S$ is a closed surface then $h$ induces a quasi--isometry $\yt{h}$ of the universal covers of $X_1$ and $X_2$, which are both identified with $\bH^2$. Thus, $\yt{h}$ induces a homeomorphism $\partial\yt{h}$ of $S^1$ which conjugates the actions of $\phi_1$ and $\phi_2$ on $S^1$.

If $S$ is not closed then $h$ need not induce a quasi--isometry because $S$ is non-compact. In this case, we consider $h$ as a self-homeomorphism of $S$ and note that $h$ induces a $\pi_1(S)$--equivariant homeomorphism between fundamental domains for $X_1$ and $X_2$ in $\bH^2$. By cutting out small neighborhoods of the cusps of $S$, we obtain a sequence of nested compact surfaces $\{S^i\}_{i\in\bN}$ whose union is $S$. Writing $X_1^i$ for the restriction of $X_1$ to $S^i$ and $X_2^i$ for the image of $X_1^i$ under $h$, we get a sequence of homeomorphisms \[\{h^i\colon X_1^i\to X_2^i\}_{i\in\bN}\] which lift to $\pi_1(S)$--equivariant homeomorphisms between fundamental domains for $X_1^i$ and $X_2^i$ in $\bH^2$. Moreover, these homeomorphisms are compatible with respect to inclusion, in the sense that if $i\leq j$ then $h^j=h^i$ when restricted to $S^i$; see~\cite[Section 8.2.7]{FM2012} for more details.

The total preimages $\yt{X_1^i}$ and $\yt{X_2^i}$ of $X_1^i$ and $X_2^i$ in $\bH^2$ (which are identified with the corresponding universal covers) have boundaries which are naturally identified with $S^1$, after including the limit sets. In the limit as $i$ tends to infinity, both boundaries are identified with $\partial\bH^2$, and \[h=\lim_{i\to\infty}h^i\] induces a homeomorphism of $S^1$ conjugating the two actions of $\pi_1(S)$.
\end{proof}

\begin{rem}
When $S$ is closed, representations corresponding to points in the two different Teichm\"uller spaces  $\mathcal{T}(S)$ and $\mathcal{T}(\overline{S})$ are not conjugate in $\Homeo_+(S^1)$ because they have different Euler numbers.
\end{rem}

Recall our convention that a \emph{surface group} means the fundamental group of a closed orientable hyperbolic surface.
The following characterizes projective surface group actions with rotation spectrum identically zero, up to conjugacy. The proof is a fairly easy combination of standard facts from Teichm\"uller theory:

\begin{thm}\label{thm:rigid}
If $L$ is a finitely generated free group or a surface group,
then there exist only finitely many conjugacy classes of faithful projective actions $\phi$ of $L$ 
with rotation spectrum $\{0\}$.
Such a representation $\phi$ is necessarily discrete.
\end{thm}

\begin{rem}
Our proof below implies that there are exactly two conjugacy classes if $L$ is a surface group.\end{rem}
\begin{proof}[Proof of Theorem~\ref{thm:rigid}]
Let us first assume $L$ is a surface group.
The fact that $\phi(L)$ is a discrete cocompact subgroup of $\PSL_2(\bR)$ is an immediate consequence of Lemma \ref{lem:dense}. By Lemma \ref{lem:teichmueller}, we have that two discrete faithful surface group representations into $\PSL_2(\bR)$ with equal Euler numbers give rise to conjugate actions on $S^1$. The Euler number of such a representation is a conjugacy invariant in $\Homeo_+(S^1)$, and it takes on exactly two values for discrete surface group representations~\cite{Goldman1988IM}.

We now assume $L$ is free.
We perform much of the same analysis as for surface groups.
As the rank--one case is trivial,
we suppose $L$ is non-elementary. 
Let $\phi$ is a representation in 
\[\Phi=\{\phi\in\Hom\left(L,\PSL_2(\bR)\right)\co \rot\circ\phi(L)=\{0\}\text{ and }\phi\text{ is faithful}\}.\]
Lemma \ref{lem:dense} implies the discreteness of $\phi$.
In particular, $S=\bH^2/\phi(L)$ is a non-compact hyperbolic surface
with cusps or funnels (i.e. flaring ends).
Denote by $S_0$ the (compact) convex core of $S$, and $S_1$ be the double of $S_0$ along the geodesic boundary  components.

Since we have \[\chi(S) = 1-\mathrm{rank}(L),\] there are finitely many possible homeomorphism types of the surface $S_1$.
Each homeomorphism type corresponds to at most two conjugacy classes (depending on the choice of the orientation)
of actions by Lemma~\ref{lem:teichmueller}.
In other words, if $\phi_1$ denotes the Fuchsian representation of $\pi_1(S_1)$ determined by $\phi$, then there are only finitely many possible conjugacy classes of $\phi_1$ in $\Homeo_+(S^1)$.
As $\phi$ is a restriction of $\phi_1$, 
there are only finitely many conjugacy classes in $\Phi$.\ep

Question \ref{que:teichmueller} from the introduction asks to what degree marked spectrum generally controls the equivalence class of a projective surface group action. In light of Theorem \ref{thm:rigid} it seems likely that there should be many inequivalent projective actions with a fixed nonzero marked rotation spectrum.

\begin{que}
How many semi-conjugacy classes of projective free group actions are there on the circle with a given nonzero unmarked rotation spectrum?
\end{que}

A marked rotation spectrum of a free group determines the semi-conjugacy class of an action. 
The discrete case follows from the classification of Fuchsian actions, and the indiscrete case is proved in \cite{WolffPreprint}. 

\subsection{Lie Subgroups of the Circle Homeomorphism Group}\label{ss:lie}
In order to  describe precisely what we mean by ``exotic'' actions of free and surface groups, let us make the following definition.

\bd\label{defn:linear}
Let $L$ be a group.
A representation
\[\rho\co L\to\Homeo_+(S^1)\]
is said to be \emph{linear}\index{representation!linear}
if there exists a finite--dimensional connected Lie group $G\le\Homeo_+(S^1)$ such that $\rho$ factors through $G$.
\ed

If $G$ is a group acting on a space $X$, we define its \emph{support}\index{support (of a homeomorphism)} as
\[
\supp G = \bigcup_{g\in G}\supp g = \{x\in X\co g(x)\ne x\text{ for some }g\in G\}.\] Lie groups in $\Homeo_+(S^1)$ can be characterized as follows.

\begin{thm}[{\cite[Section 4]{Ghys2001}; see also~\cite{Markovic2006,Mann2015IM}}]\label{thm:ghys-lie}
If $G$ is a finite--dimensional connected Lie group in $\Homeo_+(S^1)$,
then one of the following holds:
\be[(i)]
\item
$\supp G = S^1$, and $G$ is conjugate to either $\SO(2)$ or $\PSL_2^{(k)}(\bR)$ for some $1\le k<\infty$;
\item
$\supp G \ne S^1$, and $G$ is isomorphic to a subdirect product of copies of $\PSL_2^\sim(\bR)$, $\bR$ and $\operatorname{Aff}(\bR)$.
\ee
\end{thm}

The following consequence
is 
what we will use.

\begin{lem}\label{lem:lie-circle}
Let $L=\form{s_1,\ldots,s_m}$ be a finitely generated group,
and let $\rho\co L\to\Homeo_+(S^1)$ be a representation. Suppose $\rho$ is semi-conjugate to 
\[
\rho'\co L\to G\le \Homeo_+(S^1)\]
for some finite--dimensional connected Lie group $G$.
If $\Fix\rho(L)=\varnothing$, and if $\Fix\rho(s_i)\ne\varnothing$ for each $i$,
then $G$ is conjugate to $\PSL_2^{(k)}(\bR)$ for some $1\le k<\infty$.
\end{lem}

\bp
Since the existence of a global fixed point is a semi-conjugacy invariant property (Corollary~\ref{cor:fin-orbit}),
we have that $\Fix\rho'(L)=\varnothing$.
In particular, we have the alternative (i) of Theorem~\ref{thm:ghys-lie} for $G$.

Note that $\rot\circ\rho'(s_i)=\rot\circ\rho(s_i)=0$.
If $G$ is conjugate to $\SO(2)$, then we would have that $\rho'(s_i)=1$ and $\supp\rho'(L)=\varnothing$. Hence, $G$ must be conjugate to $\PSL_2^{(k)}(\bR)$ for some $1\le k<\infty$.
\ep

Let us investigate actions of $\PSL_2^{(k)}(\bR)$ in more detail. 
We fix $1\le k<\infty$. Denote by $p_k\co \bR/k\bZ\to\bR/\bZ=S^1$ the natural $k$--fold cover.
There exists a $k$--fold cover
\[q_k\co \Homeo^{(k)}_+(S^1)\to\Homeo_+(S^1)\]
such that the following diagram commutes for all $g\in\Homeo^{(k)}_+(S^1)$:
\[\xymatrix{\bR/k\bZ\ar[r]^g\ar[d]^{p_k} &\bR/k\bZ\ar[d]^{p_k}\\ S^1\ar[r]^{q_k(g)} &S^1}\]
More concretely, we can write
\[ \Homeo^{(k)}_+(S^1)=\{g\in\Homeo_+(\bR/k\bZ)\co g(x+1)=g(x)+1\}.\]
Note that $\PSL_2^{(k)}(\bR)$ is the fiber product given by the following commutative diagram, where the lower row is a central extension:
\begin{equation}\label{eq:homeok}
\xymatrix{
&&\PSL_2^{(k)}(\bR)\ar[r]\ar[d] &\PSL_2(\bR)\ar[rd]\ar[d] \\
1\ar[r]& \bZ/k\bZ \ar[ru]\ar[r]&  \Homeo^{(k)}_+(S^1)\ar[r]^{q_k}&\Homeo_+(S^1)\ar[r]&1}
\end{equation}

The following is immediate from definition.
\begin{lem}\label{lem:rot-cover}
For each $g\in\Homeo_+^{(k)}(S^1)$, we have
\[
\rot\circ q_k(g)=k\rot g\mod \bZ.\]
\end{lem}

Let us describe the limit sets of lifts of projective actions. We shall denote by $q_k^{-1}(L)$ the group defined as the central extension of $L \sse \Homeo_+(S^1)$ by the diagram above.
In particular $q_k^{-1}(L)$ is a group acting on the circle.

\begin{lem}\label{lem:min-cover}
For a finitely generated group $L\le\Homeo_+(S^1)$ and for $1\le k<\infty$, 
we have 
\[
\Lambda\circ q_k^{-1}(L) = p_k^{-1}\circ\Lambda(L).\]
\end{lem}
\bp
If $\Lambda(L)$ is empty, then both sides are empty. So we assume $\Lambda(L)\ne\varnothing$. 
Note that $p_k^{-1}\circ\Lambda(L)$ is nonempty and $q_k^{-1}(L)$--invariant. 
By minimality, we have
\[
\Lambda\circ q_k^{-1}(L) \sse p_k^{-1}\circ\Lambda(L).\]
To show the opposite inclusion, let us pick arbitrary \[x\in p_k^{-1}\circ\Lambda(L),\quad y\in \bR/k\bZ.\]
There exists a sequence $\{g_n\}\sse L$ such that $g_n\circ p_k(y)\to p_k(x)$.
Write $g_n = q_k(h_n)$ for some $h_n\in q_k^{-1}(L)$. 
We have $q_k(h_n)\circ p_k(y) = p_k\circ h_n(y)$. 
Hence possibly after postcomposing $h_n$ with suitable $2\pi\bZ/k$ rotations,
we have that $h_n(y)\to x$. This implies $x$ is in the closure of $q_k^{-1}(L)(y)$, as desired.
\ep

\begin{lem}\label{lem:min-cover2}
Let $L$ be a finitely generated group, and let $1\le k<\infty$.
If there exists a commutative diagram
\[\xymatrix{
&&\Homeo_+^{(k)}(S^1)\ar[d]^{q_k}\\
L\ar[rru]^\rho\ar[rr]^{\rho_0} && \Homeo_+(S^1)
}\]
then we have
\[\Lambda(\rho)=p_k^{-1}\circ\Lambda(\rho_0).\]
\end{lem}

\bp
From the central extension~\eqref{eq:homeok}, we see that
\[
\left[
q_k^{-1}\circ\rho_0(L):\rho(L)
\right]\le k<\infty.\]
By Lemma~\ref{lem:min-normal}, we obtain
\[\Lambda\circ\rho(L)=\Lambda\circ q_k^{-1}\circ\rho_0(L).\]
The desired conclusion follows from Lemma~\ref{lem:min-cover}.
\ep

We will need the following analogue of Lemma~\ref{lem:min-fuchs}. 

\begin{lem}\label{lem:psl2k-fix}
Let $1\le k<\infty$, and let $L$ be a finitely generated group. 
Suppose we have an action $\rho$ of $L$ on $\bR/k\bZ$ such that
\[\rho\co L\to \PSL_2^{(k)}(\bR)\le \Homeo_+(\bR/k\bZ).\]
Then a minimalization of $\rho$ can be chosen as some representation
\[
\bar\rho\co L\to\PSL_2^{(k)}(\bR).\]
\end{lem}
Note that \emph{a priori}, the minimalization $\bar\rho$ of $\rho$ 
merely maps into $\Homeo_+(\bR/k\bZ)$, and not necessarily into $\PSL_2^{(k)}(\bR)$.  

\bp[Proof of Lemma~\ref{lem:psl2k-fix}]
We may assume $\rho$ is faithful, by considering $L/\ker\rho$ instead of $L$ if necessary.
If $\rho$ is minimal, then we have nothing to prove. 
If $\rho$ has a finite orbit then $\bar\rho$ factors through $\SO(2)$, which can be also regarded as a subgroup of $\PSL_2^{(k)}(\bR)$.
So we will assume that $\Lambda(\rho)$ is a Cantor set.

We put $\rho_0=q_k\circ\rho$. By Lemma~\ref{lem:min-cover2}, we have
\[
\Lambda(\rho)=p_k^{-1}\circ\Lambda(\rho_0).\]
Note that $\rho_0(L)$ is discrete; otherwise, $\rho_0(L)$ would contain an irrational rotation by Lemma~\ref{lem:dense} and $\rho$ would be minimal.
The same reasoning shows that $\rho_0(L)$ is non-elementary. We see that $\Lambda(\rho_0)$ is a Cantor set.

By Lemma~\ref{lem:min-fuchs}, the minimalization $\bar\rho_0$ of $\rho_0$ can be chosen as a finite-area Fuchsian representation of 
\[
L/\ker\rho_0\cong \rho_0(L).\]
There is a Cantor function $h_0\co S^1\to S^1$
for $\Lambda\circ \rho_0(L)$ 
such that
\[
\rho_0\succcurlyeq_{h_0}\bar\rho_0.\]

We claim that there exists a Cantor function $h$ on $\bR/k\bZ$
and a representation $\bar\rho\co L\to\Homeo_+(\bR/k\bZ)$ such that 
the cube diagram in Figure~\ref{fig:psl2k-fix} holds for all $g\in L$,
where all the vertical maps denote $p_k$.
From the standard covering theory for the map $p_k$,
we can define $h$ so that the front, left, right and bottom faces all commute. 
Then $h$ is the Cantor function corresponding to $\Lambda(\rho)$. 
If $h(x)=h(y)$ for some $x,y\in\bR/k\bZ$, then $h\circ\rho(g)(x)=h\circ\rho(g)(y)$; this follows from that 
$\Lambda(\rho)$ is $\rho(L)$--invariant.
So there exists a well-defined map $\bar\rho$ such that the upper face of the cube commutes. The rear face commutes by diagram chasing.

We have $\bar\rho\succcurlyeq_{p_k}\bar\rho_0$ from the cube. 
In particular, we see
\[\bar\rho\in\Hom(L,\PSL_2^{(k)}(\bR)).\]
We conclude that $\bar\rho$ is minimal
from Lemma~\ref{lem:min-cover2} and from that
\[
\Lambda(\bar\rho)=p_k^{-1}\circ\Lambda(\bar\rho_0)=\bR/k\bZ.\qedhere\]
\ep

\begin{figure}[h!]
\[
 \xymatrix{
& \bR/k\bZ \ar[rr]^>>>>>>>>{\bar\rho(g)} \ar'[d][dd] && \bR/k\bZ \ar[dd] \\
 \bR/k\bZ \ar[ru]^h \ar[rr]^>>>>>>>>{\rho(g)} \ar[dd] &&  \bR/k\bZ \ar[ru]^h \ar[dd] & \\
& S^1  \ar'[r][rr]^>>>>>>>>{\bar\rho_0(g)} &&  S^1  \\ 
S^1   \ar[ru]^{h_0} \ar[rr]^>>>>>>>>{\rho_0(g)} &&  S^1  \ar[ru]^{h_0} &
}
\]
\caption{Proof of Lemma~\ref{lem:psl2k-fix}. All the vertical maps are $p_k$.}
\label{fig:psl2k-fix}
\end{figure}
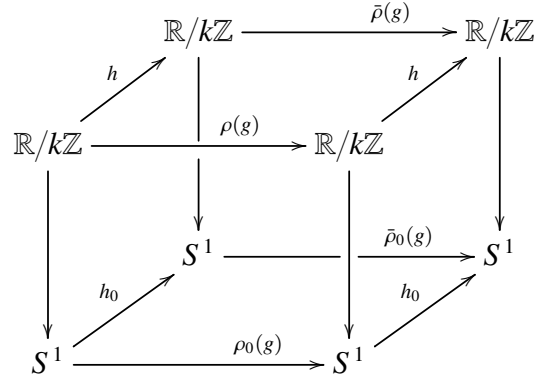

\begin{rem}\label{rem:psl2k-fix}
Let $g\in\Homeo_+(S^1)$, and let $x\in S^1$ be an isolated fixed point of $g$. 
We say $x$ is a \emph{hyperbolic fixed point}\index{fixed point!hyperbolic} if $x$ is an attracting or repelling fixed point; otherwise, we say $x$ is a \emph{parabolic fixed point}\index{fixed point!parabolic}; see for example \cite{BMNR2013MZ}. Note also that a non-elliptic element $h\in\PSL_2^{(k)}(\bR)$ may or may not have a fixed point; however, $h^k$ will have $k$  fixed points which are parabolic, or $2k$ fixed points which are hyperbolic. So in the above lemma,
for each $g\in L\setminus\ker\bar\rho$ the fixed point set of $\bar\rho(g)$ on $\bR/k\bZ$ are all hyperbolic or all parabolic.
\end{rem}

Let us now note easy constructions of circle actions by free and surface groups which are not projective even up to semi-conjugacy.
Hence it makes more sense to quest a free or surface group representation that is not conjugate or semi-conjugate to any linear action in the sense of Definition~\ref{defn:linear}.
We begin with a general observation:

\begin{exmp}
Fix a positive divisor $k>1$ of $2g-2$.
Each Fuchsian representation 
\[\rho\co \pi_1(S_g)\to\PSL_2(\bR)\le\Homeo_+(S^1)\]
has the Euler number $\mathrm{e}(\rho)=\pm(2g-2)$,
which is the (well-defined!) translation number of the lift of the relator in $\pi_1(S_g)$ to $\bR$.
As $k$ divides $2g-2$, this $\bR$--action projects to some faithful $\bR/k\bZ$--action 
\[\rho'\co \pi_1(S_g)\to\PSL_2^{(k)}(\bR)\le\Homeo_+(\bR/k\bZ).\]
By Lemma~\ref{lem:rot-cover}, the rotation spectrum is $\rho'$ is contained in $\bQ/\bZ$.
Also, the action $\rho'$ is minimal by Lemma~\ref{lem:min-cover2}.
Assume that $\rho'$ is semi-conjugate to some projective action. 
Then Lemmas~\ref{lem:dense} and~\ref{lem:min-fuchs} would imply that 
$\rho'$ is conjugate to some discrete projective action.
This is a contradiction, since the $k$--th power of every nontrivial element fixes exactly $2k$ points under $\rho'$.
We conclude that $\rho'$ is not semi-conjugate to any projective action.
\end{exmp}

\begin{exmp}
We can lift a finite--area Fuchsian representation $\rho$ of $F_n$ to:
\[\rho\co F_n\to \PSL_2^{(k)}(\bR)\le \Homeo_+(\bR/k\bZ).\]
The same argument as above shows that $\rho$ is not semi-conjugate to any projective action for $k>1$.
\end{exmp}


\subsection{Free and surface subgroups of $\pi_1(M)$}
\label{ss:3mfd}

Throughout this subsection, we let $M$ be a closed fibered hyperbolic $3$--manifold 
\[ M = S_r\tilde\times_t S^1\]
for some genus $r\ge2$ and 
some pseudo-Anosov mapping class $t$ on $S_r$;
see~\eqref{eq:univ}.
We have the universal circle action and the Kleinian action \cite{thurstonnotes}:
\begin{align*}
\rho_{\mathrm{u}}&\co \pi_1(M)\to\Homeo_+(S^1),\\
\rho_{\mathrm{K}}&\co\pi_1(M)\to\PSL_2(\bC)\le\Homeo_+(S^2).
\end{align*}
We will make two additional assumptions:
\be[(i)]
\item $\rho_{\mathrm{u}}(t)$ fixes at least four points;
\item $\rho_{\mathrm{u}}(a)$ fixes at least one point
for each $a\in \pi_1(M)$.
\ee
The assumption (i) simply means that a stable foliation of $t$ has no single-prong singularity~\cite{CB1988}.
The assumption (ii) can always be guaranteed after taking a finite--sheeted cover of $M$; see Lemma~\ref{fp3}.

We will study the restriction of $\rho_{\mathrm{u}}$ on 
 a surface subgroup $G\le\pi_1(M)$.
In the case when $G$ is quasi-Fuchsian (i.e.\ a quasiconformal deformation of a Fuchsian group in $\PSL_2(\bC)$ \cite{thurstonnotes}) and transverse to the flow, we will show that $\rho_{\mathrm{u}}\restriction_G$ is semi-conjugate to a Fuchsian actions. When $G$ is geometrically infinite, we sketch K. Mann's argument~\cite{MannThesis} that $\rho_{\mathrm{u}}\restriction_G$ is ``often'' not projective even up to semi-conjugacy.

For existential results, we find a quasi-Fuchsian surface group $G$ in $\pi_1(M)$ such that $\rho_{\mathrm{u}}\restriction_G$ is not \emph{conjugate} to any linear actions.
We will also find a finitely generated free subgroup $F\le\pi_1(M)$ such that $\rho_{\mathrm{u}}\restriction_F$ is non-linear even up to semi-conjugacy.
All of the above actions will have rotation spectra identically zero (by the second assumption), 
and hence will share some of the features of discrete projective surface group actions.

\subsubsection{Quasi-Fuchsian Surface Subgroups}
We first justify the assumption (ii):
\begin{lem}\label{fp3} 
Let $N$ be a closed fibered hyperbolic 3--manifold,
and let $\rho$ denote the universal circle action of $N$.
Then there exists a finite sheeted cover $N'$ of $N$ such that
$\rot\circ\rho(\pi_1(N'))=\{0\}$.
\end{lem}

\begin{proof}
We will denote $\rho(g)(x)=g.x$ for $g\in\pi_1(N)$ and $x\in S^1$.
Let $t$ denote the stable letter of $\pi_1(N)$ from the given fibering.
Then for some $L>0$,
the pseudo-Anosov element $t^L$ has an even number of fixed points, in such a way that attracting and repelling fixed points alternate~\cite{CB1988}.
Define a finite sheeted cover $N'\to N$ by
\[\pi_1(N')=\pi_1(S_r)\rtimes\form{t^L}.\]

We claim every element of $\pi_1(N')$ fixes at least one point.
Consider an arbitrary $ht^{Lk}\in\pi_1(N')$. For brevity we assume $k>0$, as the case  $k<0$ is similar. We may only consider the case that $\Fix h$ and $\Fix t^L$ are disjoint. 
The hyperbolic element $ h\in\PSL_2(\bR)$ has one attracting and one repelling fixed points on $S^1$. 
Pick an interval $J=(p,q)$ on $S^1$ such that $p$ and $q$ are attracting fixed points of $ h $ and of $ t^L$, respectively.
We choose $J$ to be innermost, in the sense that $J$ contain no fixed points of $ h$ or $ t^L$.
Up to reversing, we have a circular ordering
\[p <  ht^{Lk}.p <  ht^{Lk}.q < q.\]
This implies $ ht^{Lk}$ fixes some point in $J$, as desired. 
\end{proof}

The first main point of this subsection is the following.

\begin{thm}\label{ct=min}
	Let $M$ be a closed hyperbolic $3$--manifold which is a surface bundle with fiber $S$,
	and let $G$ be a quasi-Fuchsian surface subgroup of $\pi_1(M)$ corresponding to a quasi-Fuchsian surface $W$ which is transverse to the flow induced by the fibration with fiber $S$. 
	The group $\pi_1(M)$ admits the universal  circle action on $S^1 = \partial \pi_1(S)$.
	Then the restriction of the universal  circle action of $\pi_1(M)$ to $G$ is semi-conjugate to a Fuchsian action of $G$.
\end{thm}
\bp We first fix some auxiliary notation. By hypothesis, $G$ is isomorphic to the fundamental group $\pi_1(W)$ of a closed surface $W$. Also, $G$ is quasi-Fuchsian.
Set $H=G\cap\pi_1(S)$, the kernel of the map $G \to  \bZ$ induced by the fibration $M \to S^1$ when restricted to $G$. Note that $H$ is normal in $G$. 
Limit sets of $H$ or $G$ in $S^1$ and $S^2$ will be denoted by superscripts 1, 2 respectively.  Since limit sets are unique minimal closed non-empty invariant sets, we have,
\begin{itemize}
	\item $(H.x)' = \Lambda_H^1$ for any $x \in S^1$
	
	\item $(H.x)' = \Lambda_H^2$ for any $x \in S^2$
\end{itemize}

The Cannon-Thurston map $CT$ maps $\Lambda_H^1$ continuously and $H-$equivariantly onto $\Lambda_H^2$~\cite{CT}.
In particular,
$\Lambda_H^1 \subset  CT^{-1} (\Lambda_H^2)$. There are now two natural circle actions of $H$; further  the $H-$actions
 give us two natural circle actions of $G$, since $H$ is normal in $G$ with quotient $\mathbb Z$. These actions are given by:

\begin{enumerate}
	\item The minimalization of the action of $H$ on $S^1$ obtained by collapsing the complementary arcs of  $\Lambda_H^1$. Thus, the minimalization is a standard semiconjugacy and one gets an action of $G$ on $S^1$ as a result;
	\item The action of $H$ on $\Lambda_H^2 = \Lambda_G^2$, which is homeomorphic to $S^1$. Since $\Lambda_G^2$ is homeomorphic to a circle and the $G-$action on  $\Lambda_H^2$ is conjugate to the usual Fuchsian action of $G$ on $S^1$, we obtain another action of $G$ on $S^1$. 
\end{enumerate}

We shall prove that these two actions of $G$ on $S^1$ are topologically conjugate, i.e.\ the 
Cannon-Thurston map
$CT$  is (topologically) identical to the minimalization. This suffices to prove the theorem. Pairs of points on the boundary of $S^1 \setminus\Lambda_H^1$ will be referred to as {\it adjacent boundary points}.
To prove that $CT$  is (topologically) identical to the minimalization map it suffices to prove that $CT$ preserves circular order, or equivalently the following two conditions:

\begin{enumerate}
	\item
	Adjacent boundary points of $\Lambda_H^1$  are identified by $CT$;
	\item
	No other pair of points of  $\Lambda_H^1$  are identified by $CT$.
\end{enumerate}

We first show that Condition 2 follows from Condition 1. If $CT(p) = CT(q)$ for all adjacent boundary points of $\Lambda_H^1$, then $CT$ induces a map
\[\overline{CT}: \Lambda_H^1/\sim \to \Lambda_H^2,\] where $\sim$ denotes the equivalence relation given by $p\sim q$ if $p, q$ are adjacent boundary points of $\Lambda_H^1$. But $\Lambda_H^1/\sim$ is homeomorphic to the circle $S^1$ (and the quotient map $\Lambda_H^1 \to \Lambda_H^1/\sim$ is the standard "devil's staircase" map sending the Cantor set onto the circle by collapsing complementary intervals). Hence $\overline{CT}$ is a quotient map from the circle to the circle.
Further $CT$ is a finite-to-one map that is injective on the dense subset of $\Lambda_H^1$ given by the set of attracting fixed points of elements of $H$. Hence $\overline{CT}$ is a finite-to-one quotient map from the circle to the circle that is injective on a dense subset. Such a quotient map is necessarily 
a homeomorphism; in particular it is injective. Condition 2 follows.

It remains to show Condition 1.  In
\cite[Proposition 3.9]{CooperLongReid}, it is shown that if $G$ corresponds to a quasi-Fuchsian surface {\it transverse to the flow}, then
the convex hull of $\Lambda_H^1$ in $\tilde{S} (= {\mathbb H}^2)$ is an infinite sided polygon whose sides are leaves of the stable and unstable laminations. Such a polygon is referred to as a leaf-polygon in \cite{CooperLongReid}.  Recall \cite{CT} that the Cannon-Thurston map identifies precisely end-points of stable and unstable laminations (and ideal points of complementary polygons). Hence if $p, q$ are adjacent boundary points of $\Lambda_H^1$,
then the bi-infinite geodesic joining $p, q$ is a leaf of either the stable or the unstable lamination and hence $CT(p) = CT(q)$. This proves Condition 1.
\ep

We address the second main point of Section~\ref{ss:3mfd}.

\begin{thm}\label{thm:qf-nonconj}
There exists a quasi-Fuchsian surface subgroup $G\le\pi_1(M)$ such that $\rho_{\mathrm{u}}\restriction_G$ is not conjugate to any (possibly non-injective) linear action.\end{thm}

We will employ the following consequence of \cite{KM2012}.
We are grateful to J. Kahn for explaining the relevant part to us in the argument below.

\begin{lem}[See ~\cite{KM2012}]\label{lem:kahn markovic} 
	For each element $g\in\pi_1(M)$, there exists an $N\gg 0$ and a quasi-Fuchsian surface subgroup $G\le \pi_1(M)$ such that $g^N\in G$.
\end{lem}

\begin{proof} 
	The construction of Kahn and Markovic in \cite{KM2012} proceeds by gluing pairs of pants together to construct closed surface subgroups $G\le \pi_1(M)$.
	This is arranged in such a way that 
	for a given $\epsilon>0$ and for $R>0$ sufficiently large,
	if the translation length $\ell(h)$ of $h\in\pi_1(M)$ satisfies 
	$|\ell(h)-R|\le\epsilon/2$, then a surface group $G\le\pi_1(M)$ can be required to contain $h$; see also \cite[Section 4.1]{KM2012GT}.
	In particular, we can require that $g^N\in G$ for some $N\gg0$.
We remark that geometrically, the group $G$ is forced to be quasi--Fuchsian (see also \cite{CooperLongReid}).
\end{proof}

\bp[Proof of Theorem~\ref{thm:qf-nonconj}]
By Lemma~\ref{lem:kahn markovic}, there exists a quasi--Fuchsian surface subgroup $G\le\pi_1(M)$ that contains $t^N$ for some $N\gg0$.
\emph{Assume for contradiction} that $\rho:=\rho_{\mathrm{u}}\restriction_G$ is conjugate to a linear action.
Write $H=G\cap\pi_1(S_r)\unlhd G$. Since
\[G/H\le \pi_1(M)/\pi_1(S_r)\cong\bZ,\]
and since $\rho_{\mathrm{u}}$ restricts to a Fuchsian action on $\pi_1(S_r)$,
the restriction $\rho\restriction_H$ is a non-elementary Fuchsian action. In particular, $\rho$ has no global fixed point.
By Lemma~\ref{lem:lie-circle}, for some integer $k\ge1$ we may write
\[\rho_{\mathrm{u}}\co G\to \PSL_2^{(k)}(\bR).\]
We have $k>1$ since $t^N\in G$.
On the other hand, each $\rho(u)\in \rho(H)\setminus\{1\}$
has exactly two hyperbolic points on $S^1$.
This contradicts Remark~\ref{rem:psl2k-fix}.
\ep

\subsubsection{Geometrically infinite surface groups}
With regard to Theorems~\ref{ct=min} and~\ref{thm:qf-nonconj}, the following question suggests itself:
\begin{que}\label{que:kahn markovic}
For which surface subgroups $G\le\pi_1(M)$,
is the restriction $\rho_{\mathrm{u}}\restriction_G$ semi-conjugate to a linear action?
\end{que}

By Theorem \ref{ct=min}, such a surface group $G$ is either a quasi-Fuchsian surface not transverse to the flow or else
necessarily geometrically infinite.
The difficulty of repeating the argument of Theorem~\ref{thm:qf-nonconj} in addressing Question~\ref{que:kahn markovic} has to do with the location of the limit set of a certain subgroup of the fiber group. Indeed, let $G\le\pi_1(M)$ be a surface subgroup containing $t^N$ for some $N\gg0$, as in the proof of Theorem~\ref{thm:qf-nonconj}. Put $H=G\cap\pi_1(S_r)$ 
and write 
\[
\Lambda:=\Lambda\circ\rho_{\mathrm{u}}(G)=\Lambda\circ\rho_{\mathrm{u}}(H).\]
The set $\Lambda$ in general will be smaller than the limit set of $\pi_1(S_r)$, since $H$ is not normal in $\pi_1(S_r)$. It is possible that the entirety of $\Lambda$ lies in between two consecutive attracting points of $t$. See Figure~\ref{f:pA}.

\begin{figure}[h]
\includegraphics[scale=0.45]{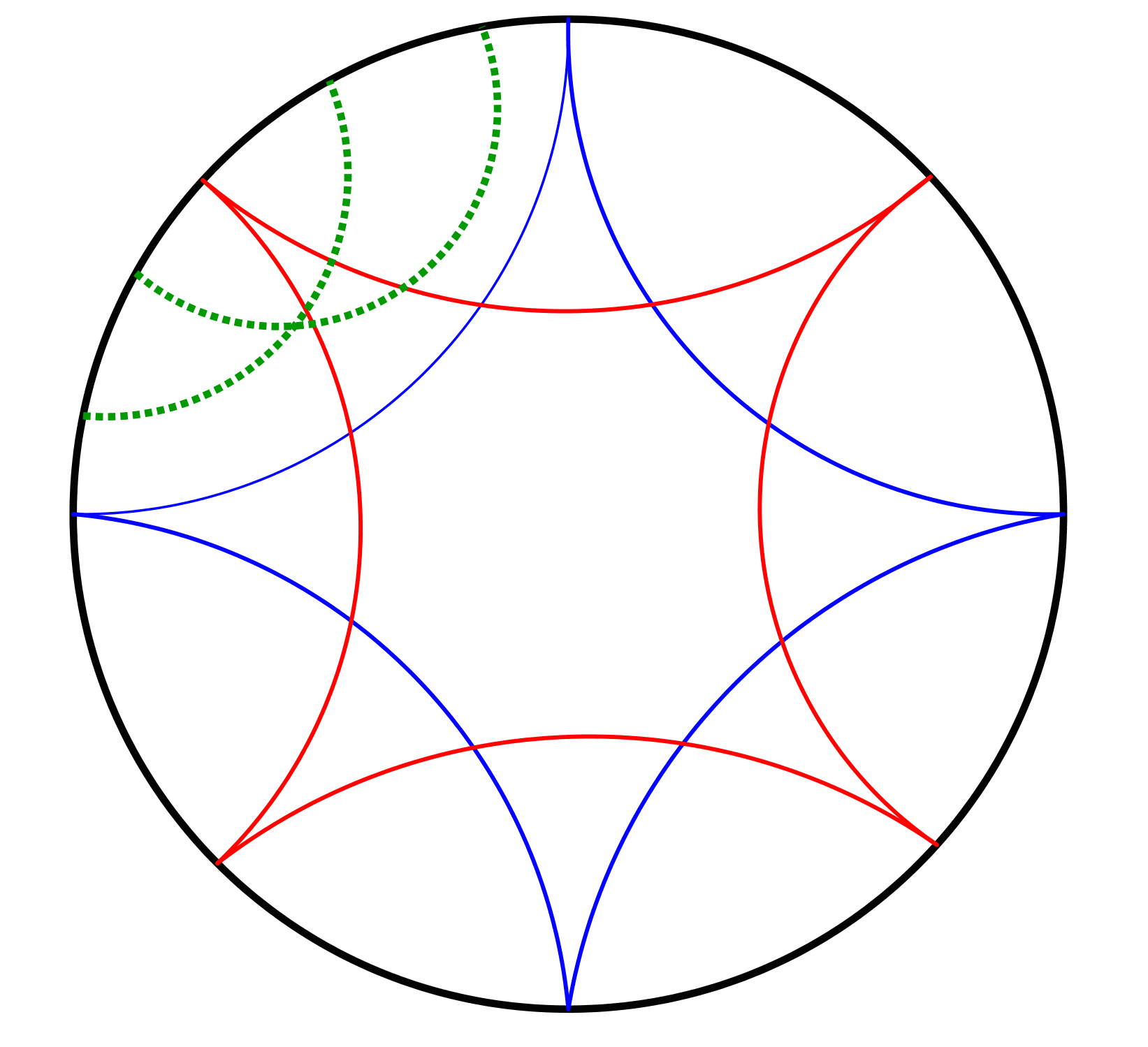}
\caption{The stable letter stabilizes the blue and the red geodesics, corresponding to attracting and repelling fixed points of the stable letter, respectively. The green dotted geodesics connect points in the limit set $\Lambda$ of $H$.}
\label{f:pA}
\end{figure}

If $\Lambda$ is entirely contained between two attracting fixed points of the stable letter, then the minimalization of the action of the action of $G$ may collapse the fixed points of $t$ into two or one fixed point, which is dynamically indistinguishable from a hyperbolic or parabolic fixed point. Note that in this situation, $\Lambda$ contains at most three fixed points of the stable letter $t$. However, we have the following proposition:

\begin{prop}\label{prop:closed nonlinear semi}
Let $G\le\pi_1(M)$ be a surface subgroup containing $t^N$ for some $N>0$. Suppose that the limit set $\Lambda$ of $H=G\cap\pi_1(S_r)$ contains more than four fixed points of $\rho_0(t)$.
Then the action of $G$ is not semi-conjugate to any (possibly non--injective) linear action.
\end{prop}


\begin{proof}
Note that the limit set $\Lambda$ is perfect, and the minimalization map $\Lambda\to S^1$ is at most two--to--one. It follows that if $\Lambda$ contains more than four fixed points of the stable letter $t$ then the action of $t$ in the minimalization will have at least three fixed points, whereas each element of $H$ will have either one or two fixed points by Lemma~\ref{lem:min-fuchs}. By the same argument as Theorem~\ref{thm:qf-nonconj}, we see that such a minimal action cannot be conjugate (and hence cannot be semi-conjugate) to any linear action.
\end{proof}

Let us now assume $G\le\pi_1(M)$ is a geometrically infinite surface subgroup.
From a result by Bonahon and Thurston \cite{Bonahon1986AM}, we have a finite cover $M'\to M$ such that $G$ is the associated fiber subgroup of some fibering of $M'$. 
We say \emph{$G$ is in the original fibered face},
if $G$ and the restriction of $\pi_1(S_r)$ belongs to the same fibered face of the Thurston norm sphere for $M'$, possibly up to multiplication by $-1$.

We can now state the third main point of Section~\ref{ss:3mfd}.
\begin{thm}\label{thm:geominf}
Let $G$ be a fiber subgroup of $\pi_1(M)$ associated to some fibering over $S^1$.
If $G$ is in the original fibered face, then $\rho_{\mathrm{u}}(G)$ is conjugate to a Fuchsian action.
\end{thm}

\bp
Theorem 14.11 of 
\cite{FLP1991} shows that a closed embedded geometrically infinite fiber surface in $M$ can be isotoped to be transverse to the flow if and only if it lies in the fibered face of the Thurston unit norm sphere. Let $F$ be any such surface with fundamental group $G (\subset \pi_1(M))$. Then $G$ acts on the leaf space of the flow (since $F$ is transverse to the flow). Since both $\yt{S_r}$ and $\tilde{F}$ are sections of the flow ($S_r$ being the fiber) on $\tilde{M}$,
the $G$ actions on  $\yt{S_r}$ and $\tilde{F}$ are conjugate. Hence the universal circle action of $G$  is conjugate to the action of $G$ on $\partial \tilde{F}$, which is the standard action.
\ep

As we remarked in the introduction, K. Mann has a construction of surface group actions on $S^1$ that is not semi-conjugate to any faithful projective action, which appears in her thesis~\cite{MannThesis}. In fact, she showed that if a fiber subgroup $G\le\pi_1(M)$ is not in the original fibered face, then $\rho_{\mathrm{u}}\restriction_G$ is not semi-conjugate to any faithful projective action.
Let us sketch her argument.

Note that the Thurston norm is the dual to the Euler class of the surface group fiber; see \cite{Thurston1986MAMS}.
Since $G\unlhd \pi_1(M)$, we have
\[\Lambda\circ\rho_u(G)=\Lambda\circ\rho_u(M)=S^1.\]
In other words, $\rho_u\restriction_G$ is minimal. 
If $\rho_u\restriction_G$ is semi-conjugate to a projective action,
then one can find 
\[\rho\co G\to\PSL_2(\bR)\]
which is conjugate to $\rho_u\restriction_G$.
As $\rot\circ\rho_u(G)=\{0\}$, we see $\rho$ is actually Fuchsian (Lemma~\ref{lem:dense}). 
In particular, $\rho^*\eu$ is a maximal Euler class. 
But this contradicts the fact that the surface subgroup came from a different fibered face of the Thurston norm ball.

\subsubsection{Free subgroups of $\pi_1(M)$}
Let us now show that the universal circle action restricts to some non-linear free group action, in the spirit of Theorem~\ref{thm:qf-nonconj}. 
We recall the following straightforward fact, which follows from an easy ping--pong argument; see for instance~\cite[p. 467]{BH1999}.

\begin{lem}\label{lem:pingpong}
Let $G$ be a torsion-free word-hyperbolic group,
and let $\{g_1,\ldots,g_n\}\sse G$ be elements such that no two of them generate a cyclic group. Then for all sufficiently large $N>0$, the group $\langle g_1^N,\ldots,g_n^N\rangle$ is free of rank $n$.
\end{lem}

The last main point of Section~\ref{ss:3mfd} is the following.
\begin{thm}\label{thm:free nonlinear}
There exists a finitely generated free subgroup $F$ of $\pi_1(M)$ such that $\rho_{\mathrm{u}}\restriction_F$ is not semi--conjugate into any  linear action of $F$.
\end{thm}

\begin{proof}
Since $\pi_1(S_r)$ acts minimally, we can choose a minimal collection of elements $X\sse\pi_1(S_r)$
such that each component of \[S^1\setminus\Fix\rho_{\mathrm{u}}(t)\] 
contains a fixed point of some element in $\rho_{\mathrm{u}}(X)$.
We set $F$ to be the free group generated by sufficiently high powers of $t$ and of the elements of the finite set $X$, as guaranteed by Lemma~\ref{lem:pingpong}. 
Say $t^N\in F$ for some $N>0$, and put $\rho=\rho_{\mathrm{u}}\restriction_F$.

Write $H=F\cap\pi_1(S_r)\unlhd F$. Since
\[F/H\le \pi_1(M)/\pi_1(S_r)\cong\bZ,\]
the restriction $\rho\restriction_H$ is a non-elementary Fuchsian action. 
The limit set \[\Lambda:=\Lambda\circ\rho(F)=\Lambda\circ\rho(H)\]  contains $\Fix\rho(t)$ and $\Fix\rho(a)$ for all $a\in H$.

\emph{Assume for contradiction} that $\rho$ is semi-conjugate to a linear action $\bar\rho$. 
By Lemma~\ref{lem:lie-circle}, there exists some integer $k\ge1$ such that
\[
\bar\rho\co F\to \PSL_2^{(k)}(\bR).\]
By Lemma~\ref{lem:psl2k-fix}, we may assume $\bar\rho$ is a minimalization of $\rho$. In particular, there exists a surjective monotone degree one map $h$ on $S^1$ such that 
\[\rho\succcurlyeq_h\bar\rho.\]

We claim now that no two fixed points of $t$ are identified under $h$. Indeed, $h$ is either two--to--one or one--to--one at every point in the limit set $\Lambda$.
Precisely, $h$ is either a homeomorphism or given by the ``devil's staircase map" or ``Nielsen's Cannon--Thurston map" (see~\cite{Floyd}). Since any interval bounded by two fixed points of $\rho(t)$ contains a point in $\Lambda$ within its interior, we have that $h$ cannot identify distinct fixed points of $t$, thus establishing the claim.
It follows that $\bar\rho(t)$ fixes at least four points
and so, $k\ge2$.

Since $\Lambda=\Lambda\circ\rho(H)$, we see that $\bar\rho\restriction_H$ is a minimalization of $\rho\restriction_H$.
Let $a\in H$.
By Remark~\ref{rem:psl2k-fix}, the set $\Fix\bar\rho(a)$ consists of one parabolic or two hyperbolic fixed points. 
This contradicts that $\bar\rho(F)\le\PSL_2^{(k)}(\bR)$.
\end{proof}


\subsection{Nonlinear Smooth Actions of Free Groups}\label{ss:smooth-free}
So far, we have described nonlinear actions of free groups regarded as subgroups of fibered hyperbolic $3$--manifolds, which in turn are regarded as subgroups of mapping class groups.
These nonlinear actions are not, however, guaranteed to be smooth, as Nielsen's actions of mapping class groups are non-smooth.
Actually, no finite index subgroups of mapping class groups admit $C^2$--smooth actions on the circle~\cite{BKK16}. 

We will now produce nonlinear actions of free groups that are $C^\infty$--smooth.

\begin{thm}\label{thm:liegp}
For each $n\ge2$, there exists a faithful action
\[\rho_n\co F_n\to \Diff^\infty_+(S^1)\]
which is not semi-conjugate to any linear action of $F_n$ on $S^1$.
\end{thm}

Recall the definition of hyperbolic and parabolic fixed points from Remark~\ref{rem:psl2k-fix}.

\begin{lem}\label{lem:hyp-fixed}
Let $g,\bar g\in \Homeo_+(S^1)$,
and let $h\co S^1\to S^1$ be a surjective monotone degree one map
such that $h\circ g = \bar g\circ h$.
If $h$ is injective on $\Fix g$,
then we have:
\be
\item
$\Fix \bar g = h(\Fix g)$;
\item
If $x\in \Fix g$ is a hyperbolic fixed point, then so is $h(x)$;
\item
If $x\in \Fix g$ is a parabolic fixed point, then so is $h(x)$.
\ee
\end{lem}

\bp
For part (1),
the inclusion $h(\Fix g)\sse\Fix\bar g$ is obvious. 
To show the opposite inclusion, 
assume
for some $x\in S^1\setminus \Fix g$
that
\[ h\circ g(x)=\bar g\circ h(x) = h(x).\]
Then $J=h^{-1}(h(x))$ is a non-degenerate closed interval containing $x$ and $g(x)$. Since $\form{g}$ acts on $J$,
we see that $\partial J\sse \Fix g$ and that $h(\partial J)=\{h(x)\}$.
This contradicts the assumption that $h\restriction_{\Fix g}$ is injective. So, part (1) is proved.

The above proof shows that
for each $y\in\Fix \bar g$ the set $h^{-1}(y)$ is a singleton.
For each isolated fixed point $x\in\Fix g$,
the point $h(x)$ is an isolated point in $\Fix \bar g$.
If $V$ is an open neighborhood of $h(x)$ such that 
\[\Fix\bar g\cap V = \{h(x)\},\]
then the only fixed point of $g$ in $U=h^{-1}(V)$ is $x$.
By looking at the dynamics of the group $\form{g}$ on $U$, we obtain  parts (2) and (3).
\ep

\bp[Proof of Theorem~\ref{thm:liegp}]
Let us first consider the case $n=2$. We put 
\[A = \{0,1/3,2/3\}\sse S^1.\]
Pick $a\in\Diff^\infty_+(S^1)$ such that 
$\Fix a = A$
and such that 
\[
a(t)
\in\begin{cases}
(t,1/3),&\text{ if }t\in(0,1/3),\\
(1/3,t),&\text{ if }t\in(1/3,2/3),\\
(2/3,t),&\text{ if }t\in(2/3,1).
\end{cases}\]
The fixed points $0$ and $1/3$ are hyperbolic
and $2/3$ is parabolic; see Figure~\ref{fig:liegp}.

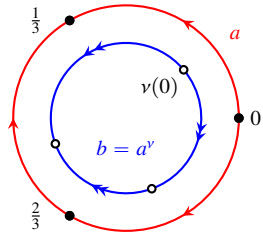
\begin{figure}[h!]
  \tikzstyle {bv}=[black,draw,shape=circle,fill=black,inner sep=1pt]
  \tikzstyle {wv}=[black,draw,shape=circle,fill=white,inner sep=1pt]  
  \tikzstyle {a}=[red,postaction=decorate,decoration={%
    markings,%
    mark=at position .5 with {\arrow[red]{stealth};}}]
  \tikzstyle {b}=[blue,postaction=decorate,decoration={%
    markings,%
    mark=at position .47 with {\arrow[blue]{stealth};},%
    mark=at position .53 with {\arrow[blue]{stealth};}}]
\begin{tikzpicture}[scale=.5,>=stealth',thick]
\foreach \a in {1,...,3}{
\draw (\a*360/3: 3) node [bv] {};
}

\draw [a] (0:3) node [black,right] {\tiny $0$} arc (0:120:3);
\draw [a] (240:3) arc (240:120:3) node [black,left=.2] {\tiny $\frac13$};
\draw [a] (360:3) arc (360:240:3) node [black,left=.2] {\tiny $\frac23$};
\draw [b] (40:2) arc (40:200:2);
\draw [b] (290:2) arc (290:200:2);
\draw [b] (400:2) arc (400:290:2);
\draw [black] (.9,.8) node {\tiny $\nu(0)$};
\draw (0,-.8) node [blue] {\tiny $b=a^\nu$};
\draw (2.9,1.8) node [above,red] {\tiny $a$};
\foreach \a in {1,...,3}{
\draw (\a*360/3: 3) node [bv] {};
}
\foreach \b in {40,200,290}{
\draw (\b: 2) node [wv] {};
}
\end{tikzpicture}%
\caption{Proof of Theorem~\ref{thm:liegp}.}
\label{fig:liegp}
\end{figure}


By Theorem~\ref{thm:circle baire1}, there is $\nu\in\Diff^\infty_+(S^1)$ such that $\form{a,a^\nu}\cong F_2$ and moreover,
\begin{equation}\label{eqn:nu}
t<\nu(t)<t+1/3\text{ for all }t\in A.\end{equation}
This is because \eqref{eqn:nu} is an open condition for $\nu$.
Let us denote a generating set of $F_2$ as $\{a_0,b_0\}$,
and define $\rho\co F_2\to\Diff^\infty_+(S^1)$ by
\[\rho(a_0)=a,\quad\rho(b_0) = a^\nu.\]

By construction, the finite set
\[C:=A\cup\nu(A)\] is contained in the closure of every orbit of $\rho$;
this follows from that the supporting intervals of $a$ and $a^\nu$ form a ``chain'' of intervals, as defined in~\cite{KKL2016}.
We see that $\Lambda(\rho)$ is infinite and contains $C$.

We claim that $\rho$ is a desired representation of the theorem for $n=2$. Assume for contradiction that $\rho$ is semi-conjugate to a linear action. By Lemma~\ref{lem:lie-circle}, we can find a representation
\[
\bar\rho\co F_2\to \PSL_2^{(k)}(\bR)\]
for some $k$ such that $\rho$ is semi-conjugate to $\bar\rho$.
By applying Lemma~\ref{lem:psl2k-fix}, we may further assume that $\bar\rho$ is minimal.
In particular, we have that $\bar\rho$ is a minimalization of $\rho$.

Let $h\co S^1\to S^1$ be a surjective monotone degree one map so that \[\rho\succcurlyeq_h\bar\rho,\] as in Definition~\ref{defn:minimalization}.
The map $h$ is either two-to-one or one-to-one at each point of $C$.
Since $A$ and $\nu(A)$ are alternating on $S^1$ and since $h$ is monotone degree--one, we see that $h\restriction_A$ is injective.
By Lemma~\ref{lem:hyp-fixed}, the set $h(A)$ will contain both hyperbolic and parabolic fixed points of $\bar\rho(a)$.
As we are assuming $\bar\rho(F_2)\le\PSL_2^{(k)}(\bR)$, we have a contradiction.


For the case $n>2$, we simply consider the embedding as a finite index subgroup
\[ F_n\to F_2\]
and define $\rho_n$ as the restriction of $\rho$ on $F_n$. 
Then for some $N>0$, we have
\[ \form{a^N, (a^{\nu})^N}\le \rho(F_n).\]
Since all the dynamical properties of $\form{a,a^\nu}$ that we used are shared by $\form{a^N,(a^\nu)^N}$ as well, we see that 
$\rho_n$ is not linear even up to semi-conjugacy.
\ep

\appendix
\section{Equivalent Notions of Semi-Conjugacy }\label{sec:append}
The aim of this appendix is to illustrate the equivalence of the different notions of semi-conjugacy in Theorem~\ref{thm:equiv}. We will mainly give an account of the fact that monotone equivalence and minimalization equivalence  both coincide with the other notions of semi-conjugacy.
We will also guide the reader through the literature for a proof that the rest of the notions of semi-conjugacy are  equivalent to each other. 

The facts in this section are mostly based on Ghys' original ideas in \cite{Ghys1987,Ghys2001}. 
Readers are also referred to very recent surveys by Mann~\cite{Mann-hb} and by Bucher--Frigerio--Hartnick~\cite{BFH2014}.

\emph{Throughout this appendix, we let $L$ be a countable group.}

\medskip\noindent {\bf Constructing a common blow-up}\\
The implication of the following lemma is that the $\succcurlyeq$--relation on $\bR$ induced by a (possibly discontinuous) semi-conjugating map can be ``generated'' by continuous semi-conjugating maps.

\begin{lem}\label{lem:real-me}
Let $L$ be a countable group.
\be
\item
Suppose we have two actions \[\rho_0,\rho_1\in\Hom(L,\Homeo_+(\bR)).\]
Then $\rho_0\succcurlyeq_h\rho_1$
for some proper nondecreasing map $h$
if and only if there exists an action $\rho\in \Hom(L,\Homeo_+(\bR))$
and surjective nondecreasing maps 
\[
h_1,h_2\co \bR\to\bR\]
such that $\rho\succcurlyeq_{h_i}\rho_i$ for $i=1,2$.
\item
Suppose we have two actions \[\rho_0,\rho_1\in\Hom(L,\Homeo_\bZ(\bR)).\]
Then $\rho_0\succcurlyeq_h\rho_1$
for some monotone degree one map $h$
if and only if there exists an action  $\rho\in \Hom(L,\Homeo_\bZ(\bR))$ and surjective monotone degree one maps 
\[
h_1,h_2\co \bR\to\bR\]
such that $\rho\succcurlyeq_{h_i}\rho_i$ for $i=1,2$.
\ee
\end{lem}


\bp 
Let us prove part (2) of the lemma, since the proof of part (1) is very similar and simpler.
We first consider the forward direction.
For each $x\in \bR$, let
\[
h(x-)=\lim_{t\to x-0}h(t)\quad\text{and}\quad
h(x+)=\lim_{t\to x+0}h(t).\]
Let us consider the strip \[X = \bR\times I \sse\bR^2.\]
For two points $p,q\in X$, we denote by $[p,q]$ the segment joining $p$ and $q$.
Define a singular foliation of $X$:
\[
\FF
=\left\{[(x,0),(y,1)] \co x\in \bR\text{ and }y\in[h(x-),h(x+)]
\right\}.\]
Note that the foliation is $\bZ \times \{0\}$ periodic, so that it is determined by what happens in
the square $[0, 1] \times I$.
Each point $(x,1/2)\in X$ belongs to a unique leaf, say $\LL(x)\in\FF$.
So there exists maps $h_i\co\bR\to\bR$ uniquely determined by the condition:
\[\LL(x)=[(h_0(x),0),(h_1(x),1)].\]
The map $h_i$ is monotone, since leaves do not intersect in the interior of $X$.
 Also \[h\circ T= T\circ h\] implies that \[h_i\circ T= T\circ h_i\] for each $i=0,1$.
Whenever $\{x_n\}$ converges to $x$, the sequence of leaves $\{\LL(x_n)\}$ converges to $\LL(x)$ in the Hausdorff distance. So $h_i$ is continuous. 

For each $g\in L$ and $x\in\bR$,  we define $\rho(g)(x)\in\bR$ by the formula:
\[
\LL(\rho(g)(x)) = \left[(\rho_0(g)\circ h_0(x),0),(\rho_1(g)\circ h_1(x),1)\right].\]
Indeed, the right-hand side is a leaf of the foliation since
 \[h_1(x)\in [h(h_0(x)-),h(h_0(x)+)].\]
 
Then it is routine to check that
 $\rho(g)\in\Homeo_\bZ(\bR)$ for each $g\in L$ and $\rho\co L\to\Homeo_\bZ(\bR)$ is a group homomorphism.
Since \[\left[(\rho_0(g)\circ h_0(x),0),(\rho_1(g)\circ h_1(x),1)\right]\]
is the unique leaf containing $(\rho(g)(x),1/2)$, the relation
$\rho\succcurlyeq_{h_i}\rho_i
$ follows.
For the backward direction, we simply use the map
\[
h(x) = h_2\circ \sup(h_1^{-1}(x)).\qedhere\]
\ep

\bp[Alternative proof of Lemma~\ref{lem:real-me} (2)]
One can compute Cartesian coordinate description of the above maps as follows. Define two strictly increasing maps \[f^-(y)=\frac{y+h(y-)}2,\quad f^+(y)=\frac{y+h(y+)}2.\]
Then $f^-$ and $f^+$ are left- and right-continuous, respectively.
Moreover, $f^-=f^+$ except for the (countably many) jump discontinuities of $h$. 
So there uniquely exists a continuous monotone map $h_0$ which is the ``inverse'' of $f^\pm$ in the following sense:
\[f^-\circ h_0(x)\le x\le f^+\circ h_0(x), \text{ for all }x.\]
Define \[h_1(x) = 2x - h_0(x).\] It is immediate that $h_0$ and $h_1$ are continuous and \[h_i\circ T=T\circ h_i.\] If $x<y$, then
\[h_1(x)=2x-h_0(x)\le h(h_0(x)+)\le h(h_0(y)-)\le 2y - h_0(y)=h_1(y).\]
So $h_1$ is also monotone.

For each $a\in L$, define 
\[
\rho(a)(x) = \frac{\rho_0(a)h_0(x)+\rho_1(a)h_1(x)}2.\]
 Therefore 
we can write \[\rho(a)(x)\in[p^-,p^+],\] where
\begin{align*}
p^\pm &= \frac{\rho_0(a)h_0(x)+\rho_1(a)\circ h(h_0(x)\pm)}2
\\
&= \frac{\rho_0(a)h_0(x)+h(\rho_0(a)\circ h_0(x)\pm)}2
=f^\pm(\rho_0(a)h_0(x)).\end{align*}
By the definition of $h_0$, we have \[h_0\rho(a)(x)=\rho_0(a)h_0(x).\]
Moreover, 
\[
h_1\rho(a)(x) = 2\rho(a)(x) - h_0\rho(a)(x)
= 2\rho(a)(x) - \rho_0 h_0(x) = \rho_1(a)h_1(x).\]

We see
\begin{align*}
\rho(a)(\rho(b)x)
&=
\frac{
\rho_0(a)h_0\rho(b)(x)+\rho_1(a)h_1\rho(b)(x)
}2
\\
&=
\frac{
\rho_0(a)\rho_0(b)h_0(x)+\rho_1(a)\rho_1(b)h_1(x)
}2
=\rho(ab)(x).
\end{align*}
Hence $\rho$ is the desired group action.
\ep

\medskip\noindent {\bf The canonical Euler cocycle}\\
The following lemma describes relations between the rotation (or translation) number
and Euler (or canonical Euler) 2-cocycles.
\begin{lem}\label{lem:eu-tau}
\be
\item
For each $g\in \Homeo_+(S^1)$, we have
\begin{align*}
\rot^\sim\circ s(g)&=\lim_{n\to\infty}\frac1n\sum_{k=1}^n \eu(g,g^k)\quad\textrm{in }\bR,\\
\rot(g)&=\lim_{n\to\infty}\frac1n\sum_{k=1}^n \eu(g,g^k)\mod \bZ.
\end{align*}
\item
As maps on $L\times L$, we have
\[
\tau=\eu-\partial(\rot^\sim\circ s).\]
\ee
\end{lem}

\bp
(1) 
For $g,h\in\Homeo_+(S^1)$, observe that $\eu(g,h)=1$ if and only if \[s(g)\circ s(h)(0)\ge 1.\]
Hence for each circle homeomorphism $g$, we have
\[
\rot^\sim\circ s(g)=
 \lim_n \frac1n \sum_{k=0}^{n-1}\left| (s(g)^k(0),s(g)^{k+1}(0)]\cap\bZ
 \right|
=
  \lim_n \frac1n \sum_{k=0}^{n-1} \eu(g,g^k).\]
The second statement obviously follows.

(2) Recall from Section \ref{sec:prelim} that for $f,g\in\Homeo_+(S^1)$ and their arbitrary lifts $\tilde f,\tilde g\in\Homeo_\bZ(\bR)$,
 the canonical Euler cocycle is given by
\[
\tau(f,g)=\rot^\sim(\tilde f\tilde g)-\rot^\sim(\tilde f)-\rot^\sim(\tilde g).\]

Thus, for $g,h\in\Homeo_+(S^1)$, we have
\begin{align*}
\tau(g,h) &=\rot^\sim(s(g)s(h))-\rot^\sim\circ s(g)-\rot^\sim\circ s(h)\\
&=\rot^\sim(s(gh))+\eu(g,h)-\rot^\sim\circ s(g)-\rot^\sim\circ s(h)\\
&=\eu(g,h)+\partial(\rot^\sim\circ s)(g,h).\qedhere\end{align*}
\ep

\bp[Proof of Theorem~\ref{thm:equiv}, (\ref{part:eu})$\Rightarrow$(\ref{part:matsumoto})]
There exists a bounded map $\beta\co L\to\bZ$ such that
\[
\rho_1^*\eu = \rho_0^*\eu+\partial\beta.\]
By part (1) of Lemma~\ref{lem:eu-tau}, for each $g\in L$ we have
\begin{align*}
\rot^\sim\circ s\circ \rho_1(g)
&=
\lim_n \frac1n \sum_{k=0}^{n-1} \rho_1^*\eu(g,g^k)\\
&=\lim_n \frac1n \sum_{k=0}^{n-1} \left(
\rho_0^*\eu(g,g^k) + \beta(g)+\beta(g^k)-\beta(g^{k+1})\right)\\
&=\rot^\sim\circ s\circ \rho_0(g)+\beta(g),
\end{align*}
which proves the first equality. Moreover, we see from part (2) of the same Lemma that
\[
\rho_1^*\tau=
\rho_1^*\eu-
\partial(\rot^\sim\circ s\circ\rho_1)
=
\rho_0^*\eu+\partial\beta
-\partial(\rot^\sim\circ s\circ\rho_0+\beta)
=\rho_0^*\tau.\qedhere\]
\ep

\medskip\noindent {\bf Properties of a blow-up}\\
Let us explain implications of having a common blow-up, closely following~\cite{Calegari2007}. Let $h\co S^1\to S^1$ be a monotone degree one map.
Following the notation from~\cite{Calegari2007}, we let $\mathrm{Gap}(h)$ denote the set of locally constant points of $h$, and  \[\mathrm{Core}(h)=S^1\setminus\mathrm{Gap}(h).\]

\begin{lem}[{\cite[Lemma 2.14]{Calegari2007}}]\label{l:core-gap}
If $h$ is a non-constant monotone degree one map on $S^1$, then the set
$\mathrm{Core}(h)$ is perfect and uncountable.
\end{lem}
\bp
This immediately follows from that $\mathrm{Gap}(h)$ is a countable disjoint union of open intervals such that no two intervals have a common endpoint.
\ep

\begin{lem}\label{lem:minimal}
Let $\rho_0$ and $\rho_1$ be circle actions of $L$.
\be
\item
If $\rho_0$ and $\rho_1$ have a common blow-up,
then we have that $\rho_0\succcurlyeq_h \rho_1$
for some monotone degree one map $h$ on $S^1$,
and that
$\rho_0^*\eu_b=\rho_1^*\eu_b$.
\item
If $\rho_0$ and $\rho_1$ are minimal, and if $\rho_0\succcurlyeq_h\rho_1$ for some monotone degree one map $h$ on $S^1$, then $\rho_0$ and $\rho_1$ are conjugate.
\ee
\end{lem}

\bp
(1)
There exists a circle action $\rho$ and surjective monotone degree one maps $h_0,h_1$ such that \[\rho\succcurlyeq_{h_i}\rho_i.\]
We define for each $x\in S^1$,
\[h(x) = h_1\circ\inf\circ h_0^{-1}(x).\]
The monotone degree one map $h$ is well-defined since 
\[\varnothing\ne h_0^{-1}(x)\ne S^1.\]

\begin{claim*}
For each $g\in L$ and $x\in S^1$, we have
\[\rho(g)\circ h_0^{-1}(x)=h_0^{-1}\circ\rho_0(g)(x).\]
\end{claim*}
Part $\sse$ can be seen from
\[
h_0\circ\rho(g)\circ h_0^{-1}(x)
=
\rho_0(g)\circ h_0\circ h_0^{-1}(x)=\rho_0(g)(x).\]
The opposite inclusion follows from
\[
h_0\circ\rho(g^{-1})\circ h_0^{-1}\circ\rho_0(g)(x)
=
\rho_0(g^{-1})\circ h_0\circ h_0^{-1}\circ\rho_0(g)(x)=x.\]

Now, we see $\rho_0\succcurlyeq_h\rho_1$ from
 the above claim
 and from
\[
\rho_1(g)\circ h(x)
=\rho_1(g)\circ h_1\circ\inf\circ h_0^{-1}(x)
=h_1\circ\rho(g)\circ\inf\circ h_0^{-1}(x)=h\circ\rho_0(g)(x).\]

For the second part, we may assume $\rho_0$ is a blow-up of $\rho_1$.
By Lemma~\ref{l:core-gap}, we may require that
\[0\in \mathrm{Core}(h),\]
and furthermore, that $0$ is not a boundary point of an interval in $\mathrm{Gap}(h)$.

Choose an arbitrary continuous monotone degree one map
$\tilde h\in\Homeo_\bZ(\bR)$ lifting $h$
and let $p=\tilde h(0)$.
Let us denote the sections
\[
s(f) =\tilde f,\quad
s^p(f) = \hat f.\]
so that \[\tilde f(0)\in[0,1),\quad \hat f(p)\in[p,p+1)\] for each $f\in\Homeo_+(S^1)$.
Then the following diagram commutes modulo $\bZ$:
\[
\xymatrix{
\bR \ar[r]^{\widetilde{\rho_1(a)}} \ar[d]^{\tilde h}& \bR \ar[d]^{\tilde h} \\
\bR \ar[r]^{\widehat{\rho_1(a)}} & \bR.}
\]

We claim that
the above diagram commutes on the nose.
Evaluating at $0$, we see 
\[
\widehat{\rho_1(a)}\tilde h(0)
=
\widehat{\rho_1(a)}p
\in[p,p+1),\quad\text{and that}\quad
p=\tilde h(0)\le \tilde h \widetilde{\rho_0(a)}(0)\le \tilde h(1)=p+1.\]
Since $h$ is not locally constant on either side of $0$, we see \[\tilde h \widetilde{\rho_0(a)}(0)\ne \tilde h(1).\] This establishes the commutativity of the diagram.

For $a,b\in  G$, we have that
\begin{align*}
\rho_0^* \eu(a,b)=1
&\Leftrightarrow \widetilde{\rho_0(a)}\widetilde{\rho_0(b)}0\ge1\\
&\Leftrightarrow \tilde h\widetilde{\rho_0(a)}\widetilde{\rho_0(b)}0\ge\tilde h1=p+1\\
&\Leftrightarrow \widehat{\rho_1(a)}\widehat{\rho_1(b)} p\ge p+1\\
&\Leftrightarrow\rho_1^* \eu^p(a,b)=1.
\end{align*}
Note we used again the condition that $0$ is not a locally constant point of $h$.
It follows that 
\[\rho_0^*\eu=\rho_1^* \eu^p\] as cocycles.
By the basepoint independence of the bounded Euler class, we have \[\rho_0^* \eu_b=\rho_1^* \eu_b\] in $H^2_b( L;\bZ)$.

(2)
Let $\rho_0\succcurlyeq_h\rho_1$ for some monotone degree one map $h$ on $S^1$.
Note that \[\rho_1(G)\left(h(S^1)\right)=h\left(\rho_0(G)(S^1)\right)=h(S^1).\]
So $h(S^1)$ is $\rho_1(G)$--invariant. By minimality, we see $h(S^1)$ is dense in $S^1$.
It follows that $h$ does not have jump discontinuities, and so surjective and  continuous.
For each open nonempty interval $J$ in $\mathrm{Gap}(h)$, we see
\[h\circ\rho_0(g)\restriction_J=\rho_1(g)\circ h\restriction_J\]
is constant. So $\mathrm{Gap}(h)$ is $\rho_0(G)$--invariant.
Since \[\mathrm{Core}(h)=S^1\setminus \mathrm{Gap}(h)\] is a closed, uncountable (Lemma~\ref{l:core-gap}) and $\rho_0(G)$--invariant, we see that $\mathrm{Core}(h)=S^1$. So $h$ is injective.
This shows that $h$ is a homeomorphism.
\ep

\noindent {\bf Finite orbits}\\
Note that a circle action $\rho$ of $L$ has a finite orbit if and only if a finite-index subgroup of $L$ has a global fixed point.
Recall the \emph{exponent}\index{exponent (of a group element)} of $g$ in a group is the smallest positive integer $N$ such that $g^N=1$.

\begin{lem}\label{lem:exponent}
Let $\rho$ be a circle action of $L$.
\be
\item
Then $\rho$ has a finite orbit of cardinality $N$ 
if and only if the exponent of $\rho^*\eu_b$ is $N$ in $H^2_b(L;\bZ)$. 
\item
If $\rho$ has a finite orbit,
then the map $\bar\rho=T\circ\rot\circ\rho$ is a group action of $L$ and satisfies $\bar\rho^*\eu_b=\rho^*\eu_b$.
\ee
\end{lem}
\bp
(1) ($\Rightarrow$) Suppose 
\[\rho(G)p=\{p_0=p,p_1,\ldots,p_{N-1}\}\] in the cyclic order of $S^1$. Then we have a surjective group homomorphism $\beta\co G\to\bZ/N\bZ$ such that $\rho(g)p_i=p_{i+\beta(g)}$.
We may assume $p=0$.
Fix a lift \[\tilde\beta\co G\to \{0,1,\ldots,N-1\}.\]
Then 
\[
\rho^*\eu(a,b)=\frac1N\left(\tilde\beta(a)+\tilde\beta(b)-\tilde\beta(ab)\right)=\frac1N\partial\tilde\beta(a,b).\]
So we have $N\rho^*\eu_b=0$. 

Let $M$ be the exponent of $\rho^*\eu_b=0$. 
We have $M\rho^*\eu=\partial\tilde\gamma$ for some bounded map $\tilde\gamma\co G\to\bZ$. 
We can write $N=Mk$ for some $k>0$ by the minimality of $M$.
We see
\[
\partial\left(\tilde\beta-k\tilde\gamma\right)
=N\left(\frac1N\partial\tilde\beta-\frac1M\partial\tilde\gamma
\right)=0.
\]
So $\tilde\beta-k\tilde\gamma\co G\to\bZ$ is a homomorphism,
which is bounded. It follows that $\tilde\beta=k\tilde\gamma$.
Since $\tilde\beta$ is surjective, we have $k=1$ and $M=N$.

(1) ($\Leftarrow$)
Suppose $N$ is the exponent of $\rho^*\eu$. Note that we allow $N=1$. Then \[N\rho^*\eu=\partial\tilde\beta\] for some 
bounded map \[\tilde\beta\co G\to \bZ.\]
We have a group homomorphism 
\[\beta\co G\to\bZ/N\bZ\] from the post-composition
$\bZ\to\bZ/N\bZ$. The minimality of $N$ implies that $\beta$ is surjective. 

Set $H=\ker\beta$ and $\sigma=\rho\restriction_H$. 
Note that \[\tilde\beta(a)/N\in\bZ\] for each $a\in H$.
By writing
\[
\sigma^*\eu(a,b)=
{\tilde\beta(a)}/N
+\tilde\beta(b)/N-\tilde\beta(ab)/N
\]
we see that \[\sigma^*\eu_b=0\in H^2_b(H;\bZ).\]
Write the cocycle
$\sigma^*\eu=\partial\gamma$ for some bounded map $\gamma$ on $H$. We can lift $\sigma$ to 
 $\tilde\sigma\co H\to\Homeo_\bZ(\bR)$ by the formula
\[\tilde\sigma(g)=T(-\gamma(g)) s\circ\sigma(g);\]
see Section~\ref{sec:prelim}.
Then 
\[
\tilde\sigma(g)(0)=s\circ\sigma(g)(0)-\gamma(0)
\in [-\|\gamma\|_\infty,\|\gamma\|_\infty+1).\]
So we can find \[p=\sup\{\tilde\sigma(g)(0)\co g\in G\}\in\Fix\tilde\sigma.\] The projection of $p$ is a global fixed point of $\sigma$.
Moreover,
\[
\left|\rho(G)p\right|
=
\left|
G/H
\right|=N.\]

(2) 
Let $\rho(G)p$ be a finite orbit.
Using the notations from the proof of (1), we observe
\[\rot\circ\rho(g)=\beta(g)/N.\]
So $\bar\rho$ is a group homomorphism. 
While proving the direction ($\Rightarrow$) in (1),
we have already seen 
that
\[
\rho^*\eu^p(a,b)=\frac1N\left(\tilde\beta(a)+\tilde\beta(b)-\tilde\beta(ab)\right)=
\bar\rho^*\eu^p(a,b)\]
 as cocycles.
\ep

\begin{lem}\label{lem:fin-orb}
Let $\rho_0,\rho_1$ be circle actions of $L$.
Suppose either 
\be[(i)]
\item
$\rho_0$ and $\rho_1$ have the same bounded Euler class, or
\item
$\rho_0$ and $\rho_1$ have a common minimalization.
\ee
If $\rho_0$ has a finite orbit, then so does $\rho_1$.
\end{lem}

\bp
In the case (i),  Lemma~\ref{lem:exponent} (1) implies the conclusion. In the case (ii), 
recall that a minimalization of $\rho_i$ is minimal if and only if $\rho$ does not have a finite orbit.
\ep

\bp[Proof of Lemma~\ref{lem:minimal-exists}]
Suppose $\rho$ has an exceptional minimal set $C$;
in particular, $C$ is a Cantor set.
In this case, we let $h\in S^1\to S^1$ be the \emph{Cantor function}  which is continuous and is locally constant on $S^1\setminus C$. Then we have $\rho\succcurlyeq_h \bar \rho$ for some minimal action $\bar \rho$. 
Uniqueness follows from 
Lemma~\ref{lem:minimal} implies that such $\bar\rho$ is unique up to conjugacy.

If $\rho$ has a finite orbit, then Lemma~\ref{lem:exponent} implies that the minimalization $\bar\rho$ is a well-defined circle action.\ep

For completeness,
let us record a result that corresponds to Lemma~\ref{lem:fin-orb}
in the case of $\Homeo_+(\bR)$.

\begin{lem}\label{l:semi}
Let $L$ be a countable group.
If
$\rho_0,\rho_1\in \Hom(L,\Homeo_+(\bR))$
 are semi-conjugate
 and if $\rho_0$ has no global fixed point, then neither does $\rho_1$.
\end{lem}

\bp
By Lemma~\ref{lem:real-me} (1), we may assume $\rho_0\succcurlyeq_h\rho_1$ for some surjective $h$. 
Let us assume the contrary and pick $y\in\Fix\rho_1$.
Then $J=h^{-1}(y)$ is a $\rho_0(G)$--invariant proper connected subset of $\bR$.
Since $\partial J$ consists of global fixed points of $\rho_0$, it follows that $J$ is empty. This contradicts the surjectivity of $h$.
\ep

\bp[Proof of Theorem~\ref{thm:equiv}]
The implication (\ref{part:r-semi})$\Rightarrow$(\ref{part:monotone})
 is immediate by Lemma~\ref{lem:real-me}.
The proof of
(\ref{part:eu})$\Rightarrow$(\ref{part:matsumoto})
was given right after Lemma~\ref{lem:eu-tau}.
If (\ref{part:r-semi}) holds, then Lemma~\ref{lem:equiv} implies that (\ref{part:symmetric}) also holds.

For (\ref{part:monotone})$\Rightarrow$(\ref{part:min}),
let $\rho$ be a common blow-up of $\rho_0$ and $\rho_1$,
and let $\bar\rho_i$ be a minimalization of $\rho_i$.
We already saw that $\rho_0^*\eu_b=\rho_1^*\eu_b$
in Lemma~\ref{lem:minimal}.  By Lemma~\ref{lem:fin-orb}, we have two cases.
If $\rho_0$ and $\rho_1$ have finite orbits, then  
\[\bar\rho_i=T\circ\rot\circ\rho_i.\]
Lemma~\ref{lem:eu-rot} implies that $\bar\rho_0=\bar\rho_1$.
So we may assume neither have finite orbits
and $\bar\rho_i$ is minimal.
By Lemma~\ref{lem:minimal}, we see $\bar\rho_0\succcurlyeq_h\bar\rho_1$
for some monotone degree one map $h$,
and moreover, $\bar\rho_0$ and $\bar\rho_1$ are conjugate.

For (\ref{part:min})$\Rightarrow$(\ref{part:eu}),
suppose $\rho_0$ and $\rho_1$ have a common minimalization $\bar\rho$.
By Lemma~\ref{lem:fin-orb}, either both have finite orbits
or $\bar\rho$ is minimal. By applying Lemma~\ref{lem:exponent} (2)
in the first case and 
by applying Lemma~\ref{lem:minimal} (1) in the second,
we see
\[\rho_0^*\eu_b=\bar\rho^*\eu_b=\rho_1^*\eu_b.\qedhere\]

For (\ref{part:monotone})$\Rightarrow$(\ref{part:mann-path}),
we may assume $\rho_0\succcurlyeq_h\rho_1$ for some surjective monotone degree one map $h$. As in~\cite[Proposition 7.6]{Mann2015IM}, we pick a path \[\{h_t\co t\in[0,1)\}\sse \Homeo_+(S^1)\]
such that $h_0=\Id$ and $\lim_{t\to 1}h_t = h$.
Then $\{\rho_t(g)=h_t\rho_0(g) h_t^{-1}\}$ is the desired path.

For (\ref{part:mann-path})$\Rightarrow$ (\ref{part:matsumoto}),
we note that the following is integer--valued:
\[\rot^\sim\circ s\circ\rho_t(g)-\rot^\sim\circ s\circ\rho_0(g)\]
By Lemma~\ref{lem:eu-tau}, we see
\[
\rho_t^*\tau-\rho_0^*\tau
=\rho_t^*\eu-\rho_0^*\eu-
(s\circ \rho_t)^*\partial\rot^\sim+(s\circ \rho_0)^*\partial\rot^\sim\]
is an integer--valued function on $L\times L\times  I $.
Since $\tau$ is continuous on $\Homeo_+(S^1)\times\Homeo_+(S^1)$,
we see that $\rho_t^*\tau=\rho_0^*\tau$, as desired.

The condition (\ref{part:matsumoto}) implies (\ref{part:r-semi}),
as elementarily (without using cohomology) proven in~\cite[Theorem 5.11]{Mann-hb}.
The implication (\ref{part:symmetric})$\Rightarrow$(\ref{part:eu}) is included in \cite[Theorem 1.4]{BFH2014}.
\ep


\section*{Acknowledgments}
The authors are grateful to two anonymous referees for their careful reading and comments. The authors thank M. Bestvina, B. Bowditch, D. Calegari, M. Casals-Ruiz, K. Fujiwara, T. Ito, J. Kahn, M. Kapovich, I. Kim, Y. Matsuda, K. Ohshika, A. Reid, M. Sakuma, Z. Sela, A. Sisto, M. Triestino, T. N. Venkataramana and H. Wilton for helpful discussions. The authors thank M. Wolff for sharing~\cite{WolffPreprint} and for several helpful comments. The authors are particularly grateful to K. Mann for a number of insightful comments and for pointing out several references, particularly~\cite{DK2006Duke}. The second and third authors thank the Tata Institute for Fundamental Research in Mumbai, and the first and second authors thank the Institute for Mathematical Sciences in Singapore, where parts of this research were conducted. The authors were supported by the National Science Foundation 
under Grant No. DMS-1440140 at  the Mathematical Sciences Research Institute in Berkeley
during Fall 2016, where this research was completed. The first author is supported by Samsung Science and Technology Foundation (SSTF-BA1301-06).
The second author is partially supported by Simons Foundation Collaboration Grant number 429836, by an Alfred P. Sloan Foundation Research Fellowship, and by NSF Grant DMS-1711488.
The third author is supported in part by a Department of Science and Technology  J.C. Bose Fellowship grant.


\bibliographystyle{amsplain}
\bibliography{ref}
\printindex

\end{document}